%% file: Alternating_surgeries.tex
\documentclass{amsart} 

\usepackage{amsmath,amssymb,amsthm} 
\usepackage{graphicx} 
\usepackage{subfig}
\usepackage[arrow, matrix, curve]{xy} 
\usepackage{color} 
\usepackage{makecell} 
\usepackage{multirow} 
\usepackage{hyperref} 
\usepackage{enumerate}
\usepackage{bm}
\usepackage{booktabs} 
\usepackage{siunitx} 

\usepackage{longtable}
\usepackage{multicol}
\usepackage{supertabular}

\usepackage{enumitem} 

\usepackage{tikz}
\usetikzlibrary{shapes.geometric, arrows}
\tikzstyle{unexpected} = [rectangle, rounded corners, minimum width=3cm, minimum height=1cm,text centered, draw=black, fill=red!30]
\tikzstyle{io} = [rectangle, rounded corners, minimum width=3cm, minimum height=1cm, text centered, draw=black, fill=blue!30]
\tikzstyle{process} = [rectangle, rounded corners, minimum width=3cm, minimum height=1cm, text centered, draw=black, fill=orange!30]
\tikzstyle{decision} = [rectangle, rounded corners, minimum width=3cm, minimum height=1cm, text centered, draw=black, fill=green!30]
\tikzstyle{arrow} = [thick,->,>=stealth]

\DeclareMathOperator{\sysmin}{sysmin}
\DeclareMathOperator{\sys}{sys}
\DeclareMathOperator{\len}{len}

\newcommand{\calA}{\mathcal{A}}
\newcommand{\D}{\mathcal{D}}

\newcommand{\Z}{\mathbb{Z}}
\newcommand{\Q}{\mathbb{Q}}

\newcommand{\Salt}{\mathcal{S}_{\rm alt}}

\newcommand{\spin}{\ifmmode{\rm Spin}\else{${\rm spin}$\ }\fi}
\newcommand{\spinc}{\ifmmode{{\rm Spin}^c}\else{${\rm spin}^c$}\fi}
\newcommand{\spinct}{\mathfrak t}
\newcommand{\spincs}{\mathfrak s}

\newcommand{\nbhd}{\mathcal{N}}

\newcommand{\e}{\bm{e}}
\newcommand{\bv}{\bm{v}}
\newcommand{\bu}{\bm{u}}
\newcommand{\bsigma}{\bm{\sigma}}

\newcommand{\norm}[1]{\left\lVert #1 \right\rVert^2}

\makeatletter
\newcommand{\Vast}{\bBigg@{2.5}} 
\makeatother

\makeatletter
\newtheorem*{rep@theorem}{\rep@title}
\newcommand{\newreptheorem}[2]{%
\newenvironment{rep#1}[1]{%
 \def\rep@title{#2 \ref{##1}}%
 \begin{rep@theorem}}%
 {\end{rep@theorem}}}
\makeatother

\newtheoremstyle{thm}{}{}{\itshape}{}{\bfseries}{}{ }{} 
\newtheoremstyle{definition}{}{}{}{}{\bfseries}{}{ }{} 

\theoremstyle{thm}
\newtheorem{Theorem}{Theorem}[section]
\newtheorem{thm}[Theorem]{Theorem}
\newreptheorem{thm}{Theorem}  
\newtheorem{lem}[Theorem]{Lemma}
\newreptheorem{lem}{Lemma}  
\newtheorem{prop}[Theorem]{Proposition}
\newtheorem{cor}[Theorem]{Corollary}
\newtheorem*{Theorem-ohne}{Theorem}
\newtheorem{con}[Theorem]{Conjecture}

\newtheorem{ques}[Theorem]{Question}

\newtheorem{alg}[Theorem]{Algorithm}
\newtheorem{strat}[Theorem]{Strategy}

\theoremstyle{definition}
\newtheorem{defi}[Theorem]{Definition}
\newtheorem{rem}[Theorem]{Remark}
\newtheorem{ex}[Theorem]{Example}

\newtheorem{case}{Case}
\newtheorem*{claim}{Claim}

\newcommand{\ra}[1]{\renewcommand{\arraystretch}{#1}} 

\begingroup\expandafter\expandafter\expandafter\endgroup
\expandafter\ifx\csname pdfsuppresswarningpagegroup\endcsname\relax
\else
\fi

\definecolor{amaranth}{rgb}{0.9, 0.17, 0.31} 
\definecolor{carrotorange}{rgb}{0.93, 0.57, 0.13} 
\definecolor{citrine}{rgb}{0.89, 0.82, 0.04} 
\definecolor{dartmouthgreen}{rgb}{0.05, 0.5, 0.06} 
\definecolor{ballblue}{rgb}{0.13, 0.67, 0.8} 
\definecolor{ceruleanblue}{rgb}{0.16, 0.32, 0.75} 
\definecolor{amethyst}{rgb}{0.6, 0.4, 0.8} 
\definecolor{amber}{rgb}{1.0, 0.75, 0.0} 
\definecolor{burlywood}{rgb}{0.87, 0.72, 0.53} 

\numberwithin{equation}{section}

\begin{document}


\title{The search for alternating surgeries} 

\author{Kenneth L. Baker}
\address{Department of Mathematics, University of Miami, Coral Gables, FL 33146, USA}
\email{k.baker@math.miami.edu}

\author{Marc Kegel}
\address{Universidad de Sevilla, Dpto.\ de Álgebra,
Avda.\ Reina Mercedes s/n,
41012 Sevilla}
\email{kegelmarc87@gmail.com}

\author{Duncan McCoy}
\address{D\'{e}partment de Math\'{e}matiques, Universit\'{e} du Qu\'{e}bec \`{a} Montr\'{e}al, Canada}
\email{mc\_coy.duncan@uqam.ca}


\date{\today} 

\begin{abstract}
Surgery on a knot in $S^3$ is said to be an alternating surgery if it yields the double branched cover of an alternating link. The main theoretical contribution is to show that the set of alternating surgery slopes is algorithmically computable and to establish several structural results. Furthermore, we calculate the set of alternating surgery slopes for many examples of knots, including all hyperbolic knots in the SnapPy census. These examples exhibit several interesting phenomena including strongly invertible knots with a unique alternating surgery and asymmetric knots with two alternating surgery slopes.
We also establish upper bounds on the set of alternating surgeries, showing that an alternating surgery slope on a hyperbolic knot satisfies $|p/q|\leq 3g(K)+4$. Notably, this bound applies to lens space surgeries, thereby strengthening the known genus bounds from the conjecture of Goda and Teragaito.
\end{abstract}

\keywords{Dehn surgery, alternating surgeries, L-space knots, SnapPy census knots, changemaker lattices}

\makeatletter
\@namedef{subjclassname@2020}{%
  \textup{2020} Mathematics Subject Classification}
\makeatother

\subjclass[2020]{57R65, 57M12, 57K30  } 

\maketitle
\setcounter{tocdepth}{1}
\tableofcontents

\section{Introduction}
Given a knot $K$ in $S^3$, we say that $p/q\in \Q$ is an \textbf{alternating surgery slope} for $K$ if $p/q$-surgery on $K$ yields the double branched cover of an alternating knot or link. Since the double branched cover of a non-split alternating link is a Heegaard Floer homology L-space, knots with alternating surgeries provide examples of L-space knots. L-space knots and their properties have generated significant interest for a number of years, see for example~\cite{OS05,OS05b,Gr14,Greene2015genusbounds,GreeneLRP}. Since alternating surgeries frequently arise on some of the simplest L-space knots, such as the torus knots, and more generally the Berge knots, it is natural to wonder to what extent alternating surgeries arise on generic L-space knots.

The set of all alternating surgery slopes on $K$ will be denoted $\Salt(K)$. This paper is devoted to understanding $\Salt(K)$ both from a theoretical and a concrete computational perspective. We show that there exists an algorithm to calculate $\Salt(K)$ and develop methods to calculate $\Salt(K)$ in practice. In particular, we calculate $\Salt(K)$ for all hyperbolic knots in the SnapPy census.
\newcommand{\thmcomputable}{There exists an algorithm that, given a knot $K$ in $S^3$, returns $\Salt(K)$.}
\begin{thm}\label{thm:computable}
\thmcomputable
\end{thm}

There is a natural family of knots $\D$, associated to unknotting crossings in alternating diagrams of knots (see \S\ref{sec:D_definition}), which plays a crucial role in the classification of alternating slopes. In particular, $\D$ contains all knots for which $\Salt(K)$ is infinite. The theoretical computability of $\Salt(K)$ rests on the following structure theorem for $\Salt(K)$. If $K$ is a non-trivial knot admitting alternating surgeries, then there is an associated non-empty tuple of integers, $\underline{\rho}(K)=(\rho_1, \dots, \rho_\ell)$ ordered so that $\rho_1\leq \dots \leq \rho_\ell$ which we call the stable coefficients of $K$ (see \S\ref{sec:changemaker}).

\newcommand{\salttrichotomy}{Let $K$ be a non-trivial knot admitting positive alternating surgeries with stable coefficients $\underline{\rho}(K)=(\rho_1, \dots, \rho_\ell)$. If $N$ is the integer
\begin{equation*}
    N=\sum_{i=1}^\ell \rho_i^2 + \max_{1\leq k \leq \ell} \left(\rho_k - \sum_{i=1}^{k-1} \rho_i \right),
\end{equation*}
then precisely one of the following holds.
\begin{enumerate}
    \item\label{it:KinDcable} If $K\in \D$ is a cable knot or a torus knot with reducible surgery slope $N$, then
    \[\Salt(K)= [N-1, N+1]\cap \Q.\]
    \item\label{it:KinDnoncable} If $K\in \D$ is not a cable knot or a torus knot with reducible surgery slope $N$, then
    \[\Salt(K)= [N-1, N]\cap \Q.\] 
    \item\label{it:KnotinDgeneral} If $K\notin \D$ and $\Delta_K(x)\neq \Delta_{K'}(x)$ for every $K'\in \D$, then 
    \[\Salt(K)\subseteq\{N-1,N\}.\]
    \item\label{it:KnotinDAlexpolyinD} If $K\notin \D$ and there exists a knot $K'\in \D$ with $\Delta_K(x)=\Delta_{K'}(x)$, then $\Salt(K)$ is a finite subset of $\Salt(K')$ and there exists a computable constant $Q(K)$ such that any $p/q\in \Salt(K)$ satisfies $|q|\leq Q(K)$.
\end{enumerate}}

\begin{thm}\label{thm:salt_trichotomy}
\salttrichotomy
\end{thm}
It is worth noting that in addition to calculating the set $\Salt(K)$, one can also determine an alternating branching set for any slope $r\in \Salt(K)$. If $K\not\in \D$, then $\Salt(K)$ is finite and this data is simply a finite set of diagrams. If $K\in\D$, then $\Salt(K)$ is infinite and the branching sets of the alternating surgeries are all obtained by an easily describable tangle replacement on a single alternating diagram.

\subsection{Calculating \texorpdfstring{$\Salt(K)$}{Salt(K)} for census knots}
A full implementation of the theoretical algorithm lying behind Theorem~\ref{thm:computable} is currently impractical, since it hinges on ingredients such as the decidability of the homeomorphism problem for compact 3-manifolds for which an algorithm has never been fully implemented. Practically speaking, however, we were able to determine $\Salt(K)$ in most examples that we tried. 

The SnapPy software package contains a census of all complements of hyperbolic knots in $S^3$ that can be triangulated with at most nine ideal tetrahedra \cite{CDGW}.  We refer to these knots as the \textbf{SnapPy Census Knots}. Since knots in $S^3$ are determined by their complements \cite{GL89}, we shall freely interchange a knot and its complement throughout this paper. 
We were able to calculate $\Salt(K)$ for every knot in the SnapPy census. The following theorem summarizes our findings.

\newcommand{\altsurgeryclassification}{
Among the 1267 SnapPy Census Knots, exactly 393 admit at least one alternating surgery. These knots are presented in Table~\ref{tab:knots_with_alternating_surgeries}.
\begin{itemize}
    \item There are 381 knots such that $K\in\D$ and $\Salt(K)=[N-1,N] \cap \Q$.
    \item There are 12 knots such that $K\notin \D$ and $\Delta_K(x)\neq \Delta_{K'}(x)$ for every $K'\in \D$. For these 12 knots, we have that $\Salt(K)=\{N-1\}$.
\end{itemize}
Here $N$ denotes the integer appearing in Theorem~\ref{thm:salt_trichotomy}.
}

\begin{thm}
\label{thm:altsurgclassification} 
\altsurgeryclassification
\end{thm}
For each of the knots admitting alternating surgeries, we also catalogued alternating branching sets for all integer and half-integer alternating surgeries. For the knots in $\D$ the branching set for the half-integer alternating surgery allows us to determine the branching set for all other alternating surgeries, since these branching sets are obtained by a rational tangle replacement on an unknotting crossing in this diagram. 

\subsection{Strongly invertible knots with a single alternating surgery}

The Snap\-Py Census knots admitting a single alternating surgery all turn out to be strongly-invertible, making these the first known examples of strongly-invertible knots admitting a single alternating surgery. Moreover each of these 12 examples arises from the following diagrammatic construction. For each such $K$ there is an alternating link $J$ with a flat unknotting band $b$ in an alternating diagram of $J$ that crosses only one edge so that $\Sigma_2(J-\nbhd(b))$, the double branched cover of the complement of the banding ball, is homeomorphic to the exterior of $K$. 
That is to say, $K$ is the double branched cover of a dealternation edge of an almost alternating unknot diagram, see for example~\cite[Figure 28]{BL17}.   

\begin{thm}\label{thm:bandings}
For each of the 12 knots in the census that admits a unique alternating surgery, there is an alternating diagram for the branching set of the unique alternating surgery which contains an arc such that
\begin{enumerate}
    \item banding along this arc yields an almost alternating diagram of the unknot
    \item the knot $K$ is obtained as the lift of the dual of this arc in the double branched cover of the unknot.
\end{enumerate}
\end{thm}

Figure~\ref{fig:12bandings} shows the 12 links $J$ along with arcs describing the bandings $b$; both the name of $J$ in the knot tables and the census name for $K$ are given.  Table~\ref{tab:uniqueinteger} gives the stable coefficients, the slope of the alternating surgery, the branching set $J$, and the geometry of the double branched cover of $J$. Among these 12 manifolds $S^3_{N-1}(K) = \Sigma_2(J)$, there are $3$ graph manifolds and the remaining $9$ are hyperbolic. 

 \begin{table}[htbp]
 	\caption{The $12$ knots in the SnapPy census with a single alternating surgery. For all such knots we have $\Salt(K)=\{N-1\}$. Moreover, in the SnapPy basis $N-1$ is the slope $(-1,1)$ for $t11556$ and the slope $(0,1)$ for all the other $11$ knots. 
 	}
	\label{tab:uniqueinteger}
	\ra{1.2}
	\begin{tabular}{@{}l|l|l|l|l|l@{}}
Knot & $\underline{\rho}(K)$ & $N$& $\Salt(K)$ &  $J$ & geometry of $\Sigma_2(J)$ \\ \hline
         t10188& (5, 4, 3, 2, 2)& 60 & 59& $K9a6$ &graph manifold\\
         t11556& (6, 4, 3, 2)&67& 66&  $L9a20$ & hyperbolic\\
         t12753 & (7, 5, 3, 3)&95& 94& $L10a85$& graph manifold\\
         o9\_32132& (7, 5, 3)&86& 85& $K10a45$& hyperbolic\\
         o9\_32588& (5, 5, 4, 3, 2, 2)&85& 84& $L10a106$& hyperbolic\\
         o9\_37754& (6, 6, 4, 3, 2)&103& 102& $L10a76$& hyperbolic \\
         o9\_39451& (7, 6, 3, 2, 2)& 104 & 103& $K10a99$& hyperbolic\\
         o9\_40179& (8, 7, 3, 2, 2)&132& 131& $K11a298$& graph manifold\\
         o9\_43001& (8, 5, 4, 2, 2)&115& 114& $L10a71$& hyperbolic\\
         o9\_43679& (7, 7, 5, 3, 3)&144& 143& $K11a227$& hyperbolic\\
         o9\_43953& (9, 4, 3, 3)&118& 117& $K11a304$ &hyperbolic\\
         o9\_44054& (9, 5, 3, 3)&127& 126& $L11a229$ &hyperbolic \\ 
	\end{tabular}
\end{table}

\begin{figure}
    \centering
    \begin{tabular}{ccc}
    \includegraphics[height=3.5cm]{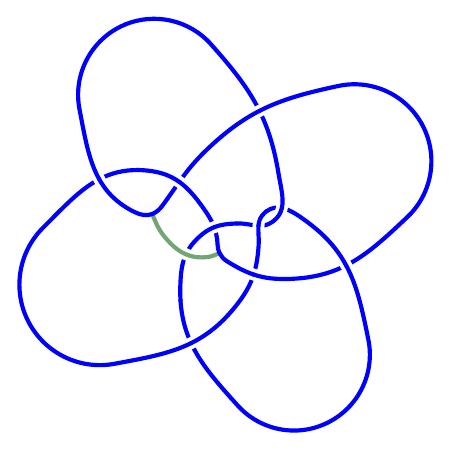} &
    \includegraphics[height=3.5cm]{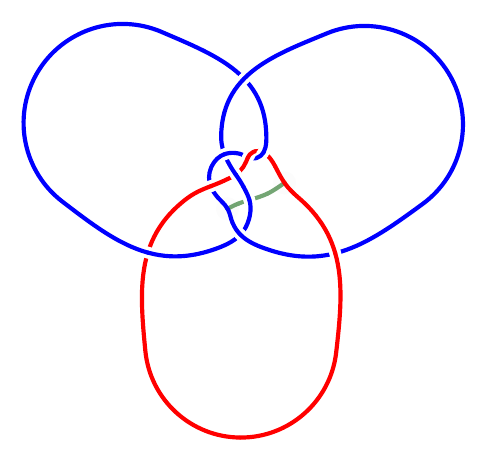} &
    \includegraphics[height=3.5cm]{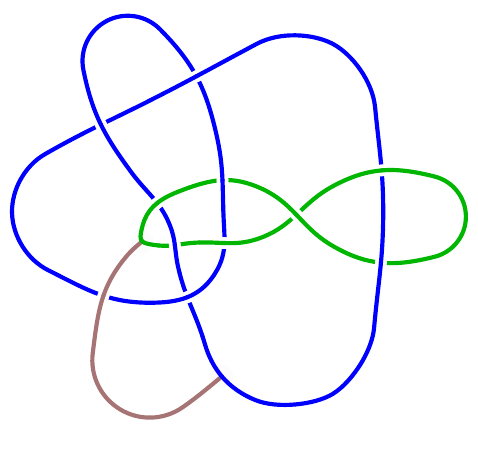}\\
    {t10188--K9a6} &
    {t11556--L9a20} &
    {t12753--L10a85}\\ \\. 
    \includegraphics[height=3.5cm]{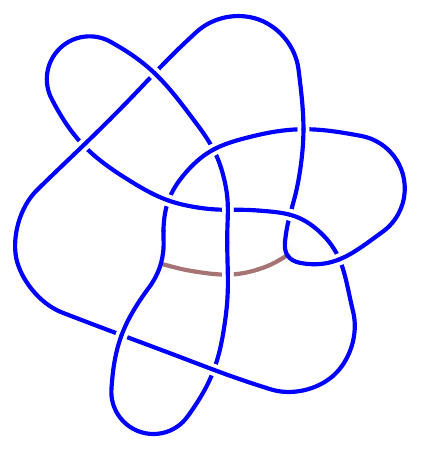} &
    \includegraphics[height=3.5cm]{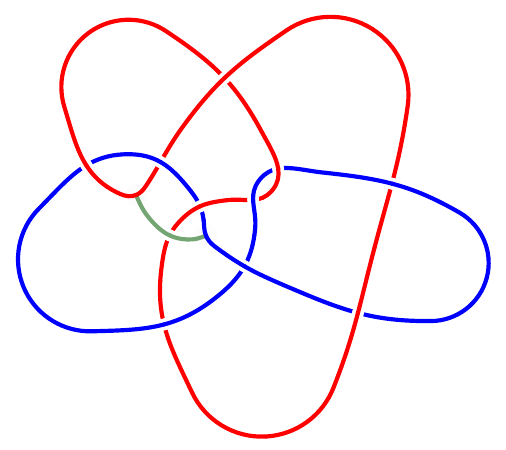} & \includegraphics[height=3.5cm]{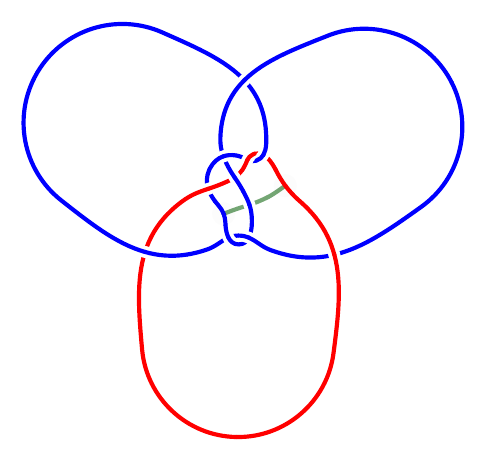} \\
    {o9\_32132--K10a45} &
    {o9\_32588--L10a106} & 
    {o9\_37754--L10a76} \\   \\
    \includegraphics[height=3.5cm]{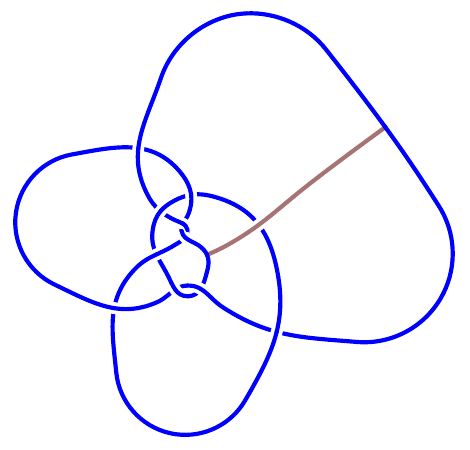} &
    \includegraphics[height=3.5cm]{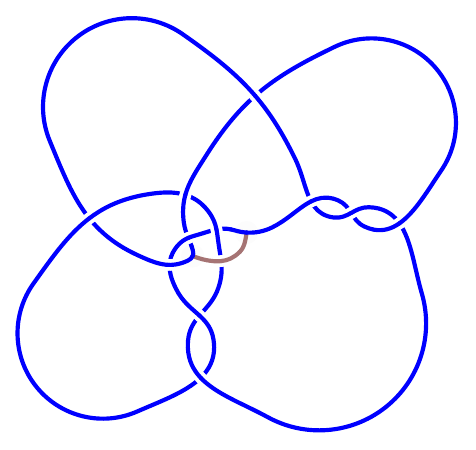} &
    \includegraphics[height=3.5cm]{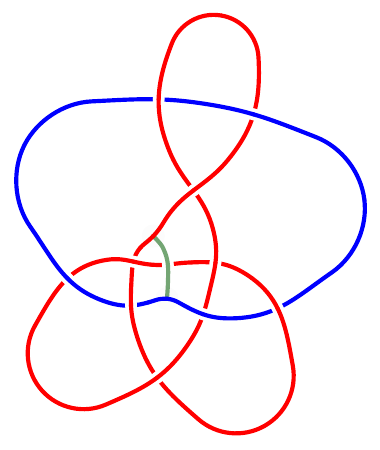}\\ 
    {o9\_39451--K10a99} & 
    {o9\_40179--K11a298} &
    {o9\_43001--L10a71}\\ \\
    \includegraphics[height=3.5cm]{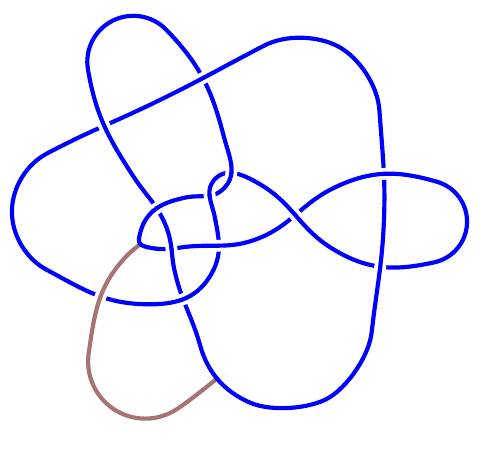}&
    \includegraphics[height=3.5cm]{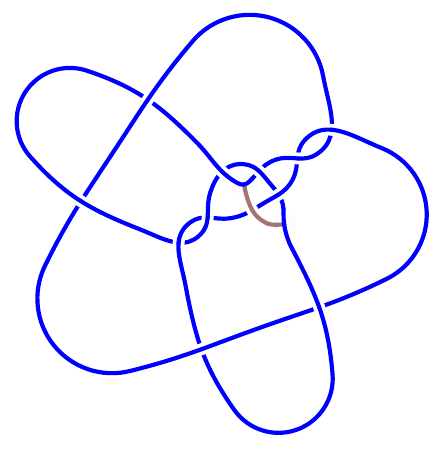} &
    \includegraphics[height=3.5cm]{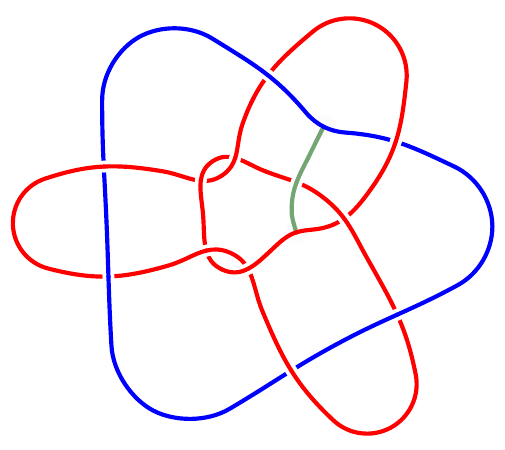} \\
    {o9\_43679--K11a277}&
    {o9\_43953--K11a304} &
    {o9\_44054--L11a229} \\
    \end{tabular}
    \caption{Alternating diagrams of twelve alternating links are shown.  Each has an arc along which a flat banding produces the unknot.  The double branched covers of the exteriors of these arcs are the 12 census knots with a single alternating surgery. Each diagram is labeled with the SnapPy manifold of the census knot and the Hoste-Thistlethwaite name of the alternating link.}
    \label{fig:12bandings}
\end{figure}

\subsection{Asymmetric L-space knots without alternating surgeries}
There are exactly nine asymmetric L-space knots amongst the SnapPy census:
\begin{align*}
\calA=\{t12533,t12681,o9\_38928,o9\_39162,o9\_40363,o9\_40487,\\o9\_40504,o9\_40582,o9\_42675\}.
\end{align*}
Theorem~\ref{thm:altsurgclassification} implies that none of the asymmetric L-space knots in the census admit an alternating surgery, thereby answering \cite[Question 10]{ABG+19}.
\begin{cor}\label{cor:asymetric}
None of the asymmetric L-space knots of the census have an alternating surgery.
\end{cor}
This provides the first known examples of asymmetric L-space knots with no alternating surgeries. The existence of such knots is of interest, since the first asymmetric L-space knots, the Baker-Luecke knots, all possess alternating surgeries \cite{BL17}. In fact, the Baker-Luecke knots were shown to be L-space knots by exhibiting an alternating surgery on each one. 

\subsection{L-space knots with exactly two alternating surgeries}
We also applied our methods to calculate $\Salt(K)$ for several Baker-Luecke knots. Surprisingly, all of these knots that we checked turned out to have a second alternating surgery in addition to the one originally constructed by Baker and Luecke. In fact these knots turned out to provide the first examples of knots for which $\Salt(K)$ takes the form $\Salt(K)=\{N-1, N\}$.

\begin{thm} 
For at least $14$ of the asymmetric hyperbolic L-space knots constructed in~\cite{BL17}, the set of alternating surgeries is of the form $\Salt(K)=\{N-1, N\}$. 
\end{thm}
Although we were able to find second alternating surgeries in these small examples, we have yet to find a theoretical explanation for the existence of these second alternating surgeries. It would be interesting to know if the phenomenon persists. 

\begin{ques}\label{ques:otherslopeBLknots}
Do all of the asymmetric L-space knots of Baker-Luecke have two alternating surgeries?
\end{ques}

\subsection{Conjectural structure of \texorpdfstring{$\Salt(K)$}{Salt(K)}}
Note that Theorem~\ref{thm:salt_trichotomy} gives a complete description of $\Salt(K)$ for knots in $\D$, so it remains only to understand $\Salt(K)$ for knots which are not in $\D$. Our results point towards the following conjecture.
\begin{con}\label{conj:atmosttwoslopes}
Let $K\notin \D$ be a knot admitting positive alternating surgeries, then
\[\Salt(K)\subseteq \{N-1,N\}.\]
In particular the only knots with non-integral alternating surgeries are in $\D$.
\end{con}

Theorem~\ref{thm:salt_trichotomy} shows that this conjecture holds for knots whose Alexander polynomials obstruct them from lying in $\D$. Thus it remains to verify the conjecture for knots $K\notin \D$ such that $\Delta_K=\Delta_{K'}$ for some $K'\in \D$. However, there are currently no known examples of such knots which admit alternating surgeries.

\begin{con}\label{conj:strong_conj}
Let $K\notin \D$ be a knot such that $\Delta_K(x)=\Delta_{K'}(x)$ for some $K'\in \D$. Then $\Salt(K)=\emptyset$.
\end{con}
Note that Conjecture~\ref{conj:strong_conj} implies Conjecture~\ref{conj:atmosttwoslopes}. 
\begin{rem}
    There are examples of L-space knots which have the Alexander polynomial of a knot in $\D$ but do not admit any alternating surgeries. The simplest example is the $(3,2)$-cable of $T_{3,2}$ which has the same Alexander polynomial as $T_{3,4}$. 
\end{rem}

It also seems natural to wonder which subsets of slopes allowed by Conjecture~\ref{conj:atmosttwoslopes} can actually occur. The knots in Theorem~\ref{thm:altsurgclassification} which admitted a single alternating surgery provide example of knots for which $\Salt(K)=\{N-1\}$. There are Baker-Luecke knots providing examples where $\Salt(K)=\{N-1,N\}$. Thus there is only one remaining possibility that has not been realized.
\begin{ques}\label{ques:N}
    Does there exist a knot $K$ such that $\Salt(K)=\{N\}$?
\end{ques}

Computations suggest that for many Baker-Luecke knots the slope $N$ is an alternating surgery slope. Thus, if any of these Baker-Luecke knots admits only a single alternating surgery, we would have an affirmative answer to Question~\ref{ques:N}.

\subsection{Genus bounds on alternating surgeries}
We obtain further restrictions on alternating surgery slopes. In \cite{Mc17} it was shown if $p/q$ is an alternating surgery slope for a non-trivial knot $K$, then $|p/q|\leq 4g(K)+3$ \cite[Theorem~1.1]{Mc17}. This bound is sharp with equality being attained by torus knots of the form $T_{2,2g+1}$. It turns out that these are the only such examples and that for all other knots a stronger bound holds.
\newcommand{\genusbound}{Let $K$ be a non-trivial knot which is not a two-stranded torus knot. Then for any alternating slope $p/q$ we have $|p/q|\leq 3g(K)+4$.}
\begin{thm}\label{thm:genusbound}
\genusbound
\end{thm}
This allows us to refine the upper bound on lens space surgeries established by Rasmussen \cite{Rasmussen2004Goda}.
\begin{cor}
If $K$ is a non-trivial knot which is not a two-stranded torus knot such that $S_p^3(K)$ is a lens space, then $p\leq 3g(K)+4.$
\end{cor}
The bound in Theorem~\ref{thm:genusbound} is sharp with equality being realised by the $T_{3,n}$ torus knots and a family of hyperbolic knots in $\D$ which begins with the $(-2,3,7)$-pretzel knot (see \S\ref{sec:genus_bounds}).

Using these ideas we also generalize a result of Ni \cite[Corollary~1.7]{Ni2020Seifert}.
\newcommand{\thmpolyform}{Let $K$ be a knot in $S^3$ which admits a positive alternating surgery. If the symmetrized Alexander polynomial of $K$ takes the form
    \[
    \Delta_K(x)=x^n-x^{n-1}+x^{n-2} + \text{lower order terms},
    \]
    then $K=T_{2,2n+1}$.}
\begin{thm}\label{thm:Alexander_poly_form}
    \thmpolyform
\end{thm}

\subsection{Knots with common surgeries}
It is a folklore theorem that for any two knots, $K$ and $K'$ in $S^3$, the set of slopes for which $S_{p/q}^3(K)$ and $S_{p/q}^3(K')$ are homeomorphic is finite. In order to prove Theorem~\ref{thm:computable} and Theorem~\ref{thm:salt_trichotomy} we need a version of this folklore result with effective bounds. Since the techniques required for such a result are peripheral to the rest of this paper we reserve its proof to the appendix. We note that the fact that one can obtain computable bounds depends on the quantitative bounds on hyperbolic Dehn fillings obtained by Futer-Purcell-Schleimer \cite{FPS19}.

\newcommand{\computableq}{Let $K$ and $K'$ be a pair of distinct knots in $S^3$. Then there is a computable integer $Q(K,K')$ such that if $S_{p/q}^3(K)\cong  S_{p/q}^3(K')$ for some $p/q\in \Q$, then $|q|\leq Q(K,K')$.} 

\begin{repthm}{thm:computable_q}
    \computableq
\end{repthm}

The constant $Q(K,K')$ appearing in this theorem is described in more detail in Definition~\ref{def:Q}. We remark also that Theorem~\ref{thm:computable_q} is implied by the results of Sorya and Wakelin \cite{Sorya2024effective}, which appeared shortly after the first version of this paper. Their work establishes the stronger result that for every knot $K$, there is a computable constant $Q(K)$ such that every slope $p/q$ with $|q|\geq Q(K)$ is a characterizing slope for $K$.

\subsection{Comparison of Theorem~\ref{thm:salt_trichotomy} with previous work}

The main structural result on $\Salt(K)$, Theorem~\ref{thm:salt_trichotomy}, is a siginificant improvement on several results appearing in \cite{Mc17}. In Theorem~1.2 of \cite{Mc17} it is shown that there exists an integer\footnote{The quantity in \cite{Mc17} is also denoted $N$, but we will reserve the notation $N$ to indicate the quantity appearing in Theorem~\ref{thm:salt_trichotomy} of the present article.} $M$ such that $\Salt(K)\subseteq [M-1,M+1]\cap \Q$. If $\underline{\rho}=(\rho_1, \dots, \rho_\ell)$ are the stable coefficients associated to a knot admitting positive alternating surgeries, then $M$ is defined by the formula
\[
M=\rho_1+\sum_{i=1}^\ell \rho_i^2.
\]
It is not hard to see that this satisfies $M\leq N$ and, since $\Salt(K)\subseteq [N-1,M+1]\cap \Q$, we also have $N\leq M+2$.

Although the formula for $M$ is simpler in comparison to the more intricate expression defining $N$, the quantity $N$ has superior theoretical properties. For example, for hyperbolic knots in $\D$ we always have $\Salt(K)=[N-1,N]\cap \Q$, whereas in terms of $M$ we only know that $\Salt(K)=[M-1,M]\cap \Q$ or that $\Salt(K)=[M,M+1]\cap \Q$.

There are examples where $N=M+1$. The manifold $m071$ is the complement of a knot $K$ in $\D$ whose stable coefficients are $(5, 2)$. From such stable coefficients, we obtain $N=25+4+3=32$, $M=31$ and alternating surgeries are given by $\Salt(K) =[31,32]\cap \Q$. For reference, the knot $K$ corresponds to an unknotting crossing in $K10a83$.

\subsection*{Code and data}
All the supporting code together with additional data can be found at~\cite{BKM}.

\subsection*{Acknowledgments}
We would like to thank Steve Boyer, David Futer, Brendan Owens, Patricia Sorya, and Claudius Zibrowius. We are also grateful to the anonymous referee for their detailed reading of the paper and their many constructive comments.

KLB was partially supported by the Simons Foundation grant \#523883 and gift \#962034.

MK was supported by the SFB/TRR 191 \textit{Symplectic Structures in Geometry, Algebra and Dynamics}, funded by the DFG (Projektnummer 281071066 - TRR 191) and is now supported by a Ram\'on y Cajal grant (RYC2023-043251-I) and by project PID2024-157173NB-I00 funded by MCIN/AEI/10.13039/501100011033, ESF+ and FEDER, EU; and by a VII Plan Propio de Investigación y Transferencia (SOL2025-36103) of the University of Sevilla.

DM is supported by an NSERC Discory grant (RGPIN-2020-05491) and a Canada Research Chair.

\part{Changemaker lattices and obtuse superbases}\label{part:lattices}
A key ingredient in the study of alternating surgeries is the theory of changemaker lattices. This part of the paper is dedicated to the lattice theory necessary for our work. We remind the reader that an integral lattice is a pair $(L,\cdot)$, where $L$ is a finitely generated free abelian group equipped with a symmetric, positive definite, bilinear pairing
\begin{align*}
    L\times L &\rightarrow \Z\\
    (x,y)&\mapsto x\cdot y.
\end{align*}
We will frequently omit the bilinear pairing from our notation and simply denote an integral lattice $(L,\cdot)$ by $L$. Given $x\in L$ we will use the notation $\norm{x}=x\cdot x$.

Throughout the paper we will refer to $\Z^r$ equipped with a symmetric bilinear pairing that admits an orthonormal basis as a diagonal lattice (of rank $r$) and we will denote the diagonal lattice by $\Z^r$. 

\subsection*{Notational convention} Given a fixed orthonormal basis $\{\e_1, \dots, \e_r\}$ for $\Z^r$, we will write vectors in the form
\[
\bv=v_1 \e_1 + \dots +v_r\e_r\in \Z^r.
\]
That is, a bold script symbol indicates an element of $\Z^r$ and the corresponding non-bold symbol with a subscript a coefficient with respect to the fixed basis. 

\section{Obtuse superbases and graph lattices}\label{sec:OSB}
The notion of an obtuse superbase will play a key role in understanding alternating surgeries.
\begin{defi}[Obtuse superbase]
Let $L$ be an integral lattice of rank $r$. We say that a set $B=\{v_0, \dots, v_r\}\subseteq L$ is an \textbf{obtuse superbase} for $L$ if it satisfies the following conditions:
\begin{enumerate}[label = (\alph*)]
    \item the $v_i$ span $L$;
    \item for $0\leq i <j \leq r$ we have $v_i \cdot v_j\leq 0$; and
    \item $v_0+ \dots + v_r=0$.
\end{enumerate}
The graph associated to $B$, denoted $G_B$, has vertex set $B=\{v_0, \dots, v_r\}$ and $|v_i \cdot v_j|$ edges between $v_i$ and $v_j$ for $i\neq j$. The obtuse superbase $B$ is \textbf{planar} if $G_B$ is a planar graph.
\end{defi}

Obtuse superbases are closely related to the notion of graph lattices.
\begin{defi}[Graph lattice]
    Let $G=(V,E)$ be a connected, finite graph with no self-loops. Define $\Lambda(G)$ to be the abelian group defined by
    \[
    \Lambda(G)= \frac{\Z\langle v_0, \dots , v_r \rangle}{\Z \langle v_0 + \dots + v_r \rangle},
    \]
    where $V=\{v_0, \dots, v_r\}$ are the vertices of $G$.
    We denote by $d(v_i)$ the degree of $v_i$ and by $e(v_i,v_j)$ the number of edges between $v_i$ and $v_j$.
    The group $\Lambda(G)$ is a free abelian group of rank $r$ and the bilinear pairing defined on the vertices by 
    \begin{equation}\label{eq:graph_lattice_pairing}
v_i\cdot v_j = \begin{cases}
    d(v_i) &\text{if $i=j$}\\
    -e(v_i,v_j) &\text{if $i\neq j$},
\end{cases}
\end{equation}
descends to a bilinear pairing on $\Lambda(G)$, which can be further shown to be positive definite.\footnote{The hypothesis that $G$ be connected is required to ensure that the pairing is positive definite and not just positive semi-definite.} The group $\Lambda(G)$ endowed with this pairing is the \textbf{graph lattice} of $G$.  
\end{defi}
\begin{rem}\label{rem:edge_count}
    Let $H$ be a subset of the vertices of $G$ and $z\in \Lambda(G)$ an element of the form $z=\sum_{v\in H} v$. Then \eqref{eq:graph_lattice_pairing} implies that the self-pairing $\norm{z}$ is equal to the number of edges between $H$ and $G\setminus H$. In particular, this is why $G$ must be connected to ensure that $\Lambda(G)$ is positive definite.
\end{rem}

By construction (the images of) the vertices of $G$ form an obtuse superbase in $\Lambda(G)$ and the associated graph for this obtuse superbase is a copy of $G$.

Thus we see that the relationship between obtuse superbases and graph lattices can be summarized as follows.
\begin{prop}\label{prop:ObtSupBase=graphlattice}
    An integral lattice $L$ admits an obtuse superbase $B$ with associated graph $G_B$ if and only if $L$ is isomorphic to the graph lattice $\Lambda(G_B)$.\qed
\end{prop}

In general, a lattice $L$ can admit multiple obtuse superbases. Recall that two connected graphs $G$ and $G'$ are \textbf{2-isomorphic} if there is a cycle-preserving bijection between their edges sets. 
\begin{lem}\label{lem:2-iso}
Let $L$ be a lattice with an obtuse superbase $B$ and let $G'$ be a graph. Then there is an obtuse superbase $B'$ of $L$ with $G_{B'}$ isomorphic to $G'$ if and only if $G'$ is 2-isomorphic to $G_B$.
\end{lem}
\begin{proof}
If $L$ possesses an obtuse superbase $B$, then $L$ is isomorphic to the graph lattice of $G_B$ by Proposition~\ref{prop:ObtSupBase=graphlattice}. However Greene has shown that two graphs have isomorphic graph lattices if and only if they are 2-isomorphic \cite[Theorem~3.8]{Gr13}.
\end{proof}

\begin{rem}
    Since planarity of a graph is preserved by 2-isomorphism, Lemma~\ref{lem:2-iso} implies that there exists a planar obtuse superbase for a lattice $L$ if and only if every obtuse superbase is planar. That is, planarity of an obtuse superbase is intrinsic to the lattice $L$ rather than the particular choice of obtuse superbase.
\end{rem}

A non-zero vector $v\in L$ is \textbf{irreducible} if for any non-zero $x,y \in L$ with $x+y=v$ we have that $x\cdot y \leq -1$. In the presence of an obtuse superbase $B$, the irreducible elements of $L$ have a nice interpretation in terms of the graph $G_B$: they correspond to connected subgraphs of $G_B$ with connected complements.

\begin{lem}[{\cite[Lemma~2.2]{McCoyAltUnknotting}}]\label{lem:irreducible}
Let $L$ be a lattice with an obtuse superbase $B$. A vector $z \in L\setminus \{0\}$ is irreducible if and only if there exist $R\subseteq B$ such that $z=\sum_{v\in R} v$ and both $R$ and $B\setminus R$ induce connected subgraphs of $G_B$. \qed
\end{lem}

A lattice $L$ is said to be \textbf{indecomposable} if $L$ cannot be decomposed as an orthogonal direct sum $L=L_1 \oplus L_2$ with both $L_1$ and $L_2$ non-zero. 

\begin{lem}[{\cite[Lemma~2.3]{McCoyAltUnknotting}}]\label{lem:2-connected}
Let $L$ be a lattice with an obtuse superbase $B$. The following are equivalent:
\begin{enumerate}
    \item $G_B$ is 2-connected;
    \item every element of $B$ is irreducible;
    \item $L$ is indecomposable. \qed
\end{enumerate}
\end{lem}
We will be focusing on lattices with no elements of length one. 
\begin{lem}[{\cite[Lemma~2.4]{McCoyAltUnknotting}}]
Let $L$ be a lattice with an obtuse superbase $B$. The graph $G_B$ is 2-edge-connected if and only if $\norm{z}\geq 2$ for all $z\in L\setminus\{0\}$. \qed
\end{lem}

Integral lattices have a general bound on the length of vectors appearing in an obtuse superbase in terms of the lattice discriminant.
\begin{lem}\label{lem:universal_bound}
Let $L$ be a lattice without any vectors of norm one. If $v\in L$ is contained in an obtuse superbase for $L$, then $\norm{v}\leq \mathrm{disc}(L)$.
\end{lem}
\begin{proof}
Let $B$ be an obtuse superbase for $L$ with corresponding graph $G_B$. Since $L$ contains no vectors of norm one, the graph $G_B$ is 2-edge-connected. Fix a vertex $v\in B$ and let $H_1, \dots, H_k$ be the connected components of $G_B\setminus \{v\}$. Let $d_i$ denote the number of edges between $H_i$ and $v$. Since $G_B$ is 2-edge-connected, we have $d_i\geq 2$ for each $i$. 
If we fix a choice of spanning tree on each of the $H_i$, then we obtain a spanning tree on all of $G_B$ by adding an edge from $v$ to $H_i$ for each of the $H_i$. Since there are $d_1 \cdots d_k$ possible choices of such a collection of edges, there are at least $d_1 \cdots d_k$ distinct spanning trees on $G_B$. Since each of the $d_i$ satisfies $d_i\geq 2$, there are at least $d_1 + \dots + d_k$ spanning trees on $G_B$. Since the determinant of the graph Laplacian counts the number of spanning trees in $G_B$ and is equal to the discriminant $\mathrm{disc}(L)$, this implies $d_1 + \dots + d_k \leq \mathrm{disc}(L)$. On the other hand, $d_1 + \dots + d_k$ is the degree of $v$ which is equal to $\norm{v}$
\end{proof}

\section{Changemaker lattices}\label{sec:changemaker}
Let $\Z^r$ be a diagonal lattice. We say that a non-zero vector $\bsigma\in \Z^r$ is a \textbf{changemaker vector} if there exists an orthonormal basis $\e_1, \dots, \e_r$ for $\Z^r$ such that $\bsigma$ takes the form
\[
\bsigma = \sigma_1 \e_1 + \dots + \sigma_r \e_r
\]
where the coefficients satisfy $\sigma_1=1$ and
\[\sigma_{i-1}\leq \sigma_i\leq 1+ \sigma_1 + \dots + \sigma_{i-1}\]
for $i=2, \dots, r$. Let $m\leq r$ be the minimal index such that $\sigma_m\geq 2$. We refer to the tuple $(\sigma_m, \dots, \sigma_r)$ as the \textbf{stable coefficients} of $\bsigma$.\footnote{In Part~2 it will be convenient to use some alternative conventions to express the stable coefficients. In that section, the stable coefficients will be written in the form $(\rho_1, \dots, \rho_\ell)$ where the coefficients $\rho_i = \sigma_{r+1-i}$ are a non-increasing sequence and $\ell = r+1-m$.} If no such index exists, that is, $\sigma_r=1$, then the tuple of stable coefficients is defined to be empty.

If there is an index $k\geq 2$ such that $\sigma_k=1+\sigma_1+\dots + \sigma_{k-1}$, then the coefficient $\sigma_k$ is \textbf{tight} and correspondingly the vector $\bsigma$ is \textbf{tight}.

The key combinatorial property of changemaker vectors is contained in the following proposition, whose proof is left as a simple exercise.
\begin{prop}\label{prop:CM_condition}
Let $\bsigma = \sigma_1 \e_1 + \dots + \sigma_r \e_r$ be a changemaker vector.
Then for any $k\leq r$ and any integer $T\leq \sigma_1 + \dots + \sigma_k$, there is a subset $A\subseteq \{1,\dots, k\}$ such that 
\[
\pushQED{\qed} 
T=\sum_{i\in A}\sigma_i.\qedhere
\popQED
\]
\end{prop}

We also make the following observation about changemaker vectors with a given set of stable coefficients.
\begin{prop}\label{prop:CM_existence_bound}
Let $\underline{\rho}=(\rho_1, \dots, \rho_\ell)$ be a tuple of non-decreasing integers with $\rho_\ell \geq 2$ and let $n$ be an integer. 
Then there exists a changemaker vector $\bsigma$ with stable coefficients $\underline{\rho}$ and $\norm{\bsigma}=n$ if and only if $n\geq N-1$ where
\[
N=\sum_{i=1}^\ell \rho_i^2 + \max_{1\leq k \leq \ell} \left(\rho_k - \sum_{i=1}^{k-1} \rho_i \right).
\]
\end{prop}
\begin{proof}
A changemaker vector with these stable coefficients takes the form
\[
\bsigma = \e_{1} + \dots + \e_{m-1} + \rho_1 \e_{m} + \dots + \rho_\ell \e_{m+\ell-1}\in \Z^{m+\ell-1}.
\]
If we write $\bsigma$ in the form
\[
\bsigma = \sigma_1 \e_1 + \dots + \sigma_{m+\ell-1} \e_{m+\ell-1},
\]
then the coefficients satisfy
\[
\sigma_i=
\begin{cases}
1 &\text{if $1\leq i <m$}\\
\rho_{i-m+1} &\text{if $m\leq i \leq m+\ell-1$.}
\end{cases}
\]
Since these coefficients are non-decreasing, we see that $\bsigma$ is changemaker vector if and only if the inequality
\[
\sigma_j\leq 1+ \sigma_1 + \dots+\sigma_{j-1}
\]
holds for all $j$ in the range $m\leq j \leq m+\ell-1$. If we write such a $j$ in the form $j=m-1+k$ for $k$ in the range $1\leq k\leq \ell$ this condition becomes
\begin{align*}
    \rho_{k}&=\sigma_j\\
    &\leq 1+ \sigma_1 + \dots+\sigma_{j-1}\\
    &=m+\rho_{1}+\dots + \rho_{k-1}.
\end{align*}
Thus $\bsigma$ is a changemaker vector if and only if 
\[
m\geq \max_{1\leq k \leq \ell} \left(\rho_k - \sum_{i=1}^{k-1} \rho_i \right).
\]
However, $\norm{\bsigma}= m-1 + \sum_{i=1}^\ell \rho_i^2$, so we see that $\bsigma$ is a changemaker vector if and only if
$\norm{\bsigma}=n\geq N-1$, where $N$ is as in the statement of the proposition.
\end{proof}

For any rational number $p/q>0$, we can define the notion of a $p/q$-changemaker lattice. For technical reasons it is convenient to give the definition of integer and non-integer changemaker lattices separately. We mention also that the vast majority of our work will be with integer changemaker lattices, so the following definition is the most important of the two.
\begin{defi}[Integer changemaker lattice]\label{def:int_CM_lattice}
For an integer $n\geq 1$, an \textbf{$n$-change\-maker lattice} $L$ is an orthogonal complement:
\[
L=\langle\bsigma\rangle^\bot \subseteq \Z^r,
\]
where $\bsigma\in \Z^r$ is a changemaker vector satisfying $\norm{\bsigma}=n$. We will refer to the stable coefficients of $\bsigma$ as the stable coefficients of $L$. We say also that $L$ is tight if the changemaker $\bsigma$ is tight.
\end{defi}
The corresponding definition for non-integer changemaker lattices is the following.
\begin{defi}[Non-integer changemaker lattice]\label{def:non_int_CM_lattice}
Let $p/q>0$ be a rational number which is not an integer. This admits a unique continued fraction of the form
\[p/q=[a_0, \dots, a_l]^-=
a_0 -
    \cfrac{1}{a_1
        - \cfrac{1}{\ddots
            - \cfrac{1}{a_l} } },\]
where $a_0\geq 1$ and $a_i\geq 2$ for $i\geq 1$. Let $\bm{w}_0, \dots, \bm{w}_l\in \Z^{r+s+1}$ be vectors satisfying,
\[
\bm{w}_i\cdot \bm{w}_j =
\begin{cases}
a_i &\text{if $i=j$,}\\
-1 &\text{if $|i-j|=1$,}\\
0 &\text{if $|i-j|>1$,}
\end{cases}
\]
such that for some orthonormal basis $\{\e_{-s}, \dots, \e_r\}$ of $\Z^{r+s+1}$ following hold:
\begin{enumerate}[label=(\Roman*)]
\item $\bm{w}_0$ takes the form $\bm{w}_0=\e_0+\bsigma$ where
\[
\bsigma\in \Z^r=\langle \e_1, \dots, \e_r\rangle
\]
is a changemaker vector.

\item $|\bm{w}_k\cdot \e_i|\leq 1$ for all $1\leq k\leq l$ and all $\e_i$;
\item for any $0\leq i < j\leq l$,
\[\bm{w}_i\cdot \bm{w}_j = -|I_j \cap I_i|=
\begin{cases}
-1 &\text{if $j=i+1$}\\
0 &\text{if $j>i+1$};
\end{cases}\]
where $I_k$ is the set $I_k=\{\e_j\, | \, \bm{w}_k\cdot \e_j \ne 0\}$; and
\item for any $\e_i$ there is $\bm{w}_k$ such that $\e_i\cdot \bm{w}_k\ne 0$.
\end{enumerate}
Then the orthogonal complement
\[L=\langle \bm{w}_0, \dots, \bm{w}_l \rangle^\bot \subseteq \Z^{r+s+1}\]
is a {\em $p/q$-changemaker lattice}. We call the stable coefficients of $\bsigma$ the {\em stable coefficients} of $L$.
\end{defi}
The only examples of non-integer changemaker lattices that we will ever need to consider in detail will be half-integer changemaker lattices, which will make an appearance in \S\ref{sec:slack_lattices} and \S\ref{sec:half_int_basis_search}.
\begin{ex}\label{exam:half_integer_CM}
Definition~\ref{def:non_int_CM_lattice} implies that an $(n-\frac{1}{2})$-changemaker lattice takes the form
\[
L=\langle \e_{-1} -\e_0 ,\e_0 + \sigma_1 \e_1 + \dots + \sigma_r \e_r \rangle^\bot \subseteq \Z^{r+2},
\]
where $\bsigma=\sigma_1 \e_1 + \dots + \sigma_r \e_r$ is a changemaker vector with $\norm{\bsigma}=n-1$.
\end{ex}
The following summarizes the majority of what we need to know about arbitrary changemaker lattices. 
\begin{rem}\label{rem:CM_facts}\hfill
\begin{enumerate}
    \item\label{it:len1fact} A changemaker lattice $L$ never contains any elements of length one. That is, for all $\bv \in L\setminus\{\bm{0}\}$ we have $\norm{\bv}\geq 2$. In the non-integer case this is consequence of Condition IV.
    \item A $p/q$-changemaker lattice is determined up to isomorphism by its stable coefficients. In both the integer and non-integer cases, all remaining non-zero coefficients of $\bsigma$ are equal to one and so the number of such coefficients is determined by the requirement that $\norm{\bsigma}=\lfloor p/q \rfloor$. In the non-integer case, the remaining $\bm{w}_i$ are determined up to automorphism of the ambient diagonal lattice by Conditions II and III.
    \item\label{it:rational_existence} Moreover given a non-decreasing tuple of integers $(\rho_1, \dots, \rho_\ell)$, there is a $p/q$-changemaker lattice $L$ with these stable coefficients whenever $p/q$ is sufficiently large. Explicitly there is a $p/q$-changemaker changemaker lattice with these stable coefficients if and only if $p/q \geq N-1$, where $N$ is the quantity defined in Proposition~\ref{prop:CM_existence_bound}.
    \item\label{it:disc}The discriminant of a $p/q$-changemaker lattice $L$ is $\operatorname{disc}(L)=p$. See \cite[Lemma~3.10]{GreeneLRP} for the integer case and \cite[Remark~2.19]{Mc16} for the general case.
\end{enumerate}
\end{rem}

\section{Integer changemaker lattices}
We now turn our attention to understanding when an integer changemaker lattice can admit an obtuse superbase. Throughout this section we will take
\[
\bsigma=\sigma_1 \e_1 + \dots +\sigma_r\e_r\in \Z^r
\]
to be a changemaker vector with $\sigma_r\geq 2$ and
\[
L=\langle \bsigma\rangle^\bot\subseteq \Z^r
\]
to be the corresponding integer changemaker lattice. The condition that $\sigma_r\geq 2$ implies that the stable coefficients of $L$ are non-empty and that there exists a minimal $1<m\leq r$ such that $\sigma_m>1$.

\begin{rem}\label{rem:empty_coefs_examples}
A changemaker lattice with empty stable coefficients takes the form
\[
L=\langle \e_1 + \dots + \e_r \rangle^\bot\subseteq \Z^r.
\]
This always has an obtuse superbase. For example, one can take
\[
B=\{\e_1-\e_2, \dots, \e_{r-1}-\e_r, \e_r-\e_1\},
\]
as an obtuse superbase. The associated graph $G_B$ is a cycle with $r$ vertices.
\end{rem}

Using Proposition~\ref{prop:CM_condition} one can construct the standard basis for $L$ \cite{GreeneLRP}. The standard basis for $L$ is a collection of vectors $\bv^{(2)}, \dots, \bv^{(r)}\in L$ of the form
\[
\bv^{(k)}=-\e_k + \e_{k-1}+\dots + \e_2 + 2\e_1
\]
if $\sigma_k$ is tight and
\[
\bv^{(k)}=-\e_k + \sum_{i\in A_k} \e_i
\]
where $A_k\subseteq \{1,\dots, k-1\}$ is a subset such that
\begin{enumerate}
    \item[(i)] $\sigma_k=\sum_{i\in A_k} \sigma_i$ and
    \item[(ii)] $A_k$ is maximal with respect to the lexicographical ordering on subsets of $\{1,\dots, k-1\}$
\end{enumerate}
 if $\sigma_k$ is not tight. It can be shown that the standard basis is, in fact, a basis for $L$ and that all elements of a standard basis are irreducible \cite[Lemma~3.13]{GreeneLRP}. 
\begin{ex}
    If $\sigma_k=\sigma_{k-1}$, then the lexicographical maximality condition in the definition of the standard basis implies that $\bv^{(k)}$ takes the form $-\e_k+\e_{k-1}$.
\end{ex}

The following lemma is useful for finding irreducible vectors in a changemaker lattice. For example, it implies, among other things, that the standard basis vector $\bv^{(k)}$ is irreducible if $\sigma_k$ is not tight.
\begin{lem}[{\cite[Proposition 4.11]{Mc16}}]
\label{lem:irred_examples}
Let $\bv \in L$ be a non-zero vector of the form $\bv=\sum_{i\in A} \e_i -\sum_{j\in B} \e_j$ for disjoint $A,B \subseteq \{1, \dots, r\}$. Then $\bv$ is reducible if and only if there are proper non-empty subsets $A'\subseteq A$, $B'\subseteq B$ such that
\[
\pushQED{\qed} 
\sum_{i\in A'}\sigma_i =\sum_{j\in B'} \sigma_j.\qedhere
\popQED
\]
\end{lem}

Next we recall restrictions on decomposable changemaker lattices.
\begin{lem}[{\cite[Lemma~5.1]{GreeneLRP}}]\label{lem:decomposable_properties}
Suppose that the changemaker lattice $L$ is decomposable. Then $L$ has two indecomposable summands. Moreover, if $m$ is the minimal index such that $\sigma_m>1$, then $\sigma_m = m-1$. \qed
\end{lem}
This allows us to show that an obtuse superbase for $L$ contains a basis consisting of irreducible vectors.
To do this we will make use of the following construction of 2-isomorphic graphs. Let $G$, $G'$ denote a pair of disjoint graphs with distinguished vertices $w_1,w_2\in G$ and $w'\in G'$. For $i=1,2$, let $G_i$ be the graph obtained by identifying the vertices $w_i$ and $w'$ into a vertex $v_i$. The graphs $G_1$ and $G_2$ are 2-isomorphic via the natural bijection between their edge sets and we say that $G_2$ is obtained from $G_1$ by a \textbf{special switch} \cite[\S3.4]{Gr13}.
\begin{lem}\label{lem:one_reducible}
Any obtuse superbase for an integer changemaker lattice $L$ contains at most one reducible vector. In particular, any obtuse superbase for $L$ contains a basis of irreducible vectors.
\end{lem}
\begin{proof}
Let $B$ be an obtuse superbase for $L$ containing two distinct reducible vectors $\bm{v}$ and $\bm{w}$. By Lemma~\ref{lem:2-connected}, both $\bm{v}$ and $\bm{w}$ are cut-vertices in $G_B$. We see that there is a graph $G'$, obtained by a special switch on $G_B$, such that $G'$ contains a vertex $u$ such that $G'\setminus\{u\}$ has at least three connected components. Since $G'$ and $G_B$ are 2-isomorphic, Lemma~\ref{lem:2-iso} shows there is an obtuse superbase $B'$ for $L$ such that $G_{B'}\cong G'$. However, the connected components of $G'\setminus\{u\}$ correspond to pairwise orthogonal subspaces of $L$ and so we see that $L$ must split into at least three indecomposable summands. However, this contradicts Lemma~\ref{lem:decomposable_properties} and so we conclude that $B$ can contain at most one reducible vector.
\end{proof}

We also obtain bounds on the coefficients of irreducible vectors.
\begin{lem}\label{lem:irreducible_bounds}
Let $\bm{z}\in L$ be an irreducible vector. If $\sigma_k=1$, then 
    \[|z_k|\leq 
    \begin{cases}
    2 &\text{if $L$ is tight,}\\
    1 &\text{otherwise.}
    \end{cases}
    \]
\end{lem}
\begin{proof}
We can assume without loss of generality that $k=1$ and that $z_1\geq 1$. Let $g>1$ be minimal such that $z_g\leq 0$. By Proposition~\ref{prop:CM_condition}, there is a subset $A\subseteq \{1,\dots, g-1\}$ such that $\sigma_g -1 =\sum_{i\in A}\sigma_i$. If $\bsigma$ is not tight, then the vector
\[
\bsigma'=\sigma_2\e_2 + \dots+ \sigma_r\e_r,
\]
is also a changemaker vector. Thus if $\bsigma$ is not tight, we may apply Proposition~\ref{prop:CM_condition} to $\bsigma'$ instead of $\bsigma$ in order to assume that $1\not\in A$. Thus $\bm{x}=-\e_g +\e_1+ \sum_{i\in A}\e_i$ is a vector in $L$ satisfying $1\leq x_1\leq 2$ with $x_1=2$ only if $\sigma_g$ is tight. If $\bm{z}=\bm{x}$, then $\bm{z}$ satisfies the required bounds by construction. Thus we assume that $\bm{x}\neq \bm{z}$. The irreducibility of $\bm{z}$ gives
\begin{align*}
    -1&\geq \bm{x}\cdot(\bm{z}-\bm{x})\\
    &=-(z_g+1)+ \sum_{i \in A\setminus\{1\}}(z_i -1) + x_1(z_1-x_1)\\
    &\geq -1 + x_1(z_1-x_1).
\end{align*}
Using the last line we obtain the inequality $z_1\leq x_1$. Since $x_1\leq 2$ with equality only if $L$ is tight, this establishes the result.
\end{proof}
We can also control the elements satisfying $\norm{\bv}=2$ in an obtuse superbase. Note that the following lemma is also true in the case of empty stable coefficients as shown by the obtuse superbase exhibited in Remark~\ref{rem:empty_coefs_examples}.
\begin{lem}\label{lem:standard_norm_twos} Let $L$ be a changemaker lattice with non-trivial stable coefficients.
If $L$ admits an obtuse superbase, then it admits an obtuse superbase $B$ such that $-\e_k+\e_{k-1}\in B$ for all $k$ such that $\sigma_k=\sigma_{k-1}$.
\end{lem}
\begin{proof}
Let $B$ be the obtuse superbase for $L$ such that $G_B$ contains the maximal possible number of vertices of degree two.
\begin{claim}
    Let $\bm{x}\in L$ be a vector with $\norm{\bm{x}}=2$. Then there is a path of degree two vertices $\bm{w}_0, \dots, \bm{w}_k$ in $G_B$ such that
    \[
    \bm{x}=\pm(\bm{w}_0+ \dots+ \bm{w}_k)
    \]
\end{claim}
\begin{proof}[Proof of Claim]
As discussed in Remark~\ref{rem:CM_facts}\eqref{it:len1fact}, $L$ does not contain any vectors of length one. Thus, $\bm{x}$ is irreducible. By Lemma~\ref{lem:irreducible}, this implies that there is a connected subgraph $R$ of $G_B$ such that $\bm{x}=\sum_{\bv\in R} \bm{v}$ and the complement of $R$ is also connected. Let $S$ denote the complement of $R$. Since $\norm{\bm{x}}=2$, Remark~\ref{rem:edge_count} shows that there are precisely two edges between $R$ and $S$.

We will see that if $R$ and $S$ both contain vertices of degree at least three, then $G_B$ is 2-isomorphic to a graph with more vertices of degree two. By Lemma~\ref{lem:2-iso}, this would contradict the maximality assumption on $G_B$.

Suppose that both $R$ and $S$ contain a vertex of degree at least three. Let $e$ and $e'$ be the two edges between $R$ and $S$. The edge $e$ is contained in a path $\bm{u}^{(0)}, \dots, \bm{u}^{(k)}$, where $\bm{u}^{(0)}\in R$ and $\bm{u}^{(k)}\in S$ both have degree at least three. We perform a 2-isomorphism by contracting the path containing $e$ to a single vertex and subdividing the edge $e'$ into $k$ edges. This increases the number of vertices of degree two by one, since the contraction eliminates $k-1$ vertices of degree two, but the subdivision creates $k$ such vertices.
\end{proof}
Suppose that we have some sequence $\sigma_a, \dots, \sigma_b$ of changemaker coefficients satisfying 
\[
\sigma_a= \dots = \sigma_b,
\]
with $a<b$. Moreover suppose that this is maximal in the sense that $a=1$ or $\sigma_{a-1}<\sigma_a$ and $b=r$ or $\sigma_{b+1}>\sigma_b$. We will show that after permuting the unit vectors $\e_a, \dots, \e_b$ we can assume that $-\e_{a+1}+\e_a, \dots, -\e_{b}+\e_{b-1}$ are in $B$.

Let $H$ be the subgraph of $G_B$ induced by the vectors of $B$ taking the form $-\e_i + \e_j$ for some $i,j\in \{a,\dots, b\}$. Let $V$ denote the number of vertices of $H$ and $E$ the number of edges. The above claim shows that the $b-a$ linearly independent vectors $-\e_{a+1}+\e_{a}, \dots, -\e_{b}+\e_{b-1}$ are all in the span of the vertices $H$. This implies that $V\geq b-a$. Moreover for each index $a\leq i \leq b$, there exists at least one $\bm{u}\in H$ with $u_i\neq 0$.

Next, note that $H$ contains no cycles. If $H$ contained a cycle, then the connectivity of $G_B$ would imply that $G_B$ is itself a cycle. In this case, every irreducible vector in $L$ would satisfy $\norm{\bm{x}}=2$, which cannot happen unless the stable coefficients of $L$ are trivial. Thus, we have $E\leq V-1$.

Next observe that if $\bm{u},\bv\in H$ are distinct vertices, then $\bm{u}$ and $\bv$ are adjacent in $H$ if and only if there is a unique index $a\leq i\leq b$ such that $u_i$ and $v_i$ are both non-zero\footnote{This is easily seen by writing $\bm{u}$ and $\bv$ in the forms $\bm{u}=-\e_\alpha + \e_\beta$ and $\bv=-\e_\gamma +\e_\delta$ and calculating the various possibilities}. Moreover, when $\bm{u}$ and $\bv$ are adjacent, we necessarily have $u_i=-v_i$ for this index. Thus for each index $a\leq i\leq b$, there are at most two vertices in $H$ with $v_i\neq 0$ and these are necessarily adjacent. Thus we obtain a bijection between edges of $H$ and indices $i$ such that $\sum_{\bv \in H} |v_i|=2$. Moreover for all other indices we have $\sum_{\bv \in H} |v_i|=1$.

In particular, we see that we may compute $E$ via the following sum
\[
E=\sum_{i=a}^b \left(-1+\sum_{\bv\in H} |v_i|  \right)=\sum_{i=a}^b \sum_{\bv\in H} |v_i| -(b-a+1).
\]
However, we have $\sum_{i=a}^b \sum_{\bv\in H} |v_i|=2V$ and so we obtain the equality
\[
E=2V -(b-a+1).
\]
When combined with the earlier observations that $V\geq b-a$ and $E\leq V -1$, this quickly yields the fact that $V=b-a$ and $E=V-1$. Thus $H$ is a path of length $b-a$.

Let $\bv^{(1)}, \dots, \bv^{(b-a)}$ be the vertices of the path $H$ ordered so that $\bv^{(j)}$ and $\bv^{(j+1)}$ are adjacent for all $j$. For each pair $\bv^{(j)}$ and $\bv^{(j+1)}$, there is a unique index $a\leq i \leq b$ such that $\bv^{(j)}_i=-\bv^{(j+1)}_i\in \{\pm 1\}$ and $\bv^{(k)}_i=0$ for all other $k$. Given this, it is not hard see that by permuting the basis vectors $\e_a, \dots, \e_b$ and potentially reversing the order of labelling on the path $\bv^{(1)}, \dots, \bv^{(b-a)}$ we may assume that these vertices take the form   
\[
\bv^{(1)}=-\e_{a+1}+\e_a, \dots, \bv^{(b-a)}=-\e_b+\e_{b-1},
\]
as required.
\end{proof}

The following is useful for ruling out certain sets of stable coefficients from possessing obtuse superbases.
\begin{prop}\label{prop:4_string}
Let $L$ be a changemaker lattice with non-trivial stable coefficients that admits an obtuse superbase. Assume the changemaker coefficients satisfy \[\sigma_k = \dots = \sigma_{k+3}\] for some $k$. If $\bm{z}\in L$ is an irreducible vector with $z_k=z_{k+1}$ and $z_{k+2}=z_{k+3}$, then $z_{k+2}=z_{k+1}$.
\end{prop}
\begin{proof}
By Lemma~\ref{lem:standard_norm_twos} we can assume that $L$ has an obtuse superbase $B$ containing vectors $\bm{w}_0=\e_k-\e_{k+1}, \bm{w}_1=\e_{k+1}-\e_{k+2}, \bm{w}_2=\e_{k+2}-\e_{k+3}$. If $z_k=z_{k+1}$ and $z_{k+2}=z_{k+3}$, then we have $\bm{z}\cdot \bm{w}_0=\bm{z}\cdot \bm{w}_2=0$. Since $\bm{z}$ is irreducible, Lemma~\ref{lem:irreducible} gives a connected subgraph $R$ of $G_B$ such that $\bm{z}=\sum_{\bv\in R}\bv$.

Up to sign, the pairing $\bm{z}\cdot \bm{w}_0$ counts the number of edges incident to $\bm{w}_0$ with precisely one end point in $R$. Therefore, if precisely one of $\bm{w}_0$ and $\bm{w}_1$ were contained in $R$, then we would have $|\bm{z}\cdot \bm{w}_0|=1$. Since $\bm{z}\cdot \bm{w}_0=0$, it follows that either both $\bm{w}_0$ and $\bm{w}_1$ are contained in $R$ or both are contained in the complement of $R$. Since $\bm{z}\cdot \bm{w}_2=0$, a similar argument applies to show that we cannot have precisely one of $\bm{w}_1$ and $\bm{w}_2$ contained in $R$.
Thus all three of $\bm{w}_0, \bm{w}_1, \bm{w}_2$ are contained in $R$ or all three are contained in the complement of $R$. In either case, neither of the edges incident to $\bm{w}_1$ has precisely one end point in $R$. This implies that $\bm{w}_1\cdot \bm{z} =0$ or, equivalently, that $z_{k+2}=z_{k+1}$.
\end{proof}

When applied to the changemaker coefficients satisfying $\sigma_k=1$, Proposition~\ref{prop:4_string} yields the following restriction. This formed a key step in the proof of \cite[Theorem~1.2]{Mc17}.
\begin{prop}\label{prop:OSB_m_bound}
Let $L$ be an integral changemaker lattice of the form
\[
L=\langle \sigma_1 \e_1+ \dots+ \sigma_r \e_r\rangle^\bot\subseteq \Z^r
\]
with non-trivial stable coefficients. Let $m$ be minimal such that $\sigma_m\geq 2$. If $L$ admits an obtuse superbase, then
\[
\sigma_m\geq m-2
\]
and
\[\norm{\bsigma}\leq 1+ \sigma_m + \sum_{i=m}^{r}\sigma_i^2.\]
\end{prop}
\begin{proof}

Suppose that $m\geq 3+\sigma_m$. We will show that $L$ does not admit an obtuse superbase. Consider the vector $\bv=-\e_m+\e_1+\dots + \e_{\sigma_m}\in L$. By Lemma~\ref{lem:irred_examples}, $\bv$ is irreducible. However we have 
\[\sigma_{\sigma_m-1}=\sigma_{\sigma_m}=\sigma_{\sigma_m+1}=\sigma_{\sigma_m+2}=1\]
and
\[
1=v_{\sigma_m-1}=v_{\sigma_m}\neq v_{\sigma_m+1}=v_{\sigma_m+2}=0
\]
so Proposition~\ref{prop:4_string} implies that $L$ does not admit an obtuse superbase. The bound on $\norm{\bsigma}$ follows from the fact that
\[
\norm{\bsigma}=m-1+\sum_{i=m}^r \sigma_i^2.\qedhere
\]
\end{proof}

\section{Very slack changemaker lattices}\label{sec:slack_lattices}
The aim of this section is to understand when a very slack changemaker lattice can admit an obtuse superbase. The conclusion to this work is Theorem~\ref{thm:very_slack_technical}, which is the only result in this section that is used elsewhere in the paper. The bulk of the section is devoted to proving a technical lemma, namely Lemma~\ref{lem:modified_basis}.

Let $\bsigma= \sigma_1 \e_1 + \dots + \sigma_r \e_r$ be a changemaker vector with $\sigma_1=1$ and $\sigma_r>1$. We take $m>1$ to be minimal such that $\sigma_m>1$. The vector $\bsigma$ is \textbf{very slack} if it satisfies
\[
\sigma_k \leq \sigma_1 + \dots + \sigma_{k-1} -1
\]
for all $\sigma_k>1$. Unsurprisingly, we say that a changemaker lattice
\[
L=\langle \bsigma \rangle^\bot \subseteq \Z^r
\]
is very slack if $\bsigma$ is a very slack changemaker vector.
\begin{rem}\label{rem:very_slack}
An argument similar to that of Proposition~\ref{prop:CM_existence_bound} shows that if $L$ is a $n$-changemaker lattice with stable coefficients $\underline{\rho}=(\rho_1,\dots, \rho_\ell)$ (ordered as an non-decreasing tuple) and
    \[
N=\sum_{i=1}^\ell \rho_i^2 + \max_{1\leq k \leq \ell} \left(\rho_k - \sum_{i=1}^{k-1} \rho_i \right),
\]
then $L$ is very slack if and only if $n\geq N+1$.
\end{rem}

The very slack condition leads to the following refinement of Proposition~\ref{prop:CM_condition}.
\begin{lem}\label{lem:veryslack_CM_condition}
Let $\bsigma= \sigma_1 \e_1 + \dots + \sigma_r \e_r$ be a very slack changemaker vector and let $m>1$ be minimal such that $\sigma_m >1$. Then for every $\sigma_k>1$, there is $A\subseteq \{1,\dots, k-1\}$ such that
\[
\text{$\sigma_k = \sum_{i\in A} \sigma_i$ and $1\leq |A \cap \{1, \dots, m-1\}|\leq m-2$.}
\]
\end{lem}
\begin{proof}
We prove the statement by induction on $k\geq m$. Firstly the statement is true for $\sigma_m$ since the condition that $\bsigma$ is very slack implies that $\sigma_m \leq m-2$ and hence that
\[\sigma_m =\sigma_1 + \dots + \sigma_{\sigma_m}.\]
Now suppose that $k>m$. By Proposition~\ref{prop:CM_condition}, there exists $B\subseteq\{1, \dots, k-1\}$ such that $\sigma_k = \sum_{i\in B} \sigma_i$. There are three possibilities for the size of $|B \cap \{1, \dots, m-1\}|$ to consider.
\begin{itemize}
    \item If $1\leq |B \cap \{1, \dots, m-1\}|\leq m-2$, then we may simply take $A=B$.
    \item If $|B \cap \{1, \dots, m-1\}|=0$, then take $\ell=\min B\geq m$. By induction there exists $A'\subseteq \{1, \dots, \ell -1\}$ such that $\sigma_\ell = \sum_{i\in A'} \sigma_i$ and $1\leq |A' \cap \{1, \dots, m-1\}|\leq m-2$. We may take $A=(B\setminus \{\ell\}) \cup A'$ in this case.
    \item If $|B \cap \{1, \dots, m-1\}|=m-1$, then we take $\ell\geq m$ to be the minimal index not contained in $B$. Since $\sigma_k<\sigma_1 + \dots + \sigma_{k-1}$, we have $\ell<k$. By induction, there exists $A'\subseteq \{1, \dots, \ell -1\}$ such that $\sigma_\ell = \sum_{i\in A'} \sigma_i$ and $1\leq |A' \cap \{1, \dots, m-1\}|\leq m-2$. We may take $A=(B\setminus A') \cup \{\ell\}$ in this case.\qedhere
\end{itemize}
\end{proof}
We will also using the following two general results on obtuse superbases.
\begin{lem}[{\cite[Lemma~2.1]{McCoyAltUnknotting}}]\label{lem:vertex_sums}
Let $L$ be a lattice with an obtuse superbase $B$. If $\bm{z}=\sum_{\bv\in R} \bv$ for $R\subset B$, then for any $\bm{x}\in L$, we have
\[
\pushQED{\qed} 
(\bm{z}-\bm{x})\cdot \bm{x}\leq 0.\qedhere
\popQED
\]  
\end{lem}
\begin{lem}[{\cite[Lemma~3.5]{Mc17}}]\label{lem:basis_modification}
 Let $L$ be an indecomposable lattice with an obtuse superbase $B$. Suppose that $\bv\in B$ can be decomposed as $\bv=\bm{x}+\bm{y}$ with $\bm{x}\cdot \bm{y} =-1$. Then there exists a unique pair $\bu^{(1)},\bu^{(2)}\in B\setminus\{\bv\}$ such that $\bu^{(1)}\cdot \bm{x}=1$, $\bu^{(2)}\cdot \bm{y} =1$ and $\bu^{(1)}\cdot \bu^{(2)}=-1$.
Moreover,
\[
B'=B\setminus\{\bv,\bu^{(1)},\bu^{(2)}\}\cup \{\bm{x},\bm{y}, \bu^{(1)}+\bu^{(2)}\}
\]
is also an obtuse superbase for $L$.\qed
 \end{lem}

\subsection{Proving Lemma~\ref{lem:modified_basis}}
This entire subsection is devoted to the proof of the following lemma.
\begin{lem}\label{lem:modified_basis}
Let $L$ be a very slack changemaker lattice. If $L$ admits an obtuse superbase, then $L$ an admits an obtuse superbase $B$, such that $\bv^{(2)}=-\e_2 +\e_1$ is in $B$ and there is a vertex $\bm{w}$ with $w_1=w_2=-1$. Moreover for any other vertex $\bv\in B$ distinct from $\bv^{(2)}$ and $\bm{w}$ we have $v_1=0$.
\end{lem}

\noindent
\textbf{Working Assumption:}  Throughout this section we take
\[
L=\langle \sigma_1 \e_1 + \dots + \sigma_r \e_r \rangle^\bot\subseteq \Z^r,
\]
to be a very slack changemaker lattice admitting an obtuse superbase $B$. We take $m> 2$ to be minimal such that $\sigma_m\geq 2$. The very slack condition implies that $\sigma_m \leq m-2$ and hence that $m\geq 4$.

This has a number of consequences:
\begin{itemize}
    \item Proposition~\ref{prop:OSB_m_bound} implies that $\sigma_m=m-2$.
    \item Lemma~\ref{lem:decomposable_properties} shows that $L$ is indecomposable.
    \item Lemma~\ref{lem:2-connected} implies that every element of $B$ is therefore irreducible and the associated graph $G_B$ is 2-connected.
\end{itemize}

Given the existence of the obtuse superbase, we may further refine Lemma~\ref{lem:veryslack_CM_condition}
\begin{lem}\label{lem:veryslack_condition}
If $L$ is very slack and admits an obtuse superbase, then for all $\sigma_k>1$, there exists a subset $A\subseteq\{1, \dots, k-1\}$ such that
\[
\sigma_k = \sum_{i\in A} \sigma_i
\]
and 
\[
\text{$A \cap \{1, \dots, m-1\}=\{1\}$ or $A \cap \{1, \dots, m-1\}=\{2, \dots, m-1\}$.}
\]
\end{lem}
\begin{proof}
Lemma~\ref{lem:veryslack_CM_condition} shows that there exists a set $A\subseteq\{1, \dots, k-1\}$ such that
\[
\sigma_k = \sum_{i\in A} \sigma_i
\]
and
\[1\leq |A \cap \{1, \dots, m-1\}|\leq m-2.\]
To conclude the proof it suffices to show that we have $A$ with 
\[
\text{$|A \cap \{1, \dots, m-1\}|=1$ or $m-2$}.
\]
Thus suppose that 
\[|A \cap \{1, \dots, m-1\}|=\ell\] for $2\leq \ell \leq m-3$. 
Up to reordering the indices, we may assume that
\[
A \cap \{1, \dots, m-1\}=\{1, \dots, \ell\}.
\]
The vector
\[
\bv=-\e_k +\sum_{i\in A} \e_i \in L
\]
is irreducible by Lemma~\ref{lem:irred_examples}. But by construction $\bv$ satisfies 
\[1=v_{\ell-1}=v_{\ell}\neq v_{\ell+1}=v_{\ell+2}=0,\] which Proposition~\ref{prop:4_string} shows to be incompatible with the existence of an obtuse superbase for $L$.
\end{proof}

Now we turn our attention to understanding the possibilities for the obtuse superbase $B$. By Lemma~\ref{lem:standard_norm_twos} we can assume that the the standard basis vectors 
\[\bv^{(2)}=\e_1-\e_2, \dots, \bv^{(m-1)}=\e_{m-2}-\e_{m-1}\]
are all in $B$. There are vertices $\bu^{(1)},\bu^{(2)}\not\in\{\bv^{(2)}, \dots, \bv^{(m-1)}\}$ such that $\bu^{(1)}$ pairs non-trivially with $\bv^{(2)}$ and $\bu^{(2)}$ pairs non-trivially with $\bv^{(m-1)}$. Since $G_B$ is 2-connected, we see that $\bu^{(1)}$ and $\bu^{(2)}$ must be distinct. Now we study the possibilities for $\bu^{(1)}$ and $\bu^{(2)}$.
\begin{lem}
The vertex $\bu^{(1)}$ must take one of the two forms:
\begin{enumerate}
    \item [(A1)] $u^{(1)}_1=-1$ and $u^{(1)}_2=\dots =u^{(1)}_{m-1}= 0$
    \item [(A2)] $u^{(1)}_1=0$ and $u^{(1)}_2=\dots =u^{(1)}_{m-1}= 1$
\end{enumerate}
The vertex $\bu^{(2)}$ must take one of the following two forms:
\begin{enumerate}
    \item [(B1)] $u^{(2)}_{m-1}=1$ and $u^{(2)}_1=\dots =u^{(2)}_{m-2}= 0$
    \item [(B2)] $u^{(2)}_{m-1}=0$ and $u^{(2)}_1=\dots =u^{(2)}_{m-2}= -1$
\end{enumerate}
Any other vertex $\bv\not\in\{\bv^{(2)},\dots, \bv^{(m-1)}, \bu^{(1)},\bu^{(2)}\}$ must satisfy one of the following conditions:
\begin{enumerate}
    \item [(C1)] $v_1=\dots =v_{m-1}= 0$
    \item [(C2)] $v_1=\dots =v_{m-1}= 1$
    \item [(C3)] $v_1=\dots =v_{m-1}= -1$ 
\end{enumerate}
\end{lem}
\begin{proof}
Since $L$ is not tight, Lemma~\ref{lem:irreducible_bounds}, shows that for any $\bv\in B$, we have $|v_i|\leq 1$ for $i=1, \dots, m-1$. The remaining restrictions on coefficients come from considering the pairings with $\bv^{(2)}, \dots, \bv^{(m-1)}$. In particular, $\bu^{(1)}\cdot \bv^{(k)}=0$ for $k>2$ and $\bu^{(2)}\cdot \bv^{(k)}=0$ for $k<m-1$.
\end{proof}

The following lemma allows us to restrict the number of $\bv\in B$ with $v_1\neq 0$.
\begin{lem}\label{lem:total_e1_bound}
Let $\bm{z}=\sum_{\bv\in R} \bv$ for $R\subseteq B$. Then
\[
|z_k| \leq 2
\]
for any $k$ with $\sigma_k=1$.
\end{lem}
\begin{proof}
It suffices to show that $|z_1| \leq 2$. Without loss of generality, we can consider the case that $z_1 >0$. Let $g>1$ be minimal such that $z_g \leq 0$. By Proposition~\ref{prop:CM_condition}, there is a subset $A \subseteq \{1, \dots, g-1\}$ such that $\sigma_g -1 = \sum_{i \in A}\sigma_i$. This allows us to define $\bm{x}\in L$ by
\[
\bm{x}=-\e_g + \e_1 + \sum_{i\in A} \e_i.
\]
By construction we have $x_1 \in \{1,2\}$.
By Lemma~\ref{lem:vertex_sums}, we have that
\begin{align*}
    0&\geq \bm{x}\cdot(\bm{z}-\bm{x}) = -(z_g +1) + (x_1)(z_1 -x_1)+ \sum_{i\in A\setminus\{1\}} (z_i -1)\\
    &\geq -1 + x_1(z_1 -x_1).
\end{align*}
Since $x_1 \in \{1,2\}$, this implies that $z_1\leq \frac52$. Since $z_1$ is an integer, this implies $z_1\leq 2$, as required.
\end{proof}
This implies several possible combinations of vertices in $B$.
\begin{lem}\label{lem:veryslack_trichotomy} We can assume one of the following situation holds:
\begin{enumerate}[label=(\Roman*)]
    \item\label{it:vslackI} $\bu^{(1)}$ is of type $A1$, $\bu^{(2)}$ is of type $B1$, there is one vertex of type $C2$, one vertex of type $C3$, and all other vertices are of type $C1$.
    \item\label{it:vslackII} $\bu^{(1)}$ is of type $A2$, $\bu^{(2)}$ is of type $B1$, and there is one vertex of type $C3$ and all other vertices are of type $C1$.
    \item\label{it:vslackIII} $\bu^{(1)}$ is of type $A2$, $\bu^{(2)}$ is of type $B2$, and all other vertices are of type $C1$.
\end{enumerate}
\end{lem}
\begin{proof}
Lemma~\ref{lem:total_e1_bound} puts some immediate restrictions on $B$. We see immediately that $B$ contains at most one vertex of type $C2$ and at most one vertex of type $C3$. Furthermore,  if $\bu^{(1)}$ is of type $A2$, then $B$ contains no vertices of type $C2$. Likewise, if $\bu^{(2)}$ is of type $B2$, then $B$ contains no vertices of type $C3$.

Suppose that $\bu^{(1)}$ is of type $A1$ and $\bu^{(2)}$ is of type $B1$. By Lemma~\ref{lem:irred_examples}, the vector
\[
\bv^{(m)}=-\e_m+\e_{m-1}+\dots + \e_2
\]
is irreducible and hence, by Lemma~\ref{lem:irreducible}, expressible as a sum of distinct elements of $B$. This is only possible if $B$ contains a vertex of type $C2$. Since the elements of $B$ sum to zero it follows that there must also be a vertex of type $C3$. This puts us in case~\ref{it:vslackI}.

Suppose that $\bu^{(1)}$ is of type $A2$ and $\bu^{(2)}$ is of type $B1$. Since the vertices of $B$ sum to zero it follows that there exists a vertex of type $C2$. This puts us in case~\ref{it:vslackII}.

Suppose that $\bu^{(1)}$ is of type $A1$ and $\bu^{(2)}$ is of type $B2$. If one reverses the indexing on the $\e_1, \dots, \e_{m-1}$ and multiplies all elements of $B$ by minus one, then one obtains an obtuse superbase in which the element $\bu^{(1)}$ is of type $A2$ and the new $\bu^{(2)}$ is of type $B1$. Thus we can again assume that we are in case~\ref{it:vslackII}.

Finally, if $\bu^{(1)}$ is of type $A2$ and $\bu^{(2)}$ is of type $B2$, then there are no vertices of type $C2$ or $C3$ and we are in case~\ref{it:vslackIII}.
\end{proof}
We notice that if $B$ is of the form given by cases~\ref{it:vslackII} or~\ref{it:vslackIII}, then the conclusion of Lemma~\ref{lem:modified_basis} is already satisfied. Thus we need to address case~\ref{it:vslackI}. The following lemma will allow us to apply Lemma~\ref{lem:basis_modification}.
\begin{lem}\label{lem:typeC2_splitting}
Let $\bv$ be a vertex of type $C2$. Then there are vectors $\bm{x},\bm{y}\in L$ satisfying $\bv=\bm{x}+\bm{y}$ and $\bm{x}\cdot \bm{y} =-1$ such that
\[x_1 = 1 \quad \text{and}\quad x_2 = \dots = x_{m-1}=0,
\]
and
\[y_1 = 0 \quad \text{and}\quad y_2 = \dots = y_{m-1}=1.
\]
\end{lem}
\begin{proof}
Let $g\geq m$, be minimal such that $v_g \leq 0$. By Lemma~\ref{lem:veryslack_condition}, there is a subset $A\subset \{1, \dots, g-1\}$ such that
\[
\sigma_g = \sum_{i\in A} \sigma_i
\]
and
\[
A \cap \{1, \dots, m-1\}=\{1\} \,\text{or}\, \{2, \dots, m-1\}.
\]
Thus we can set $\bm{z}=-\e_g + \sum_{i\in A} \e_i \in L$. 

Since $\bm{z}\neq \bv$, irreducibility of $\bv$ implies that $\bm{z}\cdot (\bv-\bm{z})\leq -1$. However, we also have the lower bound
\[
    \bm{z}\cdot(\bv-\bm{z}) = -(v_g +1) +\sum_{i\in A} (v_i -1)\geq -1.
\]
Thus we have that $\bm{z}\cdot (\bv-\bm{z})= -1$.
If $A \cap \{1, \dots, m-1\}=\{1\}$, then we take $\bm{x}=\bm{z}$ and $\bm{y}=\bv-\bm{z}$. Otherwise we take $\bm{y}=\bm{z}$ and $\bm{x}=\bv-\bm{z}$.
\end{proof}

With this in place we prove Lemma~\ref{lem:modified_basis}.
\begin{proof}[Proof of Lemma~\ref{lem:modified_basis}]
We may assume that $L$ admits an obtuse superbase $B$ satisfying one of the three cases described by Lemma~\ref{lem:veryslack_trichotomy}.
In cases~\ref{it:vslackII} or~\ref{it:vslackIII}, $B$ already satisfies the necessary conditions so there is nothing further to verify. In case~\eqref{it:vslackII}, $\bm{w}$ is the unique vertex of type $C3$. In case~\ref{it:vslackIII}, $\bm{w}$ is the unique vertex of type $B2$.

Thus we assume that we are in case~\ref{it:vslackI}. We will take $\bm{w}$ to be the unique vertex of type $C3$, however we must modify the obtuse superbase to arrange that $\bv^{(2)}$ and $\bm{w}$ are the only vertices with $v_1 \neq 0$. Let $\bv$ be the unique vertex of type $C2$ in $B$. By Lemma~\ref{lem:typeC2_splitting} we can decompose this as $\bv=\bm{x}+\bm{y}$, where $\bm{x},\bm{y}\in L$ satisfy $\bm{x}\cdot \bm{y} =-1$, 
\[x_1 = 1 \quad \text{and}\quad x_2 = \dots = x_{m-1}=0,
\]
and
\[y_1 = 0 \quad \text{and}\quad y_2 = \dots = y_{m-1}=1.
\]

Now we apply Lemma~\ref{lem:basis_modification} to produce a new obtuse superbase containing $\bm{x}$ and $\bm{y}$.

We have $\bm{x}\cdot \bv^{(2)} =1$ and $\bm{y}\cdot \bv^{(3)}=0$. Since $\bm{y}$ must have positive pairing with a vertex of $B$ which is adjacent to $\bv^{(2)}$, we must have $\bm{y}\cdot \bu^{(1)} =1$. Thus, by Lemma~\ref{lem:basis_modification}, we have that
\[
B'= (B\setminus\{\bv^{(2)}, \bu^{(1)}, \bv \})\cup \{\bm{x}, \bm{y} ,\bu^{(1)}+\bv^{(2)} \}
\]
is also an obtuse superbase for $L$. This is an obtuse superbase in which $\bm{x}$ and $\bm{w}$ are the only vertices with $v_1 \neq 0$ (note that $u_1^{(1)}+v_1^{(2)}=0$). However $B'$ no longer contains $\bv^{(2)}$. We apply Lemma~\ref{lem:basis_modification} again. Since $x_1=1$ and $x_2=0$. We have $(\bm{x}-\bv^{(2)})\cdot \bv^{(2)}=-1$. We have $\bv^{(2)}\cdot (\bu^{(1)}+\bv^{(2)})=1$ and $(\bm{x}-\bv^{(2)})\cdot \bv^{(3)}=1$. Thus we obtain yet another obtuse superbase
\[
B''= (B' \setminus\{\bm{x},\bu^{(1)}+\bv^{(2)},\bv^{(3)} \})\cup \{\bv^{(2)}, \bm{x}-\bv^{(2)}, \bu^{(1)}+\bv^{(2)}+\bv^{(3)}\},
\]
which contains $\bv^{(2)}$ and $\bm{w}$ as the unique vertices with non-zero pairing with $\e_1$. $B''$ can also be written directly as 
\[
B''= (B \setminus\{\bv,\bu^{(1)},\bv^{(3)} \})\cup \{\bm{y}, \bm{x}-\bv^{(2)}, \bu^{(1)}+\bv^{(2)}+\bv^{(3)}\}.\qedhere
\]
\end{proof}

\subsection{The main result}
Armed with Lemma~\ref{lem:modified_basis}, we are ready to proof the key technical result of this section. 
\begin{thm}\label{thm:very_slack_technical}
Let $L_n$ be a very slack $n$-changemaker lattice which admits an obtuse superbase $B$. Let $L_{n-\frac12}$ and $L_{n-1}$ be the $(n-\frac12)$- and $(n-1)$-changemaker lattices with the same stable coefficients as $L_n$. Then
\begin{enumerate}
    \item $L_{n-\frac12}$ admits an obtuse superbases which is planar if $B$ is planar and
    \item $L_{n-1}$ is decomposable.
\end{enumerate}
\end{thm}
\begin{proof}
    Let
    \[
        L_n=\langle \sigma_1 \e_1+ \dots+ \sigma_r \e_r\rangle^\bot\subseteq \Z^r
    \]
    be the description of $L_n$ as a changemaker lattice. The lattices $L_{n-\frac12}$ and $L_{n-1}$ can be identified (via a relabelling of basis vectors) with
    \[
    L_{n-\frac{1}{2}}=\langle \e_0-\e_1, \sigma_1 \e_1+ \dots+ \sigma_r \e_r\rangle^\bot\subseteq \langle \e_0, \dots, \e_r \rangle \subseteq \Z^{r+1}.
    \]
    and
    \[
    L_{n-1}=\langle \sigma_2 \e_2+ \dots+ \sigma_r \e_r\rangle^\bot\subseteq \langle \e_2, \dots, \e_r \rangle \subseteq \Z^{r-1},
    \]
    respectively.
    We start with an obtuse superbase $B$ for $L_n$ satisfying the conclusions of Lemma~\ref{lem:modified_basis}, that is, $B$ contains $\bv^{(2)}=-\e_2+\e_1$ and a vertex $\bm{w}$ with $w_1=w_2=-1$ and that for all other elements of $\bv\in B\setminus\{\bv^{(2)},\bm{w}\}$ we have $v_1=0$.
    
    Set $\widetilde{\bv}^{(2)}=\bv^{(2)}+\e_0= -\e_2+\e_1+\e_0$ and $\widetilde{\bm{w}}=\bm{w}-\e_0$. Let 
    \[B'=B\setminus\{\bv^{(2)},\bm{w}\}\cup \{\widetilde{\bv}^{(2)},\widetilde{\bm{w}}\}\]
    be the result of replacing $\bv^{(2)},\bm{w}$ with $\widetilde{\bv}^{(2)},\widetilde{\bm{w}}$. The elements of $B'$ span $L_{n-\frac{1}{2}}$, since $B$ spans $L_n$ and a vector $\bv$ is in $L_n$ if and only if $\bv+v_1 \e_0$ is in $L_{n-\frac{1}{2}}$.
    Furthermore, since 
    \[\widetilde{\bv}^{(2)}\cdot \widetilde{\bm{w}}=-1=\bv^{(2)}\cdot \bm{w}-1, \]
    and all other pairings are unchanged in the passage from $B$ to $B'$, we see that $B'$ is an obtuse suberbase for $L_{n-\frac12}$ such that graph $G_{B'}$ is obtained from $G_B$ by adding a single edge between the vertices corresponding to $\bv^{(2)}$ and $\bm{w}$.

    The proof of \cite[Proposition~4.2]{Mc17} shows that the set 
    \[
    B''=B\setminus\{\bv^{(2)},\bm{w}\} \cup \{\bv^{(2)}+\bm{w}\}=B'\setminus \{\widetilde{\bv}^{(2)},\widetilde{\bm{w}}\} \cup \{\widetilde{\bv}^{(2)}+\widetilde{\bm{w}}\}
    \]
    forms an obtuse superbase for the lattice $L_{n-1}$. The graph $G_{B''}$ is obtained by identifying the vertices $\bv^{(2)}$ and $\bm{w}$ in $G_B$. However notice that the vertex $\bv^{(2)}+\bm{w}\in B''$ satisfies  $v_2^{(2)}+w_2=-2$. Since $L_{n-1}$ is slack, Lemma~\ref{lem:irreducible_bounds} implies that $\bv^{(2)}+\bm{w}$ cannot be irreducible in $L_{n-1}$. Thus Lemma~\ref{lem:irreducible} implies that $\bv^{(2)}+\bm{w}$ is a cut vertex in $G_{B''}$. This implies that $\{\bv^{(2)},\bm{w}\}$ is a cut set of $G_B$ and that $L_{n-1}$ is decomposable.

     Finally suppose that $B$ is a planar obtuse superbase. This means that the graph $G_B$ can be embedded in the plane. Since $G_B$ is 2-connected and the vertices $\bv^{(2)}$ and $\bm{w}$ form a cut set in $G_B$, there is a simple closed curve encircling one of the connected components of $G_B\setminus\{\bv^{(2)}, \bm{w}\}$ in the plane whose intersection with $G_B$ consists of the vertices $\bv^{(2)}$ and $\bm{w}$. Thus by using an arc from this simple closed curve as an edge, one sees that the graph $G_{B'}$, which is obtained from $G_B$ by adding an edge between $\bv^{(2)}$ and $\bm{w}$, also admits an embedding in the plane. That is, $B'$ is also a planar obtuse superbase.
\end{proof}

\section{Changemaker foundations for genus bounds on alternating surgeries}
In this section we develop the necessary changemaker material to prove the genus bounds on alternating surgeries comprising Theorem~\ref{thm:genusbound}. The key technical result on changemaker lattices will be the following.
\begin{thm}\label{thm:2s_to_3_or_bigger}
Let $L$ be the $n$-changemaker lattice of the form
\[
L=\langle \sigma_1 \e_1+ \dots+ \sigma_r \e_r\rangle^\bot\subseteq \Z^r
\]
where $\sigma_r\geq 3$. If $L$ admits an obtuse superbase, then
\begin{equation}\label{eq:n_genus_bound}
n\leq 4 + \frac32\sum_{i=1}^r \sigma_i(\sigma_i-1).
\end{equation}
\end{thm}
As an application of Proposition~\ref{prop:4_string} we get the following restriction on changemaker lattices admitting an obtuse superbase.

\begin{prop}\label{prop:long_string_obs}
Suppose that $L$ is an integral changemaker lattice of the form
\[
L=\langle \sigma_1 \e_1+ \dots+ \sigma_r \e_r\rangle^\bot\subseteq \Z^r.
\]
Suppose that for some $\ell\geq 4$ and some $b>1$ there the changemaker coefficients satisfy
\[
\sigma_{b-1}<\sigma_{b}= \dots = \sigma_{b+\ell-1}.
\]
If $L$ admits an obtuse superbase, then for any $\sigma_k>\sigma_{b}$, we have $\sigma_{k}\geq (\ell-1) \sigma_{b}$.
\end{prop}
\begin{proof}
We will show that if there is $\sigma_k$ satisfying $\sigma_{b}<\sigma_{k}<(\ell-1) \sigma_{b}$, then $L$ does not admit an obtuse superbase. First suppose that $\sigma_b <\sigma_{k}\leq 2\sigma_b$. This implies that $\sigma_b>2\sigma_b - \sigma_{k}\geq 0$. By Proposition~\ref{prop:CM_condition}, we can find a subset $A\subseteq \{1, \dots, b-1\}$ such that
\[
\sum_{i \in A} \sigma_i = 2\sigma_b - \sigma_{k}.
\]
This allows us to take $\bm{z}=\e_{k}-\e_b -\e_{b+1} + \sum_{i\in A} \e_i\in L$. By Lemma~\ref{lem:irred_examples}, $\bm{z}$ is irreducible. However $\bm{z}$ satisfies $z_{b}=z_{b+1}=-1$, $z_{b+2}=z_{b+3}=0$ so $L$ cannot admit an obtuse superbase by Proposition~\ref{prop:4_string}.

Now suppose that $2\sigma_b \leq \sigma_{k}<(\ell-1)\sigma_b$. In this case we can write $\sigma_{k}= s\sigma_b + r$, where $0\leq r < \sigma_b$ and $2\leq s\leq \ell-2$. By Proposition~\ref{prop:CM_condition}, there is $A\subseteq \{1, \dots, b-1\}$ such that
\[
\sum_{i \in A} \sigma_i = r.
\]
This allows us to take $\bm{z}=-\e_{k}+ \e_b+ \dots + \e_{b+s-1}+\sum_{i\in A} \e_i\in L$, which is irreducible by Lemma~\ref{lem:irred_examples}. By Proposition~\ref{prop:4_string}, there can be no obtuse superbase since $z_{b+s-2}=z_{b-s-1}=1$ and $z_{b+s}=z_{b+s+1}=0$ (by hypothesis $b+s+1\leq b+\ell-1$, so we have $\sigma_{b+s-2}=\sigma_{b-s-1}=\sigma_{b+s}=\sigma_{b+s+1}$).
\end{proof}

The proof of Theorem~\ref{thm:2s_to_3_or_bigger} is broken down into several cases based on the stable coefficients of the lattice $L$. We will make frequent use of the following identity
\begin{equation}\label{eq:genus_bound_identity}
  n= \sum_{i=1}^r\sigma_i^2= \sum_{i=1}^r\frac32 \sigma_i(\sigma_i-1) + \sum_{i=1}^r\frac12 \sigma_i(3-\sigma_i).
\end{equation}
Firstly, we treat the case where all the stable coefficients are at least three.
\begin{lem}\label{lem:genus_bound_no_twos}
Let $L$ be an $n$-changemaker lattice of the form
\[
L=\langle \sigma_1 \e_1+ \dots+ \sigma_r \e_r\rangle^\bot\subseteq \Z^r
\]
where $\sigma_r\geq 3$ and $\sigma_i\neq 2$ for all $i$. If $L$ admits an obtuse superbase, then \eqref{eq:n_genus_bound} holds.
\end{lem}
\begin{proof}
    Let $m$ be minimal such that $\sigma_m\geq 2$. By hypothesis, we have in fact that $\sigma_k\geq 3$ for all $k\geq m$. We thus have the bound

\begin{align*}
n - \frac32\sum_{i=1}^r   \sigma_i(\sigma_i-1) &= \frac12\sum_{i=1}^r \sigma_i(3-\sigma_i)\\
& = m-1+\frac12\sum_{i=m}^r \sigma_i(3-\sigma_i)\\
&\leq \frac12 \sigma_m(3-\sigma_m) +\sigma_m+1\\
&=  \frac12 \sigma_m(5-\sigma_m)+1\\
&\leq 4,
\end{align*}

where we used that: $\sigma_i(3-\sigma_i)\leq 0$ for all $\sigma_i\geq 3$; that $m\leq \sigma_m+2$ by Proposition~\ref{prop:OSB_m_bound}; and that $\sigma_m(5-\sigma_m)\leq 6$ since $\sigma_m\geq 2$ is an integer.
\end{proof}

Next we address the case that the largest stable coefficient is at least four.
\begin{lem}\label{lem:genus_bound_large_coef}
Let $L$ be an $n$-changemaker lattice of the form
\[
L=\langle \sigma_1 \e_1+ \dots+ \sigma_r \e_r\rangle^\bot\subseteq \Z^r
\]
where $\sigma_r\geq 4$. If $L$ admits an obtuse superbase, then \eqref{eq:n_genus_bound} holds.
\end{lem}
\begin{proof}
    Let $m$ be minimal such that $\sigma_m\geq 2$. If $\sigma_m>2$, then Lemma~\ref{lem:genus_bound_no_twos} yields the desired conclusion. Thus we can assume that $\sigma_m=2$. By Proposition~\ref{prop:OSB_m_bound}, this implies that $m\leq 4$. Let $\ell\geq 1$ be the number of $\sigma_i$ which equal 2, that is, we have that $\sigma_m=\dots = \sigma_{m+\ell-1}=2$ and $\sigma_{m+\ell}\geq 3$.
     By Proposition~\ref{prop:long_string_obs}, we have either $\ell \leq 3$ or $\sigma_{m+\ell}\geq 2(\ell-1)$. In either case we have that
     \[\ell \leq \max\left\{3, 1+\frac{\sigma_{m+\ell}}{2}\right\}.\]
Since $\sigma_r\geq 4$ and $\sigma_r\geq \sigma_{m+\ell}$, this yields 
\[\ell\leq 1+ \frac{\sigma_r}{2}.\]
Thus we obtain the bound
    \begin{align*}
n - \frac32\sum_{i=1}^r   \sigma_i(\sigma_i-1) &= \frac12\sum_{i=1}^r \sigma_i(3-\sigma_i)\\
& = \ell + m-1 +\frac12\sum_{i=m+\ell}^r \sigma_i(3-\sigma_i)\\
&\leq 4+\frac{\sigma_r}{2}+ \frac12 \sigma_r(3-\sigma_r)\\
&\leq 4,
\end{align*}
where the final inequality comes from the observation that
\[\frac12 \sigma_r(3-\sigma_r) +\frac{\sigma_r}{2}=\frac{1}{2}\sigma_r(4-\sigma_r)\leq 0,\] since $\sigma_r\geq 4$. 
\end{proof}
By far the most delicate case is the one where the stable coefficients consist solely of twos and threes. 

\begin{lem}\label{lem:genus_bound_twos_and_threes}
Let $L$ be an $n$-changemaker lattice of the form
\[
L=\langle \e_1+ \dots+\e_{m-1} +2\e_m + \dots + 2\e_{m+\ell-1}+3\e_{m+\ell}+\dots + 3 \e_r\rangle^\bot\subseteq \Z^r
\]
for integers $m\geq 2$ and $\ell \geq 1$. If $L$ admits an obtuse superbase, then \eqref{eq:n_genus_bound} holds or the stable coefficients take the form $(2,2,2,3)$, $(2,2,2,3,3)$ or $(2,2,2,3,3,3)$. 
\end{lem}
\begin{proof}
For such a changemaker lattice, we have
\[n - \frac32\sum_{i=1}^r   \sigma_i(\sigma_i-1) = \frac12\sum_{i=1}^r \sigma_i(3-\sigma_i)=m+\ell -1.\]
Thus, we are required to show that if $L$ admits an obtuse superbase, then $m+\ell\leq 5$.

If $\ell \geq 4$, then Proposition~\ref{prop:long_string_obs} shows that $L$ does not admit an obtuse superbase. Thus we consider the possibilities $\ell=1,2,3$.

If $L$ admits an obtuse superbase, then Proposition~\ref{prop:OSB_m_bound} shows that $m\leq 4$. This automatically yields the required bound when $\ell=1$.

Next suppose that $\ell=2$. If $m=4$, then the lattice $L$ is very slack. Thus if $L$ admitted an obtuse superbase, then Theorem~\ref{thm:very_slack_technical} would imply that the changemaker lattice
\[L''=\langle \e_1+\e_{2} +2\e_3 + 2\e_{4}+3\e_{5}+\dots + 3 \e_{r-1}\rangle^\bot\subseteq \Z^{r-1}\]
is decomposable. However, one can easily show that such a lattice is indecomposable using the techniques of \cite[\S3.1]{GreeneLRP}.\footnote{If one considers the standard basis for $L''$, then the corresponding pairing graph is connected. Since the standard basis consists of irreducible vectors, this implies that $L''$ is indecomposable.} Therefore, if $\ell=2$ and $L$ admits an obtuse superbase, then $m\leq 3$ and the desired bound holds.

Finally, suppose that $\ell =3$, in this case $L$ takes the form
\[
L=\langle \e_1+ \dots+\e_{m-1} +2\e_m + 2\e_{m+1} + 2\e_{m+2}+3\e_{m+3}+\dots + 3 \e_r\rangle^\bot\subseteq \Z^r.
\]
The stable coefficients of this lattice take the form
\[
(2,2,2, \underbrace{3,\dots , 3}_{r-m-2})
\]
Suppose that there are at least four changemaker coefficients equal to three. This allows us to apply Proposition~\ref{prop:4_string} to the vector
\[
\bm{z}= \e_{m+3}+\e_{m+4}-\e_{m+2}-\e_{m+1}-\e_{m} \in L,
\]
which is irreducible by Lemma~\ref{lem:irred_examples}. Thus $L$ does not admit an obtuse superbase if there are at least four changemaker coefficients equal to three ($r\geq m+6$). 

Thus we have established \eqref{eq:n_genus_bound} except when the stable coefficients of $L$ take the form $(2,2,2,3)$, $(2,2,2,3,3)$ or $(2,2,2,3,3,3)$.
\end{proof}

It remains to consider the cases where $L$ has stable coefficients of the form $(2,2,2,3)$, $(2,2,2,3,3)$ or $(2,2,2,3,3,3)$. This is handled with the aid of a computer calculation.
\begin{lem}\label{lem:small_genus_bound_case}
    Let $L$ be an integer changemaker lattice with stable coefficients of the form $(2,2,2,3)$, $(2,2,2,3,3)$ or $(2,2,2,3,3,3)$, then $L$ does not admit an obtuse superbase.
\end{lem}
\begin{proof}
Let $L$ be a changemaker lattice of the form
\[
L=\langle \e_1+ \dots+\e_{m-1} +2\e_m + 2\e_{m+1} + 2\e_{m+2}+3\e_{m+3}+\dots + 3 \e_r\rangle^\bot\subseteq \Z^r,
\]
where $r=m+3, m+4$ or $m+5$. If $L$ admits an obtuse superbase, then $m\leq 4$ by Proposition~\ref{prop:OSB_m_bound}. This gives nine changemaker lattices that could potentially admit obtuse superbases. However, the algorithms described in \S\ref{sec:OSB_search} quickly show that none of these lattices actually admit an obtuse superbase. 
\end{proof}

Putting all these results together we get the main theorem of this section.
\begin{proof}[Proof of Theorem~\ref{thm:2s_to_3_or_bigger}]
Let $(\sigma_m,\dots, \sigma_r)$ be the stable coefficients for $L$, where we assume that $\sigma_r\geq 3$.
Lemma~\ref{lem:genus_bound_no_twos} yields the bound when $\sigma_m\geq 3$. Lemma~\ref{lem:genus_bound_large_coef} yields the bound when $\sigma_r\geq 4$. The remaining case is that $\sigma_m=2$ and $\sigma_r=3$, which puts the stable coefficients in the form
\[(\sigma_m,\dots, \sigma_r)=(2,\dots,2,3,\dots,3).\]
This is handled by Lemma~\ref{lem:genus_bound_twos_and_threes} and Lemma~\ref{lem:small_genus_bound_case}.
\end{proof}

\section{Algorithmically searching for obtuse superbases.}\label{sec:OSB_search}
Although Lemma~\ref{lem:universal_bound} is sufficient on its own to provide a naive algorithm that will find an obtuse superbase for a given changemaker lattice $L$ whenever it exists (list the finitely many vectors $\bv\in L$ with $\norm{\bv}\leq \operatorname{disc}(L)$ and then check each subset of appropriate size in turn to see if it forms an obtuse superbase), this algorithm is far too slow in reality to be practical. This section collects various results that allow us to develop a reasonably efficient algorithm for finding an obtuse superbase or, more importantly, proving that no obtuse superbase exists. The results of this section will be used purely for computational purposes and will not be used in any of the theoretical results in this paper. We concentrate on the case of integer and half-integer changemaker lattices.

\subsection{Integer lattices}
Throughout this subsection we will take $L$ to be an $n$-change\-ma\-ker lattice in the form
\[
L=\langle \sigma_1 \e_1 + \dots + \sigma_r \e_r \rangle^\bot \subseteq \Z^r,
\]
where $\sigma_r>1$ and $m>1$ is minimal such that $\sigma_m>1$.
We can exploit Lemma~\ref{lem:universal_bound} and the Cauchy-Schwarz inequality to get a coordinate-wise bound for obtuse superbase elements in a changemaker lattice. 
\begin{lem}\label{lem:Cauchy_Schwarz_bound}
Let $L$ be an $n$-changemaker lattice which admits an obtuse superbase $B$. Then for any element $\bm{z}\in B$ and any index $1\leq k \leq r$ we have
\[
|z_k| \leq \sqrt{n-\sigma_k^2}.
\]
\end{lem}
\begin{proof}
By Lemma~\ref{lem:universal_bound} and the fact that $n$ is the discriminant of $L$ \cite[Lemma~3.10]{GreeneLRP}, we have that $\norm{\bm{z}}\leq n$.
Now since $\bm{z}\in L$ we have that $\bsigma\cdot \bm{z} =0$, thus for any $k$ we have
\[
z_k \sigma_k =-\sum_{\substack{i=1\\ i\neq k}}^r \sigma_i z_i.
\]
However, by the Cauchy-Schwarz inequality we have that
\begin{align*}
\left(\sum_{\substack{i=1\\ i\neq k}}^r \sigma_i z_i\right)^2 &\leq \left(\sum_{\substack{i=1\\ i\neq k}}^r z_i^2\right) \left(\sum_{\substack{i=1\\ i\neq k}}^r \sigma_i^2\right)\\
&=(\norm{\bm{z}}-z_k^2)(n-\sigma_k^2)\\
&\leq (n-z_k^2)(n-\sigma_k^2).
\end{align*}
Thus we have that
\[
|z_k \sigma_k|^2\leq (n-z_k^2)(n-\sigma_k^2).
\]
Rearranging and taking square roots gives the desired bound.
\end{proof}

\begin{lem}\label{lem:irreducible_bounds2}
Let $\bm{z}\in L$ be an irreducible vector. Then the coordinates $z_k$ satisfy the following bounds.
\begin{enumerate}[label = (\roman*)]
    \item\label{it:sigk=1} If $\sigma_k=1$, then 
    \[|z_k|\leq 
    \begin{cases}
    2 &\text{if $L$ is tight}\\
    1 &\text{otherwise.}
    \end{cases}
    \]
    \item\label{it:sigknottight} If $\sigma_k>1$ is not tight and $\sigma_k=\sum_{i\in A} \sigma_i$ for $A \subseteq \{1, \dots, k-1\}$, then
    \[|z_k|\leq \sum_{i \in A}(|z_i| +1).
    \]
    \item\label{it:sigktight} If $\sigma_k>1$ is tight, then
    \[|z_k|\leq 5+ \sum_{i =1}^{k-1}(|z_i| +1).
    \]
\end{enumerate}
\end{lem}
\begin{proof}
The bound \ref{it:sigk=1} is a restatement of Lemma~\ref{lem:irreducible_bounds}.

Suppose that $\sigma_k>1$ is not tight and $\sigma_k=\sum_{i\in A} \sigma_i$ for $A \subseteq \{1, \dots, k-1\}$. Without loss of generality, suppose that $z_k\leq 0$. Set $\bm{x}=-\e_k + \sum_{i\in A}\e_i \in L$. Since $\bm{x}$ satisfies the bound in \ref{it:sigknottight}, we can assume that $\bm{x}\neq \bm{z}$. Using irreducibility gives
\begin{align*}
    -1&\geq \bm{x}\cdot(\bm{z}-\bm{x})\\
    &=-(z_k+1)+ \sum_{i \in A}(z_i -1).
\end{align*}
Rearranging gives
\[
z_k\geq \sum_{i \in A}(z_i -1)\geq -\sum_{i \in A}(|z_i| +1),
\]
which is the bound \ref{it:sigknottight}.

Finally, suppose that $\sigma_k>1$ is tight. Without loss of generality assume that $z_k\leq 0$ In this case take $\bm{x}$ to be the standard basis element $\bm{x}= -\e_k +\e_{k-1}+ \dots + \e_2 + 2\e_1$. Again we may assume that $\bm{x}\neq \bm{z}$. So by irreducibility, we have 
\begin{align*}
    -1&\geq \bm{x}\cdot(\bm{z}-\bm{x})\\
    &=-(z_k+1)+ \sum_{i=2}^{k-1}(z_i -1) + 2(z_1-2).
\end{align*}
Rearranging this inequality gives
\[
z_k \geq \sum_{i=1}^{k-1}(z_i -1) + z_1-3\geq -\sum_{i=1}^{k-1}(|z_i| +1) -5,
\]
where we used the bound $z_1\geq -2$, derived from \ref{it:sigk=1} in the last inequality. This establishes the bound \ref{it:sigktight}.
\end{proof}

We will also make use of a more elaborate bound on irreducible elements.
\begin{lem}\label{lem:sigma_m_bound}
If $\bm{z}\in L$ is irreducible, then $|z_m|\leq \sigma_m$.
\end{lem}
\begin{proof}
Without loss of generality let us suppose that $z_m>0$. Consider also the sets
\[
A_+=\{i\,|\, z_i\geq 1, 1\leq i \leq m-1\}
\]
and
\[
A_-=\{i\,|\, z_i\leq -1, 1\leq i \leq m-1\}.
\]
First we observe that at least one of $A_+$ or $A_-$ must be empty. If we had $a\in A_+$ and $b\in A_-$, then $\bm{x}=\e_a-\e_b\in L$ would be a vector $\bm{x}\neq \bm{z}$ satisfying 
\[
\bm{x}\cdot (\bm{z}-\bm{x})=z_a-z_b-2\geq 0.
\]
Since this contradicts the irreducibility of $\bm{z}$, we see that at least one of $A_+$ or $A_-$ must be empty.

Note that if $\sigma_m$ is tight, then $\sigma_m=m$ and if $\sigma_m$ is not tight, then $\sigma_m\leq m-1$. We prove the lemma by considering three separate cases:
\begin{enumerate}
    \item\label{it:A-non-empty} $A_-$ is non-empty and $\sigma_m$ is tight.
    \item\label{it:A+small} $\sigma_m$ is not tight and $|A_+|+\sigma_m\leq m-1$
    \item\label{it:A+large} $A_-$ is empty and $|A_+|+\sigma_m\geq m$.
\end{enumerate}
We note that these three cases encompass all possibilities. Together \eqref{it:A-non-empty} and \eqref{it:A+small} cover the case that $A_-$ is non-empty. The case that $A_-$ is empty is covered by \eqref{it:A+small} and \eqref{it:A+large}. In particular, note that if $\sigma_m$ is tight, then the condition $|A_+|+\sigma_m\geq m$ automatically holds. 

First, suppose that \eqref{it:A-non-empty} holds: we have $A_-$ is non-empty and $\sigma_m=m$. Without loss of generality we can assume that $1\in A_-$, i.e we have $z_1\leq -1$ and $z_i\leq 0$ for $i=2,\dots, m-1$. Consider $\bm{x}=\e_m-\e_{m-1}-\dots - \e_2 -2\e_1\in L$. Note that if $\bm{z}=\bm{x}$, then $z_m=1$. Thus we can assume that $\bm{x}\neq \bm{z}$.
The irreducibility of $\bm{z}$ gives 
\begin{align*}
    -1&\geq \bm{x}\cdot (\bm{z}-\bm{x}) \\
    &= z_m-1 - \sum_{i=2}^{m-1} (1+z_i) -2(z_1+2)\\
    & \geq z_m -1 -(m-2) -2\\
    &= z_m -1 -m.
\end{align*}
Rearranging gives $z_m\leq m=\sigma_m$, which is the desired bound if $A_-$ is nonempty and $\sigma_m$ is tight.

Now suppose that \eqref{it:A+small} holds: we have $\sigma_m\leq m-1$, and that $|A_+|+\sigma_m\leq m-1$. Note that this encompasses the case that $A_+$ is empty. In this case we consider $\bm{x}=\e_m-\e_1-\dots - \e_{\sigma_m}\in L$. The hypothesis that $|A_+|+\sigma_m\leq m-1$ implies that after relabelling the $\e_i$ we can assume that $A_+ \cap \{1, \dots, \sigma_m\}=\emptyset$, that is that $z_i\leq 0$ for $i=1, \dots, \sigma_m$. Again we can assume that $\bm{z}\neq \bm{x}$ since otherwise we would have $z_m=1$. Thus we find that
\begin{align*}
    -1&\geq \bm{x}\cdot (\bm{z}-\bm{x}) \\
    &= z_m-1 - \sum_{i=1}^{\sigma_m} (z_i+1)\\
    &\geq z_m -1 -\sigma_m.
\end{align*}
Rearranging this gives $z_m\leq \sigma_m$. Thus we have established the required bound when $|A_+|+\sigma_m\leq m-1$.

Finally suppose that \eqref{it:A+large} holds: we have that $A_-$ is empty and $|A_+|+\sigma_m\geq m$. Under these conditions we have that $z_i\geq 0$ for $i=1,\dots, m-1$ and, by assumption, that $z_m>0$. Since every non-zero vector in a changemaker lattice must have non-zero coordinates of both signs, there is some $i$ with $z_i<0$. Moreover we see that such an $i$ satisfies $i>m$. We take $g>m$ to be minimal such that $z_g\leq 0$.

By Proposition~\ref{prop:CM_condition}, there is a subset
$B\subset \{1, \dots, g-1 \}$ such that
\[
\sigma_g - \sigma_m= \sum_{i\in B}\sigma_i.
\]
Thus we can define an element $\bm{x}\in L$ of the form
\[
\bm{x}=-\e_g + \varepsilon \e_m + \sum_{i\in C} \e_i +\sum_{j\in D} \e_j,
\]
where $C=\{1,\dots, m-1\}\cap B$, $D= \{m+1, \dots, g-1\}\cap B$ and
\[
\varepsilon=\begin{cases}
    1 &\text{if $m\not\in B$}\\
    2 &\text{if $m\in B$.}
\end{cases}
\]
Furthermore, we assume that we have chosen $B$ so that $A_+\cap C$ is as large as possible. I.e so that $|C\setminus A_+|=\max\{0, |C|-|A_+| \}$. If $\bm{z}=\bm{x}$, then $z_m=\varepsilon\leq 2$, which is the desired bound. Thus we assume that $\bm{z}\neq \bm{x}$.
Since $\bm{z}$ is assumed to be irreducible, we have 
\begin{align*}
    -1&\geq \bm{x}\cdot (\bm{z}-\bm{x}) \\
    &=-z_g-1 + \varepsilon(z_m -\varepsilon) + \sum_{i\in C} (z_i-1) +\sum_{j\in D} (z_j-1)\\
    &\geq -1 + \varepsilon(z_m -\varepsilon) + \sum_{i\in C} (z_i-1)\\
    &\geq -1 + \varepsilon(z_m -\varepsilon) -|C\setminus A_+|.
\end{align*}
Thus rearranging we get the bound
\[
z_m\leq \frac{|C\setminus A_+|}{\varepsilon} +\varepsilon.
\]
Firstly, note that if $|C\setminus A_+|=0$, then we get the bound $z_m\leq \varepsilon\leq 2$. Thus we can assume that
$|C\setminus A_+|>0$. Since $|C|\leq m-1$ and $|A_+|+\sigma_m\geq m$, we have 
\[
|C\setminus A_+|=|C|-|A_+|\leq m-1-(m-\sigma_m)\leq \sigma_m -1.
\]
Since $\varepsilon\in \{1,2\}$, this gives the bound
\[
z_m\leq \max\left\{\sigma_m, \frac{\sigma_m+3}{2}\right\}.
\]
Since $\sigma_m\geq 2$ and $z_m$ is an integer, this is sufficient to show $z_m \leq \sigma_m$. This concludes the final case of the lemma.
\end{proof}

\begin{rem}
We note that the bound in Lemma~\ref{lem:sigma_m_bound} is sharp. For example, one finds that
$\bv=-\e_{m+1}+ n \e_m$ is irreducible in a changemaker lattice of the form
\[
L=\langle \e_1 + \dots + \e_{m-1} + n\e_m + n^2 \e_{m+1} \rangle^\bot\subseteq \Z^{m+1},
\]
where $m$ is assumed sufficiently large to ensure this a changemaker lattice.
\end{rem}

Putting these bounds together we describe a practical algorithm to decide whether or not an integer changemaker lattice admits an obtuse superbase.

\begin{alg}\label{alg:findOSB}
Given an integral changemaker lattice
\[
L= \langle \sigma_1 \e_1 + \dots + \sigma_r \e_r  \rangle^\bot \subseteq \Z^r,
\]
with discriminant $n=\sum_{i=1}^r \sigma_i^2$, the following procedure either finds an obtuse superbase for $L$ or certifies that none exists.

\begin{enumerate}
    \item \label{step:bound} 
    Produce a list $V_{\rm bound}$ of non-zero elements of $L$ whose coordinates satisfy the bounds of Lemmas~\ref{lem:Cauchy_Schwarz_bound},~\ref{lem:irreducible_bounds2}, and~\ref{lem:sigma_m_bound}.
    \item \label{step:norm2} Let $V_2$ be the set of all vectors in $L$ of the form $\e_i-\e_{i+1}$.
    \item \label{step:irred} 
    Take $V_{\rm irred}$ to be the subset of $V_{\rm bound}$ consisting of $\bv$ such that
    \begin{equation}\label{eq:cutting_reducibles}
        \text{$(\bv-\bm{w})\cdot \bm{w} < 0$ for all $\bm{w} \in V_{\rm bound}\setminus \{\bv\}$.}
    \end{equation}
    and
    \begin{equation}\label{eq:norm2_pairing}
    \text{$\bv\in V_2$ or $\bv\cdot \bm{w} \in \{0,-1\}$ for all $\bm{w} \in V_2$.}
    \end{equation}
    \item \label{step:subset_iteration} For each subset $\{\bv^{(1)}, \dots, \bv^{(r-1)}\}\subseteq V_{\rm irred}$ containing $V_2$ set 
    \[\bv^{(r)}=-(\bv^{(1)}+ \dots + \bv^{(r-1)})\]
    and verify whether the following conditions are satisfied:
        \begin{enumerate}
            \item $\bv^{(1)},\dots, \bv^{(r-1)}$ span $L$
            \item $\bv^{(i)}\cdot \bv^{(j)}\leq 0$ for all $1\leq i<j\leq r$.
        \end{enumerate}
    If so, then we have found an obtuse superbase for $L$.
    \item If no obtuse superbase is found in the previous step, then $L$ does not admit an obtuse superbase.
\end{enumerate}
\end{alg}
\begin{proof}
It is clear that a collection of vectors satisfying the two conditions of Step~\ref{step:subset_iteration} form an obtuse superbase for $L$, so it suffices to justify why the set will find an obtuse superbase whenever one exists.

Suppose that $L$ admits an obtuse superbase $B$. Lemma~\ref{lem:one_reducible} shows that $B$ contains at most one reducible vector. Furthermore, Lemma~\ref{lem:standard_norm_twos} allows us to assume that $B$ contains $V_2$. Since $L$ contains no elements of norm one all the elements of $V_2$, having length two, are necessarily irreducible. Thus we can assume that $B$ takes the form
\[B=\{\bv^{(1)}, \dots, \bv^{(r-1)}, -(\bv^{(1)}+\dots + \bv^{(r-1)})\},\]
where $\{\bv^{(1)}, \dots, \bv^{(r-1)}\}$ is a set of irreducible vectors containing $V_2$. Thus each of the $\bv^{(i)}$ satisfies the coordinate bounds of Lemmas~\ref{lem:Cauchy_Schwarz_bound},~\ref{lem:irreducible_bounds2}, and~\ref{lem:sigma_m_bound} and hence is contained in $V_{\rm bound}$. Since each $\bv^{(i)}$ is irreducible, it satisfies $(\bv^{(i)}-\bm{w})\cdot \bm{w}<0$ for all non-zero $\bm{w}\in L$ such that $\bm{w}\neq \bv^{(i)}$. Consequently each $\bv^{(i)}$ satisfies \eqref{eq:cutting_reducibles}. Consider a vector $\bm{w}\in V_2$. Since $V_2$ is contained in $B$ we have for all $i$ either that $\bm{w}=\bv^{(i)}$ or that $\bm{w}\cdot \bv^{(i)}\leq 0$. On the other hand each $\bv^{(i)}$ is irreducible, so if $\bv^{(i)}\neq \bm{w}$ for some $i$ we have that
\[-1\geq-(\bv^{(i)}+\bm{w})\cdot \bm{w}=-2-\bv^{(i)}\cdot \bm{w}.\]
From this we rearrange to see that $\bv^{(i)}\cdot \bm{w}\geq -1$ if $\bv^{(i)}\neq \bm{w}$. Altogether we see that $\bv^{(i)}$ also satisfies \eqref{eq:norm2_pairing}. Thus we see that $\bv^{(i)}$ is also contained in $V_{\rm irred}$. Thus $\{\bv^{(1)}, \dots, \bv^{(r-1)}\}$ is a subset of $V_{\rm irred}$ containing $V_2$ and hence will be found at Step~\ref{step:subset_iteration} as required.
\end{proof}

\begin{rem}\label{rem:quick_OSB_search}
    In order to rigorously verify that an obtuse superbase does not exist one needs to run a full version of Algorithm~\ref{alg:findOSB} using coordinate bounds such as those in Lemmas~\ref{lem:Cauchy_Schwarz_bound},~\ref{lem:irreducible_bounds2} and~\ref{lem:sigma_m_bound}. However, when an obtuse superbase exists it can usually be found without resorting to such an exhaustive search. In all known examples, the coordinates of the superbase elements always satisfy $|v_i|\leq 2$. So, in practice, a truncated version of Algorithm~\ref{alg:findOSB} which begins with $V_{\rm bound}$ comprising non-zero elements of $L$ such that $|v_i|\leq 2$ for all $i$ will find an obtuse superbase if one exists.
\end{rem}

\subsection{Half-integer lattices}\label{sec:half_int_basis_search} We now describe an algorithm for deciding if a half-integer changemaker lattice has an obtuse superbase. Throughout this subsection we will take $L$ to be an $(n+\frac{1}{2})$-changemaker lattice of the form
\[
L=\langle \e_{-1}-\e_0 , \e_0 + \sigma_1 \e_1 + \dots + \sigma_r \e_r \rangle^\perp \subseteq \Z^r
\]
where $\bsigma=\sigma_1 \e_1 + \dots + \sigma_r \e_r$ is a changemaker vector satisfying $\norm{\bsigma}=n$. Let $L'$ be the $n$-changemaker lattice
\[
L'=\langle \sigma_1 \e_1 + \dots + \sigma_r \e_r \rangle^\bot \subseteq \Z^{r+2}. 
\]
Note that $L'$ is the rank $r-1$ sublattice of $L$ such that $\bv\in L$ is contained in $L'$ if and only if $v_0=v_{-1}=0$.

Obtuse superbases for half-integer changemaker lattices were studied extensively in \cite{McCoyAltUnknotting}. Our first result is that all but two elements of an obtuse superbase must be contained in the sublattice $L'$.
\begin{prop}
    Suppose that $L$ admits an obtuse superbase $B$. Then $B$ contains precisely two elements $\bv\in B$ with $v_0 \neq 0$.    
    \end{prop}
\begin{proof}
    Note that there is at least one element of $B$ with $v_0\neq 0$ since $B$ spans $L$. In fact, since the elements of $B$ sum to zero, this shows that there must be at least two elements of $B$ with $v_0\neq 0$.
    The fact that there are at most two elements follows immediately from \cite[Lemma~5.3(i)]{McCoyAltUnknotting}, which states that if $\bm{z}$ is a sum of elements of $B$, then $|z_0|\leq 1$.
\end{proof}
Furthermore, the following result shows that one of these basis elements in $L'$ can be assumed to be of the form $\bv=-\e_1+\e_0+\e_{-1}$.
\begin{lem}[{\cite[Lemma 5.7]{McCoyAltUnknotting}}]\label{lem:half_int_basis_form}
    Let $m>1$ be minimal such that $\sigma_m>1$. Then if $L$ admits an obtuse superbase, then it admits an obtuse superbase containing the vectors
    \[
\pushQED{\qed} 
    \text{$-\e_1+\e_0+\e_{-1}$ and $-\e_k + \e_{k-1}$ for $2\leq k <m$.}
\qedhere
\popQED
\] 
\end{lem}

This yields the following algorithm to find an obtuse superbase.
\begin{alg}\label{alg:half_int_OSB}
Given an $(n+\frac12)$-changemaker lattice
\[
L=\langle \e_{-1}-\e_0 , \e_0 + \sigma_1 \e_1 + \dots + \sigma_r \e_r \rangle^\perp \subseteq \Z^{r+2},
\]
the following procedure either finds an obtuse superbase for $L$ or certifies that none exists.
    \begin{enumerate}
        \item Produce a list $V_{\mathrm{irred}}$ of irreducible vectors in $L'$, where
        $L'$ is the $n$-changemaker lattice
\[
L'=\langle \sigma_1 \e_1 + \dots + \sigma_r \e_r \rangle^\bot \subseteq \Z^r. 
\]
        \item For each subset $\{\bv^{(1)},\dots, \bv^{(r-1)}\}$ of size $r-1$ in $V_{\mathrm{irred}}$, check whether 
        \[
        \{\bv^{(0)}, \dots, \bv^{(r-1)}, -(\bv^{(0)} + \dots + \bv^{(r-1)}) \}
        \]
        forms an obtuse superbase for $L$, where $\bv^{(0)}$ is the vector $\bv^{(0)}=-\e_1+\e_0 + \e_{-1}$.
        \item If no obtuse superbase is found in the previous step, then $L$ does not admit an obtuse superbase.
    \end{enumerate}
\end{alg}
\begin{proof}
Suppose that $L$ admits an obtuse superbase. By Lemma~\ref{lem:half_int_basis_form} there is an obtuse super base $B$ containing $\bv^{(0)}$. Moreover, precisely $r-1$ of the remaining elements of $B$ will be contained in $L'$. A half-integer changemaker lattice is automatically indecomposable \cite[Lemma~2.10]{McCoyAltUnknotting}. Thus every element of $B$ is irreducible in $L$ (Lemma~\ref{lem:2-connected}). Consequently every element of $B\cap L'$ must be irreducible as a vector in $L'$. That is, the $r-1$ elements of $B\cap L'$ are all contained in $V_{\mathrm{irred}}$.
\end{proof}

Although Lemma~\ref{lem:irreducible_bounds2} and Lemma~\ref{lem:sigma_m_bound} can be used to bound the irreducible vectors in $L'$, the bounds provided are too weak to guarantee that Algorithm~\ref{alg:half_int_OSB} will obtain a result in reasonable time for many of the examples we wish to consider, namely, the Baker-Luecke knots discussed in \S\ref{sec:Baker_Luecke}. However, in these examples we are able to derive much stronger bounds via the following result.
\begin{lem}\label{lem:half_int_coef_bound}
    Suppose that $L'$ admits an obtuse superbase $B$ such that for all $k\in \{1, \dots, r\}$ we have
    \[
    \sum_{\bv\in B} |v_k|\leq 4.
    \]
    Then for any irreducible $\bm{z}\in L'$ and any $k\in \{1, \dots, r\}$ we have
    \[
    |z_k|\leq 2.
    \]
\end{lem}
\begin{proof}
    Since the $v_k$ are integers and $\sum_{\bv \in B}v_k=0$, the bound on $\sum_{\bv\in B} |v_k|$ implies that, as an unordered tuple, the non-zero values of $v_k$ for $\bv\in B$ take one of the following forms
    \begin{equation}\label{eq:coef_tuples}
    \text{$(1,-1)$, $(1,1,-1,-1)$, $(1,1,-2)$, $(2,-1,-1)$, or $(2,-2)$.}
    \end{equation}
    If $\bm{z}\in L'$ is irreducible, then by Lemma~\ref{lem:irreducible} there is a subset $R$ of $B$ such that $\bm{z}=\sum_{\bv\in R}\bv$. Thus, if the $k$th coordinate $z_k$ is non-zero then it is equal to a sum of some subset of one of the tuples appearing in \eqref{eq:coef_tuples}. However, the absolute value of any such sum is at most two, giving $|z_k|\leq 2$ as required. 
\end{proof}

\part{Alternating surgeries}\label{part:alternating}
In this section we convert the lattice theoretic results of Part~\ref{part:lattices} into concrete results about alternating surgeries. Throughout the section we will take the convention that an L-space knot is one admitting \emph{positive} L-space surgeries.
\section{Obstructing alternating surgeries}\label{sec:sharp}

\subsection{The stable coefficients}
Let $K$ be a non-trivial L-space knot of genus $g>0$ and let
\begin{equation*}
\Delta_K(x)=a_0+\sum_{i=1}^g a_i(x^i+x^{-i})
\end{equation*}
be its symmetrized Alexander polynomial normalized so that $\Delta_K(1)=1$. The \textbf{torsion coefficients} of $K$ are defined by
\begin{equation*}
t_i(K)=\sum_{j\geq 1} j a_{i+j}.
\end{equation*}
These are known to satisfy
\[
t_{i-1}\geq t_i\geq t_{i-1}-1
\]
for all $i$ and that $t_i=0$ if and only if $i\geq g(K)$ \cite{OS05b}.

\begin{defi}\label{def:compatibility}
Let $K$ be an L-space knot. A changemaker vector $\bsigma\in \Z^r$ is said to be {\bf compatible} with $\Delta_K(x)$ if 
   \begin{equation}\label{eq:sigma_ti_relation}
    8t_i(K)=\min_{\substack{ |c\cdot \bsigma|= \norm{\bsigma}-2i \\ c \in \operatorname{Char}(\Z^r)}}   \norm{c} -r 
\end{equation}
    for all $i$ in the range $0\leq i \leq \norm{\bsigma}/2$, where $\operatorname{Char}(\Z^r)$ denotes the characteristic vectors of $\Z^r$.
\end{defi}
\begin{rem}
    We have defined compatibility only for L-space knots, since this is the level of generality required for this present work. For a knot which is not necessarily an L-space knot, the appropriate definition would be obtained by replacing the torsion coefficients $t_i$ in Definition~\ref{def:compatibility} by the integers $V_i$ derived from the knot Floer complex by Ni-Wu \cite{NiWu_Cosmetic}.
\end{rem}
The results of \cite[\S2]{Mc17} show that if two changemaker vectors $\bsigma\in \Z^r$ and $\bsigma'\in \Z^{r'}$ are both compatible with the Alexander polynomial of an L-space knot $K$, then $\bsigma$ and $\bsigma'$ necessarily have same stable coefficients.
\begin{defi}\label{def:stable_coefs_knot}
    Let $K$ be an L-space knot. If there exists a changemaker vector $\bsigma$ compatible with the Alexander polynomial $\Delta_K(x)$, then we define the stable coefficients
    \[
    \underline{\rho}(K)=(\rho_1,\dots, \rho_\ell)
    \]
    of $K$ to be the stable coefficients of $\bsigma$. We take the convention that these are ordered so that\footnote{Note that this reverses the ordering used for the stable coefficients in Part~1.}
    \[
    \rho_1\geq \dots \geq \rho_\ell\geq 2.
    \]
    If there is no changemaker vector compatible with $\Delta_K(x)$, then we say that the stable coefficients of $K$ do not exist.
\end{defi}

\subsection{The Goeritz form}
Firstly, we define the Goeritz form and the white graph of an alternating diagram. Our conventions are unashamedly chosen to simplify the discussion in this paper; the reader may well find other conventions in use elsewhere in the literature. 
Let $D$ be a reduced, non-split, alternating diagram for a link $J$. This diagram divides the plane into connected regions which we may colour black and white in a checkerboard fashion. A choice of checkerboard colouring allows one to assign an incidence number, $\mu(c)\in \{\pm 1\}$, to each crossing $c$ of $D$ according to the convention shown in Figure~\ref{fig:incidencenumber}.
  \begin{figure}[htbp] 
   \centering
   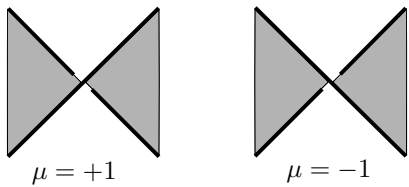
  \caption{The incidence number of a crossing.}
  \label{fig:incidencenumber}
 \end{figure}
 Since the diagram $D$ is alternating, there is a choice of checkerboard colouring which gives every crossing incidence number $\mu=1$. Fix this choice of colouring and let $R_0, \dots, R_{r}$ denote the white regions in the plane. The \textbf{white graph} $\Gamma_D$ is obtained by taking one vertex for each region $R_i$ and one edge between distinct regions $R_i$ and $R_j$ for every crossing between $R_i$ and $R_j$ in $D$. The \textbf{Goeritz lattice} of $D$ is define to be $\Lambda_D:=\Lambda(\Gamma_D)$, that is, the graph lattice associated to $\Gamma_D$.
 
 \begin{rem}\hfill\begin{enumerate}
    \item The assumption that $D$ be reduced implies that the graph $\Gamma_D$ contains no cut-edges. By Lemma~\ref{lem:2-connected}, this implies that the Goeritz lattice has no elements of norm one.
    \item Greene has shown that the homeomorphism type of the double branched cover $\Sigma_2(D)$ is determined by $\Lambda_D$. That is, two reduced alternating non-split diagrams $D$ and $D'$ have homeomorphic double branched covers if and only if $\Lambda_D$ and $\Lambda_{D'}$ are isomorphic \cite{Gr13}.
    \item The graph $\Gamma_D$ is always planar: one can construct a planar embedding by placing a vertex in each region $R_i$ and then connect them by edges passing through the crossings. Thus $\Lambda_D$ admits a planar obtuse superbase.
    \item Conversely, given a connected, 2-connected planar graph $G$ with no self-loops, there is always a reduced alternating diagram $D$ such that $\Lambda(G)\cong \Lambda_D$. One can construct such a diagram $D$ by embedding $G$ in the plane, placing a crossing of appropriate incidence on each edge and then connecting them cyclically around each vertex to obtain an alternating diagram.
 \end{enumerate}  
 \end{rem}
 
\subsection{The changemaker obstruction}
We now state our obstruction to alternating surgeries. This is a direct application of the changemaker theorem. The statement we use here is \cite[Theorem~2.7]{Mc17} which subsumes a number of earlier results \cite{GreeneLRP, Gr13, Gi15, Mc15}.
\begin{thm}\label{thm:alt_surgery_obstructions}
    Let $K$ be a non-trivial L-space knot for which $\Salt(K)$ is non-empty. Then there exist stable coefficients $\underline{\rho}(K)$ for $K$. Moreover, if $S_{p/q}^3(K)$ is an alternating surgery, then the $p/q$-changemaker lattice with stable coefficients $\underline{\rho}(K)$ is isomorphic to $\Lambda_D$ for any reduced alternating diagram such that $S_{p/q}^3(K)\cong \Sigma_2(D)$. In particular, the $p/q$-changemaker lattice with stable coefficients $\underline{\rho}(K)$ admits a planar obtuse superbase.
\end{thm}
\begin{proof}
    Consider $p/q\in \Salt(K)$. Since we are taking the convention that non-trivial L-space knots have positive L-space surgeries, this satisfies $p/q>0$. Let $D$ be a reduced alternating diagram for an alternating knot or link $J$ such that $S_{p/q}^3(K)\cong \Sigma_2(J)$. Ozsv\'ath and Szab\'o have shown there exists a sharp 4-manifold\footnote{Although we will not make any use of the definition here, we remind the reader, that a \textbf{sharp} 4-manifold is a smooth negative-definite compact manifold with connected boundary $Y$ such that for every $\spinct \in \spinc(Y)$ there exists $\spincs \in \spinc(X)$ restricting to $\spinct$ such that \[c_1(\spincs)^2+b_2(X)= 4d(Y,\spinct).\]} $X$ with $\partial X \cong \Sigma_2(J)$ and intersection form $Q_X\cong -\Gamma_D$ \cite[Theorem~3.4]{OS05}. On the other hand, the changemaker surgery obstruction \cite[Theorem~2.7]{Mc17} implies there is a changemaker vector $\bsigma$ compatible with $\Delta_K(x)$ and that $Q_X\cong -L\oplus -\Z^n$ for some $n\geq 0$ where $L$ is the $p/q$-changemaker lattice defined by $\bsigma$. By definition, the stable coefficients $\underline{\rho}(K)$ exist and are the stable coefficients $\bsigma$. Thus we have $\Gamma_D\cong L\oplus \Z^n$ for some $n$. However, the assumption that $D$ be reduced implies that $\Gamma_D$ has no elements of norm one and so we have $n=0$ in this case.
\end{proof}

This implies that the homeomorphism-type of alternating surgeries on a knot $K$ are determined by the Alexander polynomial.
\begin{cor}\label{cor:Alexpolydetermines}
Let $K$ and $K'$ be two L-space knots with $\Delta_K(x)=\Delta_{K'}(x)$. If $p/q$ is an alternating surgery slope for both $K$ and $K'$, then $S_{p/q}^3(K)\cong S_{p/q}^3(K')$.
\end{cor}
\begin{proof}
Suppose that $S_{p/q}^3(K)\cong \Sigma_2(J)$ and $S_{p/q}^3(K')\cong \Sigma_2(J')$, where $J$ and $J'$ are alternating links. Since $\Delta_K(x)=\Delta_{K'}(x)$, $K$ and $K'$ have the same stable coefficients and so Theorem~\ref{thm:alt_surgery_obstructions} implies that for any pair of reduced alternating diagrams $D$ and $D'$ for $J$ and $J'$ respectively, we have that $\Lambda_D\cong \Lambda_{D'}$. However this implies that $\Sigma_2(J)\cong \Sigma_2(J')$ \cite{Gr13}.
\end{proof}

\subsection{Computing stable coefficients}
In this section, we explain an algorithm (Algorithm~\ref{alg:stablecoeff}) that will compute the stable coefficients of an L-space knot whenever they exist.
The algorithm presented here is essentially a synthesis of the results contained in \cite[\S2]{Mc17}.  
Since Algorithm~\ref{alg:stablecoeff} returns at most one output, this recovers the proof that the stable coefficients are well-defined.

Given a non-trivial L-space knot $K$ with torsion coefficients $t_0, t_1,\dots $, we define $T_k(K)$ as the following count of torsion coefficients:
\[
T_k(K) = \# \{0\leq i <g(K) \, |\, 0<t_i(K) \leq k \}.
\]
for every integer $k$ in the range $0< k\leq t_0$. For convenience, we also define $T_0(K)=0$. Since $t_i=0$ if and only if $i\geq g(K)$, we have that $T_{t_0}(K)=g$.
Since the torsion coefficients of an L-space knot form a non-increasing sequence, we see that the sequence $(t_i)_{i\geq 0}$ can be recovered from the tuple $(T_1, \dots , T_{t_0})$. In particular, this implies also that $\Delta_K(x)$ can be calculated from $(T_1, \dots , T_{t_0})$.

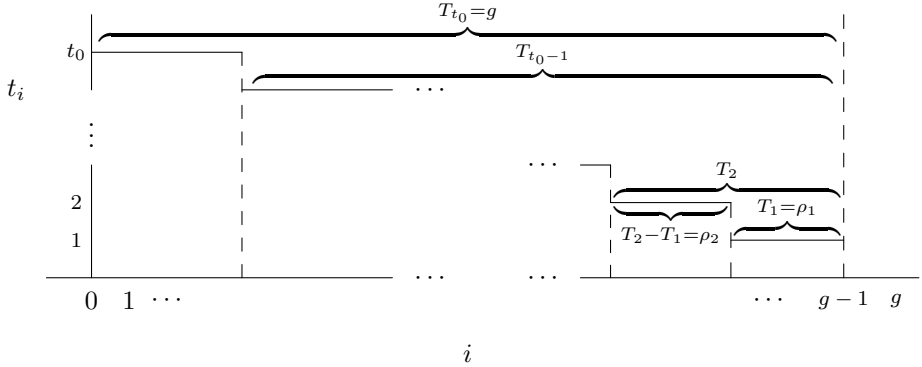
\begin{figure}[htbp] 
\begin{xy}
<1cm,1cm>*{};<1cm,2.5cm>*{}**@{-},  
<1cm,3.5cm>*{};<1cm,4.5cm>*{}**@{-},  
<0.4cm,1cm>*{};<5cm,1cm>*{}**@{-},
<5.5cm,1cm>*{\dotsb},             
<7cm,1cm>*{\dotsb},             
<7.5cm,1cm>*{};<12cm,1cm>*{}**@{-},  
<1cm,3cm>*{\vdots},             
<0cm,3.5cm>*{t_i},                
<6cm,0cm>*{i},                    
<1cm,4cm>*{};<3cm,4cm>*{}**@{-},    
<7.5cm,2.5cm>*{};<7.9cm,2.5cm>*{}**@{-},    
<3cm,1cm>*{};<3cm,4cm>*{}**@{--},   
<3cm,3.5cm>*{};<5cm,3.5cm>*{}**@{-},
<5.5cm,3.5cm>*{\dotsb},             
<7cm,2.5cm>*{\dotsb},             
<7.9cm,2cm>*{};<9.5cm,2cm>*{}**@{-},  
<7.9cm,2.5cm>*{};<7.9cm,1cm>*{}**@{--},   
<9.5cm,2cm>*{};<9.5cm,1cm>*{}**@{--},   
<9.5cm,1.5cm>*{};<11cm,1.5cm>*{}**@{-},  
<11cm,4.5cm>*{};<11cm,1cm>*{}**@{--},   
<1cm,0.7cm>*{0},        
<2cm,0.7cm>*{\dotsb},   
<10cm,0.7cm>*{\dotsb}, 
<1.5cm,0.7cm>*{1},
<11cm,0.7cm>*{\text{{\footnotesize $g-1$}}}, 
<11.7cm,0.7cm>*{\text{{\footnotesize $g$}}}, 
<0.8cm,4cm>*{\text{{\footnotesize $t_0$}}},        
<0.8cm,2cm>*{\text{{\footnotesize $2$}}},        
<0.8cm,1.5cm>*{\text{{\footnotesize $1$}}},        
<10.25cm,1.8cm>*{\overbrace{\text{\makebox[1.4cm]{}}}^{T_1=\rho_1}}, 
<9.45cm,2.3cm>*{\overbrace{\text{\makebox[3cm]{}}}^{T_2}}, 
<7cm,3.85cm>*{\overbrace{\text{\makebox[7.7cm]{}}}^{T_{t_0-1}}}, 
<6cm,4.4cm>*{\overbrace{\text{\makebox[9.8cm]{}}}^{T_{t_0}=g}}, 
<8.7cm,1.7cm>*{\underbrace{\text{\makebox[1.5cm]{}}}_{T_2-T_1=\rho_2}},
\end{xy}
\caption{A graph to show the relationship between the $t_i$ and the $T_k$. We have also shown how $\rho_0$ and $\rho_1$ occur as the number of $t_i$ equal to one and two, respectively (see Remark~\ref{rem:stable_hand_calc}).}
\label{fig:tvdiagram}
 \end{figure}

The $T_k$ are related to a changemaker vector via the following result.

\begin{lem}[{\cite[Lemma~2.10]{Mc17}}]
\label{lem:Ti_sigma_relationship}
Let $K$ be an L-space knot and let $\bsigma\in\Z^r$ be a changemaker vector compatible with $\Delta_K(x)$. For $1\leq k<t_0$ we have
\[
T_k = \max_{\bm{\alpha}\in S_{k,r}} \bsigma \cdot \bm{\alpha},
\]
where $S_{k,r}$ denotes the set
\[
S_{k,r} =\left\{ \bm{\alpha}= (\alpha_1, \dots, \alpha_r)\in \Z^r\, \middle| \, \alpha_i \geq 0, \, 2k= \sum_{i=1}^r \alpha_i(\alpha_i +1)\right\}.
\]
For $T_{t_0}$ we have that
\[
\text{$T_{t_0} = g(K) = \frac12\sum_{i=1}^{\ell} \rho_i(\rho_i -1)$ and $T_{t_0} \leq \max_{\bm{\alpha}\in S_{t_0,r}} \bsigma \cdot \bm{\alpha}$},
\]
where the stable coefficients of $\bsigma$ are $(\rho_1, \dots, \rho_\ell)$.
\qed
\end{lem}

\begin{rem}\label{rem:stable_hand_calc}
The stable coefficients $\rho_1$ and $\rho_2$ can be expressed directly in terms of $T_1$ and $T_2$. Since we are assuming that the $\rho_i$ are non-increasing, the maximum $\max_{\bm{\alpha}\in S_{k,r}} \bsigma \cdot\bm{\alpha}$ will be attained by an $\bm{\alpha} \in S_{k,r}$ for which the coefficients of $\bm{\alpha}$ are non-increasing. In particular for $k=1,2$, the maximal value will be attained by $\bm{\alpha}$ of the form
\[
\bm{\alpha}=\begin{cases}
    \{(1,0, \dots, 0) \} & \text{if $k=1$,}\\
\{(1,1,0, \dots, 0) \}& \text{if $k=2$.} 
\end{cases}
\]
Consequently, Lemma~\ref{lem:Ti_sigma_relationship} implies that
\[
\text{$\rho_1=T_1$ if $t_0>1$ and $\rho_2= T_2-T_1$ if $t_0>2$.}
\]
\end{rem}
\begin{ex}
    Let $K$ be the $(-2,3,7)$-pretzel knot. This has normalized Alexander polynomial
    \[
    \Delta_K(x)=t^5-t^4+t^2-t+1-t^{-1}+t^{-2}-t^{-4}+t^{-5},
    \]
    from which calculates torsion coefficients of the form
    \[
    \text{$t_4=1$, $t_3=1$, $t_2=1$, $t_1=2$ and $t_0=2$.}
    \]
    and torsion counts
    \[\text{$T_1=3$ and $T_2=5$}.\]
    It follows that the stable coefficient of $K$ are $\underline{\rho}(K)=(3,2,2)$. That $\rho_1=3$ follows the fact that $\rho_1=T_1$ (see Remark~\ref{rem:stable_hand_calc}). That the remaining coefficients are twos is forced by the requirement that $g = \frac12\sum_{i=1}^{\ell} \rho_i(\rho_i -1)$.
\end{ex}
Using the results of \cite{Mc17}, Lemma~\ref{lem:Ti_sigma_relationship} gives rise to the following algorithm for calculating stable coefficients.
\begin{alg}[Computation of Stable Coefficients]\label{alg:stablecoeff}
Let
\[\Delta_K(x) = a_0 + \sum_{i=1}^g a_i(x^i + x^{-i})\]
be the symmetrized Alexander polynomial of a knot of a non-trivial L-space knot $K$ normalized so that $\Delta_K(1)=1$. Then the following procedure terminates in finite time and returns compatible stable coefficients if they exist.
\begin{enumerate}
\item Compute the torsion coefficients $t_0, \dots, t_g$ and the torsion coefficient counts $T_0, T_1,\dots, T_{t_0}$, where $T_0=0$.
\item\label{step:recursive} Let $k=1$, $N=0$ and let $\underline{s}^{(0)}=()$ be the empty tuple.\\
 While $k\leq t_0$ and $T_{k}-T_{k-1}>2$ perform the following loop:
\begin{itemize}
\item Calculate 
\[
M=\begin{cases}
    \max_{\bm{\alpha}\in S_{k,N}}\bm{\alpha}\cdot \underline{s}^{(N)}. &\text{if $N>0$}\\
    0 &\text{if $N=0$}
\end{cases}
\]
\item If $M> T_k$, then terminate as no stable coefficients exist.
\item If $M <T_{k}$, then define $s_{N+1}:= T_{k}-T_{k-1}$, set
\[\underline{s}^{(N+1)}:=(s_1,\dots, s_N, s_{N+1})\]
and increment $N:=N+1$.
\item Increment $k:=k+1$ and loop.
\end{itemize}
\item\label{step:count2s} Calculate $d:= g - \frac12 \sum_{i=1}^{N} s_i(s_i-1)$.
\begin{itemize}
    \item If $d<0$, then terminate as no stable coefficients exist.
    \item If $d\geq 0$, then return
\[
\underline{\rho}=(s_1,\dots, s_{N}, \underbrace{2, \dots, 2}_{d}).
\]
as the stable coefficients.
\end{itemize}
\end{enumerate}
\end{alg}
\begin{proof}
Let
\[\bsigma=\rho_1 \e_1 +\dots + \rho_r \e_r\]
be a changemaker vector compatible with $\Delta_K(x)$ ordered so that the coefficients are non-increasing:
\[
\rho_1\geq \dots \geq \rho_r=1.
\]
Let $m\geq 0$ be maximal such that $\rho_m \geq 3$. The first task is to verify that the loop in Step~\ref{step:recursive} terminates with $N=m$ and
\[
\underline{s}^{(m)}=(\rho_1, \dots, \rho_m).
\]
If $t_0=1$, then \cite[Lemma 2.12]{Mc17} shows that $T_1=g\leq 3$ and that the stable coefficients of $\bsigma$ take the form
\[
\underline{\rho}=
\begin{cases}
(2) &\text{if $g=1$}\\
(2,2) &\text{if $g=2$}\\
(3) &\text{if $g=3$}.
\end{cases}
\]
Since $T_1-T_0=g$, we see that if $g\leq 2$, then the loop in Step~\ref{step:recursive} does not perform any steps and so terminates with $N=0$ and $\underline{s}^{(0)}=()$ as required. If $g=3$, then the loop performs one step and terminates with $N=1$ and $\underline{s}^{(1)}=(3)$, as required.

If $t_0>1$, then we may define the quantity 
\[\mu:=\min\{T_k -T_{k-1} \, | \, 1\leq k < t_0 \}.\]

If $\mu\leq 2$, then the proof of \cite[Lemma 2.16]{Mc17} shows that the loop in Step~\ref{step:recursive} will terminate with
\[\underline{s}^{(m)}=(\rho_1, \dots, \rho_m),\]
as required.

Thus it remains to consider the case that $\mu\geq 3$. In this case \cite[Lemma~2.14 and Remark~2.15]{Mc17} show that the stable coefficients of $\bsigma$ take the form
\[
\underline{\rho}= \begin{cases}
(3, \dots ,3) &\text{if $g\equiv 0 \bmod 3$}\\
(3, \dots ,3,2) &\text{if $g\equiv 1 \bmod 3$}\\
(3, \dots ,3,2,2) &\text{if $g\equiv 2 \bmod 3$},
\end{cases}
\]
where there are $\lfloor g/3 \rfloor$ threes in each case. Let $\bsigma\in \Z^r$ be a changemaker vector with stable coefficients of the form $\underline{\rho}$ given above. One can calculate that
\[
\max_{\bm{\alpha}\in S_{k,r}}\bsigma\cdot \bm{\alpha}= 3k
\]
for $1\leq k\leq \lfloor g/3 \rfloor$. The key step is to observe that the maximizing $\bm{\alpha}\in S_{k,r}$ must have all coefficients equal to zero or one. It follows from Lemma~\ref{lem:Ti_sigma_relationship} that that $T_k=3k$ for $0\leq k\leq \lfloor g/3 \rfloor$. If $g\equiv 1,2\bmod 3$, then a similar calculation shows that 
\[
\max_{\bm{\alpha}\in S_{k,r}}\bsigma\cdot \bm{\alpha}= 3k+2\geq g
\]
for $k=\lfloor g/3 \rfloor+1$. By Lemma~\ref{lem:Ti_sigma_relationship} this implies that $t_0=\lceil g/3\rceil$. 

Running the loop in Step~\ref{step:recursive} for such a sequence of $T_k$s one finds that we increment $N$ at each step and for each $0\leq N \leq \lfloor g/3 \rfloor$, we have
\[
\underline{s}^{(N)}=(\underbrace{3, \dots, 3}_{N}).
\]
Furthermore, the loop in Step~\ref{step:recursive} terminates with $N=\lfloor g/3 \rfloor$ and $k=N+1$. When $g\equiv 0\bmod 3$ this terminates because $k>t_0=\lfloor g/3 \rfloor$. When $g\equiv 1,2\bmod 3$ this terminates because $T_k-T_{k-1}\leq 2$. In either case, we see that Step~\ref{step:recursive} has calculated the required tuple.

Given that Step~\ref{step:recursive} has terminated having calculated the stable coefficients $\rho_i\geq 3$, all remaining unknown stable coefficients must have value $\rho_i=2$. The number of these coefficients are calculated in Step~\ref{step:count2s} using the relation from Lemma~\ref{lem:Ti_sigma_relationship} that
\[g = \frac12 \sum_{i=1}^\ell \rho_i(\rho_i-1).\]
Thus if a compatible changemaker vector exists, this algorithm computes the stable coefficients.
\end{proof}
\begin{rem}
    It is possible in theory that Algorithm~\ref{alg:stablecoeff} might return a potential set of coefficients, even when none exist. Thus, one should also verify that the putative stable coefficients do actually give rise to changemaker vectors that are compatible with the original Alexander polynomial.
\end{rem}

\section{The structure of \texorpdfstring{$\Salt(K)$}{Salt(K)}}
In this section, we prove Theorem~\ref{thm:salt_trichotomy}. However, first it is necessary to define and study $\D$, a class of knots admitting alternating surgeries constructed using the Montesinos trick.
\subsection{Knots in \texorpdfstring{$\D$}{D}}\label{sec:D_definition}
Let $(D,c)$ be a reduced alternating diagram $D$ of an alternating unknotting number one knot with a choice of an unknotting crossing $c$. Given such a pair we can construct a knot $K_{(D,c)}$ with alternating surgeries via the Montesinos trick \cite{Montesinos_trick}. Let $T$ be the tangle in $B^3$ obtained by taking the complement of a ball containing the crossing $c$. More specifically, the crossing arc associated to $(D,c)$ is the arc in the preimage of the double point of the crossing $c$ of the diagram $D$ that joins two points of the knot. The tangle $T$ is the complement of a ball neighborhood of the crossing arc that meets the knot in only two properly embedded arcs. Taking the double cover of $B^3$ branched along the tangle $T$ yields a 3-manifold $X_{(D,c)}$ with torus boundary. Since $c$ was an unknotting crossing, this manifold can be filled with a solid torus to yield $S^3$. Thus we can view $X_{(D,c)}$ as the complement of a knot $K_{(D,c)}$ in $S^3$. This knot admits a surgery to the double branched cover $\Sigma_2(D)$ and in fact, any replacement of $c$ by a rational tangle that yields an alternating diagram corresponds to an alternating surgery on $K_{(D,c)}$. If one calculates the slopes of these alternating surgeries carefully, one finds that these tangle replacements give an interval's worth of alternating surgeries \cite[Proposition ~5.3]{Mc15}
\begin{equation}\label{eq:intervalinSalt}
\left[ \left\lfloor \frac{\varepsilon \det D}{2} \right\rfloor, \left\lceil \frac{\varepsilon \det D}{2} \right\rceil \right]\cap \Q \subseteq \Salt(K_{(D,c)}),
\end{equation}
where the sign $\varepsilon = \pm 1$ depends on the sign of the unknotting crossing $c$ and the signature of the knot represented by $D$ via the following formula
\[
\varepsilon = \begin{cases}
(-1)^{\sigma(D)/2} &\text{if $c$ is positive}\\
-(-1)^{\sigma(D)/2} &\text{if $c$ is negative}.
\end{cases}
\]
\begin{defi}
We define $\D$ to be the set of L-space knots of the form $K_{(D,c)}$ for some alternating diagram $D$ and unknotting crossing $c$ contained therein.
\end{defi}

The class $\D$ plays a distinguished role on account of the following theorem which is a blend of \cite[Theorem~1.3]{Mc15} and \cite[Theorem~7.12]{Mc16}. Given a non-integer changemaker lattice with a planar obtuse superbase, one can always construct a corresponding knot in $\D$.
\begin{thm}\label{thm:stable_coefs_to_D}
    Let $L$ be a $p/q$-changemaker lattice for some positive $p/q\notin \Z$ with a tuple of stable coefficients $\underline{\rho}$. If $L$ admits a planar obtuse superbase, then there exists a knot $K\in \D$ with stable coefficients $\underline{\rho}(K)=\underline{\rho}$ and $[n,n+1]\cap \Q\subseteq \Salt(K)$, where $n=\lfloor p/q \rfloor$.
\end{thm}
\begin{proof}
Given a planar obtuse superbase $B$ for $L$, we may take an embedding of the graph $G_B$ in the plane. Using this embedding, we construct a reduced alternating diagram $D$ whose Goeritz lattice is isomorphic to the graph lattice $\Lambda(G_B)$ and hence to $L$. Since $p/q$ is not an integer, the implication (ii)$\Rightarrow$(iii) of \cite[Theorem~1.3]{Mc15} shows that $D$ is isotopic to another alternating diagram $D'$, which can be obtained by replacing an unknotting crossing $c$ in an alternating diagram $\widetilde{D}$ by a rational tangle. Moreover this rational tangle replacement is precisely the one showing that $K_{(\widetilde{D},c)}$ is a knot in $\D$ satisfying $S_{p/q}^3(K_{(\widetilde{D},c)})\cong \Sigma_2(D)$. The fact that the stable coefficients satisfy $\underline{\rho}(K)=\underline{\rho}$ follows from the argument in \cite[Proof of Theorem 7.12]{Mc16}.
\end{proof}
In particular, this immediately implies the following embellishment of \cite[Theorem~7.12]{Mc16}.
\begin{thm}\label{thm:nonint_alt_surg}
Let $K$ be a non-trivial 
L-space knot in $S^3$ such that $S_{p/q}^3(K)$ is an alternating surgery for some $p/q\not\in \Z$. Then there exists a knot $K'\in \D$ such that
\begin{enumerate}
    \item $S_{p/q}^3(K)\cong S_{p/q}^3(K')$,
    \item $\Delta_{K}(x)=\Delta_{K'}(x)$ and
    \item $[n, n+1] \cap \Q \subseteq \Salt(K')$, where $n=\lfloor p/q \rfloor$.
\end{enumerate}
\end{thm}
\begin{proof}
    Theorem~\ref{thm:alt_surgery_obstructions} implies that the stable coefficients $\underline{\rho}(K)$ exist and the $p/q$-changemaker lattice with these stable coefficients admits a planar obtuse superbase. Since $p/q\not\in \Z$, Theorem~\ref{thm:stable_coefs_to_D} implies the existence of a knot $K'\in \D$ with stable coefficients $\underline{\rho}(K)=\underline{\rho}(K')$ and $[n,n+1] \cap \Q \subseteq \Salt(K')$. We have $\Delta_{K}(x)=\Delta_{K'}(x)$, since the torsion coefficients, and hence the Alexander polynomial, are determined by the stable coefficients as in \eqref{eq:sigma_ti_relation}. 
\end{proof}

This implies that the only knots with infinitely many alternating surgeries in a bounded interval are those in $\D$. Corollary~\ref{cor:infte_alt_surg} extends this to show that the knots with infinitely many alternating surgeries are those in $\D$.
\begin{lem}\label{lem:infinite_alt_surg_in_interval}
Let $K$ be an L-space knot admitting alternating surgeries. If there is an integer $M$ such that $\Salt(K)\cap [M,M+1]$ is infinite, then 
\[
\text{$K\in\D$ and $[M,M+1]\cap \Q \subseteq \Salt(K)$.}
\]
\end{lem}
\begin{proof}
    By Theorem~\ref{thm:nonint_alt_surg}, there exists a knot $K'\in \D$ with the same Alexander polynomial as $K$ and $[M, M+1] \cap \Q \subseteq \Salt(K')$. By Corollary~\ref{cor:Alexpolydetermines}, we have $S_{p/q}^3(K)\cong S_{p/q}^3(K')$ for every $p/q\in [M,M+1]\cap \Salt(K)$. In particular, we can find $p/q$ with $|q|$ arbitrarily large such that $S_{p/q}^3(K)\cong S_{p/q}^3(K')$. This implies that $K$ and $K'$ are isotopic (see Theorem~\ref{thm:computable_q}).
\end{proof}

One can also prove the following proposition, which is a variant of~\cite[Proposition~4.2]{Mc17}, directly using the theory of changemaker lattices.
\begin{prop}\label{prop:altsurgimpliesintOSB2}
Let $L$ be a $p/q$-changemaker lattice for some positive $p/q\notin \Z$ which admits a planar obtuse superbase. Then for $n=\lfloor p/q \rfloor$ the $n$- and $(n+1)$-changemaker lattices with the same stable coefficients as $L$ also admit planar obtuse superbases.
\end{prop}
\begin{proof}
    By Theorem~\ref{thm:stable_coefs_to_D} there exists a knot $K\in\D$ with stable coefficients equal to the stable coefficients of $L$ such that $n$ and $n+1$ are alternating surgery slopes for $K$. By Theorem~\ref{thm:alt_surgery_obstructions}, this implies that the corresponding changemaker lattices admit planar obtuse superbases.
\end{proof}
\subsection{Analysis of \texorpdfstring{$\Salt(K)$}{Salt(K)}}
In this section we will take $K$ to be a non-trivial L-space knot admitting alternating surgeries. We will take $\underline{\rho}=(\rho_1,\dots, \rho_\ell)$ to be the stable coefficients of $K$ and we define $N$ to be the following integer:
\begin{equation}\label{eq:N_def}
    N=\sum_{i=1}^\ell \rho_i^2 + \max_{1\leq k \leq \ell} \left(\rho_k - \sum_{i=k+1}^{\ell} \rho_i \right).
\end{equation}
First we check that $\Salt(K)$ is necessarily a subset of $[N-1,N+1]$.
\begin{lem} \label{lem:salt_bounds}
    Any alternating surgery slope $p/q$ of $K$ satisfies
    \[
    N-1\leq p/q\leq N+1.
    \]    
\end{lem}
\begin{proof}
By Theorem~\ref{thm:alt_surgery_obstructions}, there is a $p/q$-changemaker lattice $L$ with stable coefficients given by $\underline{\rho}$ and $L$ admits a planar obtuse superbase. The lower bound comes from Proposition~\ref{prop:CM_existence_bound}, which, after adjusting for differences in notation, shows that the changemaker lattice $L$ exists if and only if $p/q\geq \norm{\bsigma} \geq N-1$.

If $p/q\not \in \Z$, then Proposition~\ref{prop:altsurgimpliesintOSB2} shows that the $\lceil p/q\rceil$-changemaker lattice with stable coefficients $\underline{\rho}$ also admits an obtuse superbase. If $p/q \in \Z$, then $p/q=\lceil p/q \rceil$ and the $\lceil p/q\rceil$-changemaker lattice with stable coefficients $\underline{\rho}$ is already known to admit an obtuse superbase. From Proposition~\ref{prop:OSB_m_bound} it follows that
    \[
    \lceil p/q\rceil \leq 1+\rho_\ell +\sum_{i=1}^\ell \rho_i^2,
    \]
    since $\rho_\ell$ corresponds to $\sigma_m$ in the notation of that section. Since
    \[
    \rho_\ell \leq \max_{1\leq k \leq \ell} \left(\rho_k - \sum_{i=k+1}^{\ell} \rho_i \right),
    \]
    this implies that $p/q \leq N+1$, as required.
\end{proof}

Together Lemma~\ref{lem:infinite_alt_surg_in_interval} and Lemma~\ref{lem:salt_bounds} show that the only knots with infinitely many alternating surgeries belong to $\D$.
\begin{cor}\label{cor:infte_alt_surg}
If $K$ is an L-space knot admitting infinitely many alternating surgeries, then $K\in\D$. \qed
\end{cor}

Using the results of \S\ref{sec:slack_lattices} we analyze the knots admitting alternating surgeries in the range $(N, N+1]$.

\begin{lem}\label{lem:pq>N_alt_surgery}
    Let $K$ be a non-trivial L-space knot admitting an alternating surgery. If $\Salt(K)\cap (N, N+1]$ is non-empty, then there exists a knot $K'\in\D$ such that $\Delta_K(x)=\Delta_{K'}(x)$ and $S_N^3(K')$ is a reducible surgery.
\end{lem}
\begin{proof}
By Remark~\ref{rem:CM_facts}\eqref{it:rational_existence}, for any rational $r\geq N-1$ there exists an $r$-changemaker lattice with stable coefficients $\underline{\rho}=\underline{\rho}(K)$. We use $L_r$ to denote this changemaker lattice. 
If there exists a slope $p/q$ in $\Salt(K)\cap (N, N+1]$, then $L_{N+1}$ admits a planar obtuse superbase by Proposition~\ref{prop:altsurgimpliesintOSB2}. However, as noted in Remark~\ref{rem:very_slack}, the lattice $L_{N+1}$ is very slack and so Theorem~\ref{thm:very_slack_technical} shows that $L_{N+\frac12}$ admits an obtuse superbase and that $L_N$ is decomposable lattice. Given the planar obtuse superbase for $L_{N+\frac12}$, Theorem~\ref{thm:stable_coefs_to_D} implies the existence of a knot $K'\in \D$ with stable coefficients $\underline{\rho}(K')=\underline{\rho}$ and $[N,N+1]\cap \Q \subseteq \Salt(K')$. Since the stable coefficients determine the Alexander polynomial, we have $\Delta_K(x)=\Delta_{K'}(x)$.

Now consider the alternating surgery $S_N^3(K')$. This is homeomorphic to the double branched cover of a diagram $D$, whose Goeritz lattice is isomorphic to $L_N$. Since $L_N$ is decomposable, Lemma~\ref{lem:2-connected} implies that the white graph of $D$ contains a cut-vertex. This implies that $D$ is a non-prime diagram and hence that its double branched cover is a reducible manifold. Thus $N$ is a reducible surgery slope for $K'$.
\end{proof}

\begin{lem}\label{lem:N_surgery_reducible}
    Let $K$ be a non-trivial knot in $\D$. If $N$ is a reducible surgery slope for $K$, then $\Salt(K)=[N-1,N+1]\cap \Q$ and $K$ can be written as a cable of the form $C_{r,s}\circ K'$, where $K'\in \D$, $rs=N$ and $r/s\in \Salt(K')$.
\end{lem}
\begin{proof}
Using the formula for the genus found in Lemma~\ref{lem:Ti_sigma_relationship}, we see that $N>2g(K)-1$. Thus, the results of Matignon-Sahari \cite{MS03} imply that $K$ must be a cable knot or a torus knot.

Therefore we may write $K$ as a cable $K=C_{r,s}\circ K'$, where $N=rs$ is the reducible surgery slope on $K$. Here $s\geq 2$ is the winding number of the pattern and we allow the possibility that $K'$ is unknotted. By \cite[Corollary~7.3]{Gordon1983Satellite}, we have a homeomorphism 
\begin{equation}\label{eq:cable_surgery}
S_{rs \pm \frac{1}{q}}^3(K) \cong S_{\frac{rsq \pm 1}{s^2q}}^3(K')
\end{equation}
for all $q\geq 1$. Let $M$ be the integer $M=\lfloor \frac{r}{s} \rfloor$. Since $s\geq 2$, we have
\begin{equation}\label{eq:Kprimeslopebound}
    M< \frac{r}{s}-\frac{1}{s^2} \leq  \frac{rsq \pm 1}{s^2q}\leq \frac{r}{s}+\frac{1}{s^2}<M+1,
\end{equation}
for all $q\geq 1$. As shown by \eqref{eq:intervalinSalt}, $\Salt(K)$ contains a unit interval with integer endpoints. Combined with Lemma~\ref{lem:salt_bounds}, this implies that at least one of the intervals $[N-1,N]\cap \Q$ or $[N,N+1]\cap \Q$ is contained in $\Salt(K)$. By \eqref{eq:cable_surgery} and \eqref{eq:Kprimeslopebound}, this implies that $K'$ has infinitely many alternating surgery slopes in the interval $[M,M+1]\cap \Q$. By Lemma~\ref{lem:infinite_alt_surg_in_interval}, this means that $K'\in \D$ and $\left[M,M+1 \right]\cap \Q \subseteq \Salt(K')$. On the other hand, \eqref{eq:cable_surgery}, \eqref{eq:Kprimeslopebound} and the fact that $\left[M,M+1 \right]\cap \Q \subseteq \Salt(K')$ implies that $rs\pm \frac{1}{q}$ is an alternating surgery slope for $K$ for all $q$. Thus $\Salt(K)$ contains infinitely many alternating surgeries in both the intervals $[rs-1,rs]\cap \Q$ and $[rs,rs+1]\cap \Q$. By Lemma~\ref{lem:infinite_alt_surg_in_interval} and Lemma~\ref{lem:salt_bounds} this implies that
\[
[rs-1,rs+1]\cap \Q\subseteq \Salt(K),
\]
as required.
\end{proof}

We now have all the technical ingredients to prove our main result on the structure of $\Salt(K)$. Note that in contrast to the rest of this section, the statement of Theorem~\ref{thm:salt_trichotomy} use the convention that the stable coefficients are non-decreasing.

\begin{repthm}{thm:salt_trichotomy}
  \salttrichotomy 
\end{repthm}
\begin{proof}
    By Lemma~\ref{lem:salt_bounds}, we have that $\Salt(K)\subseteq [N-1,N+1]\cap \Q$.
    
    First we consider the case that $K\in \D$. We wish to show that \eqref{it:KinDcable} or \eqref{it:KinDnoncable} holds for $K$. Suppose first that $K$ admits an alternating surgery slope $p/q>N$. By Lemma~\ref{lem:pq>N_alt_surgery}, this implies the existence of a knot $K'\in \D$ such that $\Delta_{K}(x)=\Delta_{K'}(x)$ and $S_N(K')$ is a reducible alternating surgery. By Lemma~\ref{lem:N_surgery_reducible}, $K'$ is a cable knot or a torus knot with $\Salt(K')= [N-1,N+1]\cap \Q$. Since $\Salt(K)\cap \Salt(K')$ is infinite and $\Delta_{K}(x)=\Delta_{K'}(x)$, Corollary~\ref{cor:Alexpolydetermines} and Theorem~\ref{thm:computable_q} imply that $K=K'$. Thus we see that \eqref{it:KinDcable} holds for $K$ if it admits an alternating surgery slope $p/q>N$.
    Thus we can consider the case that $K\in \D$ and $\Salt(K)\subseteq [N-1,N]\cap\Q$. As shown by \eqref{eq:intervalinSalt}, for any knot in $\D$ there is a unit interval contained in $\Salt(K)$. This implies that $\Salt(K)= [N-1,N]\cap\Q$. Lemma~\ref{lem:N_surgery_reducible} implies that $N$ is not a reducible surgery slopes for $K$. Thus \eqref{it:KinDnoncable} holds in this case.

    Now suppose that $K\not\in \D$. Recall that in this case $\Salt(K)$ is finite by Corollary~\ref{cor:infte_alt_surg}. 
    
    First we observe that if $\Salt(K)$ is not a subset of $\{N-1, N\}$, then there is a knot $K'$ in $\D$ with $\Delta_K(x)=\Delta_{K'}(x)$ and $\Salt(K)\subseteq \Salt(K')$. If $\Salt(K)$ contains a slope $p/q>N$, then Lemma~\ref{lem:pq>N_alt_surgery} provides the required knot $K'\in \D$ and Lemma~\ref{lem:N_surgery_reducible} implies that $\Salt(K')=[N-1,N+1]\cap \Q$. This leaves case that $\Salt(K)\subseteq [N-1, N]\cap \Q$. However if $\Salt(K)$ contains a non-integer slope in $[N-1, N]\cap \Q$, then Theorem~\ref{thm:nonint_alt_surg} yields a knot $K'\in\D$ with $\Delta_K(x)=\Delta_{K'}(x)$ and $[N-1, N]\cap \Q\subseteq \Salt(K')$.

    From the preceding paragraph, we see that if $\Delta_K(x)\neq \Delta_{K'}(x)$ for all $K'\in \D$, then $\Salt(K)\subseteq\{N-1,N\}$. That is, \eqref{it:KnotinDgeneral} holds.

    Finally, suppose that there exist $K'\in \D$ such that $\Delta_K(x)=\Delta_{K'}(x)$. Moreover, we may assume that $\Salt(K)\subseteq \Salt(K')$. If $\Salt(K)$ is not a subset of $[N-1,N]\cap \Q$, then this follows by taking $K'$ as constructed above. If $\Salt(K)$ is a subset of $[N-1,N]\cap \Q$, then this is simply the observation that $[N-1,N]\cap \Q\subseteq \Salt(K')$ for every $K'\in \D$.
    Since $K$ and $K'$ are distinct, Corollary~\ref{cor:Alexpolydetermines} and Theorem~\ref{thm:computable_q} furnish a computable upper bound on $q$ for all $p/q$ in $\Salt(K)$. Thus \eqref{it:KnotinDAlexpolyinD} holds in this case.
\end{proof}

This implies that knots in $\D$ are determined by their Alexander polynomials.
\begin{cor}\label{cor:D_determined_by_Alex_poly}
Let $K$ and $K'$ be knots in $\D$ admitting positive alternating surgeries. If $\Delta_K(x)=\Delta_{K'}(x)$, then $K$ and $K'$ are isotopic.
\end{cor}
\begin{proof}
Since the genus of a knot admitting alternating surgeries is determined by the degree of its Alexander polynomial, we may assume that both $K$ and $K'$ are non-trivial. If $K$ and $K'$ have the same Alexander polynomial, then for both knots the integer $N$ appearing in Theorem~\ref{thm:salt_trichotomy} is the same. In particular for this integer we have $[N-1,N]\cap \Q \subseteq\Salt(K)\cap \Salt(K')$. Consequently, Corollary~\ref{cor:Alexpolydetermines} and Theorem~\ref{thm:computable_q} together imply that $K$ and $K'$ are isotopic.
\end{proof}

\section{Computability of \texorpdfstring{$\Salt(K)$}{Salt(K)}}

In this section, we leverage Theorem~\ref{thm:salt_trichotomy} to prove that $\Salt(K)$ is algorithmically computable for any knot $K$. First we exclude the case of $0$-surgery.

\begin{lem}\label{lem:0_slope}
    The slope $0\in \Q$ is an alternating surgery slope for $K$ if and only if $K$ is the unknot.
\end{lem}
\begin{proof}
    Firstly, since $S_0^3(U)\cong S^1 \times S^2$ and $S^1\times S^2$ is the double branched cover of the 2-component unlink, we see that $0\in \Salt(U)$. Conversely, if $0\in \Salt(K)$ for a knot $K$, then $S_0^3(K)\cong \Sigma_2(J)$, where $J$ is an alternating link with $\det(J)=0$. However, an alternating link has determinant zero if and only if it is a split link \cite{lickorish}. Thus $\Sigma_2(J)$ contains an $S^1\times S^2$ summand and Gabai's proof of property R implies that $K$ is the unknot \cite{Ga87}.
\end{proof}

Next, we remind the reader that by the resolution of the Tait conjecture~\cite{Taite1,Taite2,Taite3} any reduced alternating diagram realizes the crossing number. 
Since for any non-split alternating link we have $\operatorname{cr}(J)\leq \det (J)$~\cite{det_alt}, it follows that for a given $d>0$, there are only finitely many alternating knots or links $J$ with $\det J=d$. In particular, one can find all alternating links with $\det J=d$ by enumerating all reduced alternating diagrams with at most $d$ crossings.

\begin{lem}\label{lem:pq_in_salt}
Given a knot $K$ and a fixed slope $p/q\in \Q$, there is an algorithm to decide whether $S_{p/q}^3(K)$ is an alternating surgery.  
\end{lem}
\begin{proof}
According to Lemma~\ref{lem:0_slope}, deciding if $0\in \Salt(K)$ is equivalent to checking whether or not $K$ is trivial.
For fixed $|p|>0$, first enumerate the finitely many alternating links $J_1, \dots, J_n$ with determinant $\det J_i=|p|$. Deciding whether $p/q\in \Salt(K)$ is then equivalent to deciding whether $S_{p/q}^3(K)\cong \Sigma_2(J_i)$ for some $i\in \{1,\dots, n\}$. Since the homeomorphism problem for 3-manifolds is decidable \cite{Kuperberg2019homeomorphism}, this shows that one can decide if $p/q\in \Salt(K)$ for any fixed $K$. 
\end{proof}

\begin{lem}\label{lem:exists_knot_in_D}
Given a symmetric polynomial $\Delta\in \Z[x^{\pm 1}]$ with $\Delta(1)=1$, there is an algorithm to construct a knot $K\in \D$ with $\Delta_K=\Delta$ or show that no such knot exists.
\end{lem}
\begin{proof} 
Let $g$ be the degree of $\Delta$, by which we mean the smallest $g\geq 0$ such that $\Delta(x)$ can be written in the form $\Delta(x)=a_0+\sum_{i=1}^g a_i (x^{i}+x^{-i})$ with $a_g\neq 0$. Let $(D,c)$ be a reduced alternating diagram $D$ with an unknotting crossing $c$ such that the corresponding knot $K=K_{(D,c)}$ is a knot in $\D$ with $\Delta_K=\Delta$. By Corollary~\ref{cor:D_determined_by_Alex_poly}, the knot $K$ is necessarily unique. Since any $K\in \D$ is an L-space knot, we have that $g(K)=g$.
 
It follows from the construction in \S\ref{sec:D_definition} that $S_{d/2}^3(K)\cong \Sigma_2(D)$ is an alternating surgery, where $d=\det D$. By \cite[Theorem~1.1]{Mc17}, this implies that $d/2 < 4g(K)+3$. Thus we see that $\operatorname{cr}(D)\leq \det D\leq 8g+5$. Thus given $\Delta$ of degree $g$, we can exhibit a knot $K\in \D$ with $\Delta_K=\Delta$ or verify its non-existence by iterating over the finitely many pairs $(D,c)$ where $D$ is a reduced alternating diagram with at most $8g+5$ crossings and $c$ is an unknotting crossing.
\end{proof}

\begin{repthm}{thm:computable}
   \thmcomputable
\end{repthm}
\begin{proof}
Firstly, one can decide whether or not $K$ is trivial. If $K$ is trivial, then $\Salt(K)=\Q$. Thus it suffices to provide an algorithm for a non-trivial knot $K$. By Lemma~\ref{lem:0_slope}, $0\notin \Salt(K)$ for any non-trivial $K$. Furthermore, for any $p/q\in \Q$, there is a homeomorphism $S_{-p/q}^3(K)\cong -S_{p/q}^3(mK)$. Therefore, it suffices to calculate in $\Salt(K)\cap \Q_{>0}$, since $\Salt(K)\cap \Q_{<0}$ is determined by $\Salt(mK)\cap \Q_{>0}$.

 We now describe an algorithm to calculate $\Salt(K)\cap \Q_{>0}$ for a non-trivial knot $K$. This algorithm is depicted in Figure~\ref{fig:algorithm_outline}. As a first step calculate the Alexander polynomial $\Delta_K$, symmetrized and normalized so that $\Delta_K(1)=1$ and $\Delta_K(x)=\Delta_K(x^{-1})$. Let $g$ be the degree of $\Delta_K(x)$. If $g=0$, then $\Salt(K)=\emptyset$, since the only L-space knot with $\Delta_K(x)=1$ is the unknot and we are assuming that $K$ is non-trivial. If $g\geq 1$, then we calculate the torsion coefficients $t_0, \dots, t_{g-1}$ of $\Delta_K(x)$. If $K$ is an L-space knot then the $t_i$ satisfy
 \begin{equation}\label{eq:Lspace_ti_conditions}
     \text{$t_{i+1}+ 1\geq t_i\geq t_{i+1}$ for $0\leq i\leq g-1$ and $t_{g-1}=1$}.
 \end{equation}
 Since alternating surgeries are L-space surgeries, we see that if the conditions of \eqref{eq:Lspace_ti_conditions} are not satisfied, then $\Salt(K)=\emptyset$. Thus we reduce to the case that the conditions \eqref{eq:Lspace_ti_conditions} are satisfied. Applying Algorithm~\ref{alg:stablecoeff} we can calculate the stable coefficients of $K$ if they exist. If Algorithm~\ref{alg:stablecoeff} does not return any stable coefficients, then $\Salt(K)=\emptyset$ as per Theorem~\ref{thm:alt_surgery_obstructions}. Thus, we may assume that Algorithm~\ref{alg:stablecoeff} returns a tuple of non-trivial stable coefficients for $K$ and consequently that we can calculate the integer $N$ of Theorem~\ref{thm:salt_trichotomy}. By Lemma~\ref{lem:exists_knot_in_D}, there is an algorithm to construct a knot $K'$ in $\D$ with $\Delta_K=\Delta_{K'}$ or show that no such knot exists. If no such knot exists, then Theorem~\ref{thm:salt_trichotomy}\eqref{it:KnotinDgeneral} shows that $\Salt(K)\subseteq\{N,N-1\}$. Lemma~\ref{lem:pq_in_salt} then shows that $\Salt(K)$, being a subset of an explicit finite set is computable. Thus suppose that there exists a knot $K'\in \D$ with $\Delta_K=\Delta_{K'}$. The next step is to check whether $K$ and $K'$ are isotopic. If so, we have $K\in \D$ and Theorem~\ref{thm:salt_trichotomy} yields two possibilities for $\Salt(K)$ which can be distinguished by checking whether $N+\frac12\in\Salt(K)$. Thus the final possibility is that $K$ and $K'$ are distinct. In this case, Theorem~\ref{thm:salt_trichotomy}\eqref{it:KnotinDAlexpolyinD} shows that there is a computable constant $Q(K, K')$ (defined via Theorem~\ref{thm:computable_q}) such that
 \[
 \Salt(K)\subseteq \{p/q \in \Q\cap [N-1,N+1] \mid q\leq Q(K,K')\}.
 \]
 As $\Salt(K)$ is contained in an explicit finite set, Lemma~\ref{lem:pq_in_salt} again implies that $\Salt(K)$ is computable.
\end{proof}

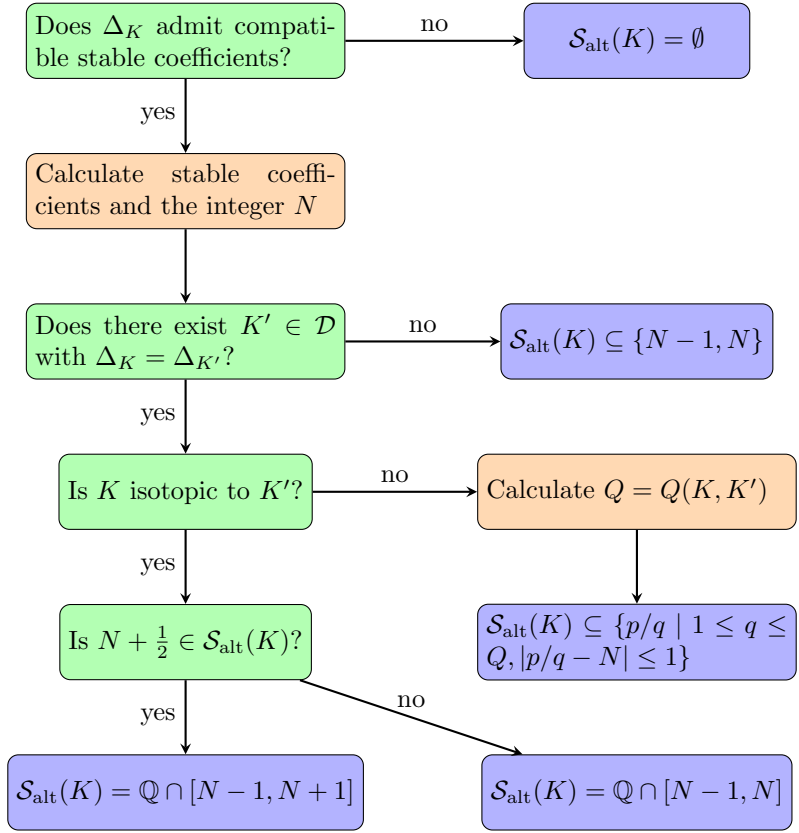
\begin{figure}
\begin{tikzpicture}[node distance=2cm]
\node (stable) [decision] {\parbox{4cm}{Does $\Delta_K$ admit compatible stable coefficients?}};
\node (returnnostable) [io, right of=stable, xshift=4cm] {$\Salt(K)=\emptyset$};
\node (calcstable) [process, below of=stable] {\parbox{4cm}{Calculate stable coefficients and the integer $N$}};
\node (polyinD) [decision, below of=calcstable] {\parbox{4cm}{Does there exist $K'\in\D$ with $\Delta_K=\Delta_{K'}$?}};
\node (polynotinD) [io, right of=polyinD, xshift=4cm] {\parbox{4.5cm}{$\Salt(K)\subseteq\{N-1,N\}$ computable by Lemma~\ref{lem:pq_in_salt}}};
\node (KinD) [decision, below of=polyinD] {Is $K$ isotopic to $K'$?};
\node (calcQ) [process, right of=KinD, xshift=4cm] {\parbox{4cm}{Calculate $Q=Q(K,K')$}};
\node (comparetoD) [io, below of=calcQ] {\parbox{5cm}{$\Salt(K)\subseteq\{p/q\mid 1\leq q\leq Q, |p/q-N|\leq 1\}$ computable by Lemma~\ref{lem:pq_in_salt}}};
\node (halfint) [decision, below of=KinD] {Is $N+\frac12\in \Salt(K)$?};
\node (returnKinD1) [io, below of=halfint] {$\Salt(K)=\Q\cap [N-1,N+1]$};
\node (returnKinD2) [io, right of=returnKinD1, xshift=4cm] {$\Salt(K)=\Q\cap [N-1,N]$};
\draw [arrow] (stable) --node[anchor=south] {no} (returnnostable);
\draw [arrow] (stable) --node[anchor=east] {yes} (calcstable);
\draw [arrow] (calcstable) -- (polyinD);
\draw [arrow] (polyinD) --node[anchor=south] {no} (polynotinD);
\draw [arrow] (polyinD) --node[anchor=east] {yes} (KinD);
\draw [arrow] (KinD) --node[anchor=south] {no} (calcQ);
\draw [arrow] (KinD) --node[anchor=east] {yes} (halfint);
\draw [arrow] (calcQ) -- (comparetoD);
\draw [arrow] (halfint) --node[anchor=east] {yes} (returnKinD1);
\draw [arrow] (halfint) --node[anchor=south] {no} (returnKinD2);
\end{tikzpicture}
\caption{The outline of an algorithm to calculate $\Salt(K)$ for a non-trivial knot $K$.}
\label{fig:algorithm_outline}
\end{figure}

\section{Genus bounds on \texorpdfstring{$\Salt(K)$}{Salt(K)}}\label{sec:genus_bounds}
In this section we prove the genus bound on alternating surgeries provided by Theorem~\ref{thm:genusbound}. 

\begin{lem}\label{lem:2strandpoly}
    Let $K$ be a knot in $S^3$ which admits a positive alternating surgery. If the symmetrized Alexander polynomial of $K$ takes the form
    \[
    \Delta_K(x)=x^g-x^{g-1}+x^{g-2} + \text{lower order terms},
    \]
    then the stable coefficients of $K$ take the form $\underline{\rho}(K)=(\underbrace{2,\dots,2}_{g})$.
\end{lem}
\begin{proof}
This follows from the calculation of $\rho_1$ in Remark~\ref{rem:stable_hand_calc} and the genus
\begin{equation}\label{eq:genus_formula}
    g(K)=\frac12 \sum_{i=1}^\ell \rho_i(\rho_i-1)
\end{equation}
from Lemma~\ref{lem:Ti_sigma_relationship}.
If $g=1$ or $g=2$, then according to \eqref{eq:genus_formula}, the only possibilities are $\underline{\rho}(K)=(2)$ or $\underline{\rho}(K)=(2,2)$, respectively. If $g\geq 3$, then we calculate that the torsion coefficients satisfy $t_{g}=0$, $t_{g-1}=t_{g-2}=1$ and $t_{g-3}=2$. By definition, this means $T_1=2$. Since $t_0\geq t_{g-3}>1$, we can apply Remark~\ref{rem:stable_hand_calc} to see that $\rho_1=2$. Since the stable coefficients are non-increasing it follows that $\underline{\rho}(K)=(\underbrace{2,\dots,2}_{g})$, as required.
\end{proof}

Notably the symmetrized Alexander polynomial of the knot $T_{2,2g+1}$ takes the form
\[
\Delta_{T_{2,2g+1}}(x)=x^g -x^{g-1} + x^{g-2}+ \dots + x^{-g},
\]
so the stable coefficients of $T_{2,2g+1}$ are of the form provided by Lemma~\ref{lem:2strandpoly}.

\begin{lem}\label{lem:T1=2}
Let $K$ be a knot which admits positive alternating surgeries. If the stable coefficients of $K$ satisfy $\underline{\rho}(K)=(\underbrace{2,\dots,2}_{g})$, then $K$ is the torus knot $T_{2,2g+1}$.
\end{lem}
\begin{proof}
Since the stable coefficients of a knot determine its Alexander polynomial, we see that $K$ and $T_{2,2g+1}$ must have the same Alexander polynomials. From Theorem~\ref{thm:salt_trichotomy} and by Corollary~\ref{cor:Alexpolydetermines} it follows that
\[
\Salt(K)\subseteq [4g+1,4g+3]\cap \Q = \Salt(T_{2,2g+1})
\]
and for any $p/q\in \Salt(K)$ we have $S_{p/q}^3(T_{2,2g+1})\cong S_{p/q}^3(K)$. Greene has shown that the reducible surgery slope $p/q=4g+2$ is a characterizing slope for $T_{2,2g+1}$ \cite{Greene2015genusbounds}. Rasmussen has shown that $p/q=4g+3$ is a characterizing slope for $T_{2,2g+1}$ \cite{Rasmussen2007lens}. Ni has shown that if $S_{4g+1}^3(T_{2,2g+1})\cong S_{4g+1}^3(K)$, then $K=T_{2,2g+1}$ or $\Delta_K=\Delta_{T_{5,4}}$ \cite{Ni2020Seifert}. In any case, these results imply that if $K$ has an integer alternating surgery, then $K=T_{2,2g+1}$.

Thus let us suppose that $p/q$ is a non-integer alternating surgery slope for $K$. For such a slope $S_{p/q}^3(T_{2,2g+1})\cong S_{p/q}^3(K)$ is an atoroidal Seifert fibered space. By Thurston's trichotomy for knots in $S^3$, $K$ must be either a hyperbolic knot, a satellite knot or a torus knot. Since $p/q>4g(K)$, the work of Ni immediately rules out the possibility that $K$ is a hyperbolic knot \cite[Theorem~1.1]{Ni2020Seifert}. If $K$ is a satellite knot, then $K$ must be a cable knot \cite[Corollary~1.1]{Miyazaki1997Seifert}. Suppose that $K$ is the $(a,b)$-cable of a non-trivial knot $K'$. Since $K$ is an L-space knot we must have $a,b>1$. Thus we have that $g(K)=\frac{(a-1)(b-1)}{2} + a g(K')$. Moreover, since the resulting space is atoroidal, the slope $p/q$ must take the form $p/q=ab\pm 1/q$. Since $K'$ is non-trivial we have $g(K')\geq 1$. Thus we have
\[
p/q\geq 4g(K)+1\geq 2(a-1)(b-1)+4a +1= ab+(b+2)(a-2)+7>ab+7.
\]
However this is a contradiction, since we can have an atoroidal surgery on $K$ only if $p/q$ takes the form $p/q=ab\pm 1/q\leq ab+1$.
Thus the only remaining case is that $K$ is a torus knot. However since $\Delta_{K}(X)=\Delta_{T_{2,2g+1}}(X)$, the only possibility is that $K=T_{2,2g+1}$, as required.
\end{proof}
Theorem~\ref{thm:Alexander_poly_form} is an easy consequence of Lemma~\ref{lem:T1=2}.
\begin{repthm}{thm:Alexander_poly_form}
    \thmpolyform
\end{repthm}
\begin{proof}
    Lemma~\ref{lem:2strandpoly} determines the stable coefficients of $K$ and Lemma~\ref{lem:T1=2} show that this implies that $K=T_{2,2g+1}$, as required.
\end{proof}

Similarly, we can easily obtain our genus bound on alternating surgeries.
\begin{repthm}{thm:genusbound}
    \genusbound
\end{repthm}
\begin{proof}
Suppose that $p/q>0$ is an alternating surgery slope for a non-trivial knot $K$ which is not a two-stranded torus knot. By Theorem~\ref{thm:alt_surgery_obstructions}, there exist stable coefficients for $K$ which we denote by $\underline{\rho}=(\rho_1,\dots, \rho_\ell)$. Furthermore, Proposition~\ref{prop:altsurgimpliesintOSB2} implies that the $n$-changemaker lattice with stable coefficients $\underline{\rho}$ admits an obtuse superbase, where $n=\lceil p/q \rceil$. Since $K$ is not a two-stranded torus knot, Lemma~\ref{lem:T1=2} implies that $\rho_1\geq 3$. By Theorem~\ref{thm:2s_to_3_or_bigger} and the genus formula from Lemma~\ref{lem:Ti_sigma_relationship}, this implies that the stable coefficients satisfy
\[p/q \leq n \leq 4+ \frac{3}{2}\sum_{i=1}^\ell \rho_i(\rho_i-1) = 4+3g(K),\]
which is the required bound.  
\end{proof}
We now exhibit two families of knots which show that the bound in Theorem~\ref{thm:genusbound} is sharp.
\begin{ex}
    For $g\geq 3$, the torus knot of the form $T_{3,g+1}$ satisfies
\[\text{$g(T_{3,g+1})=g$ and $\Salt(T_{3,g+1})=[3g+2,3g+4]\cap \Q$.}\]
One can verify that these have stable coefficients
\[
\underline{\rho}(T_{3,g+1})=\begin{cases}
    (3,\dots,3) &\text{if $g\equiv 0 \bmod{3}$}\\
    (3,\dots,3,2) &\text{if $g\equiv 1 \bmod{3}$}
\end{cases}.
\]
\end{ex}
Now we turn our attention to constructing a family of hyperbolic knots.
\begin{ex}\label{ex:hyp_family}
    For $n\geq 1$ the $(9n+\frac{19}{2})$-changemaker lattice
    \[
L=\langle \e_{-1}-\e_0, \e_0+\e_1+2\e_2+2\e_3+3\e_4+ \dots + 3\e_{n+3} \rangle^\bot\subseteq \Z^{n+5}
\]
    admits a planar obtuse superbase $B$. By embedding the corresponding graph $G_{B}$ in the plane, one can construct an alternating diagram $D_n$ whose Goeritz lattice is isomorphic to $L$. An example of such a diagram can be seen in Figure~\ref{fig:322stable_diag}. As per the proof of Theorem~\ref{thm:stable_coefs_to_D}, these diagrams allow us to construct a knot $K_n\in \D$ with stable coefficients 
    \[
    \underline{\rho}(K_n)=(\underbrace{3,\dots,3}_{n},2,2).
    \]
    The main result of \cite{Donald2019diagrams} implies that $K_n$ is a hyperbolic knot because $D_n$ is not a diagram of a two-bridge knot and does not contain substantial Conway sphere, as defined in that paper. Using the expressions for $g$ and $\Salt(K_n)$ in Lemma~\ref{lem:Ti_sigma_relationship} and Theorem~\ref{thm:salt_trichotomy}\eqref{it:KinDnoncable}, respectively, we see that
    \[
    \text{$g(K_n)=3n+2$ and $\Salt(K)=[3g+3,3g+4]\cap\Q$.}
    \]    
    For $n=1$, one obtains the $(-2,3,7)$-pretzel knot, which famously has two lens space surgeries (18 and 19-surgery) \cite{Fintushel1980lenssurgery}. In fact, for $n=1, \dots, 7$, the complement of the knot $K_n$ can be found in the SnapPy census. These are $m016$, $m082$, $m144$, $s086$, $v0165$, $t00324$ and $o9\_00644$, respectively.

    \end{ex}
\begin{figure}
     \centering
     \includegraphics[width=.4\textwidth]{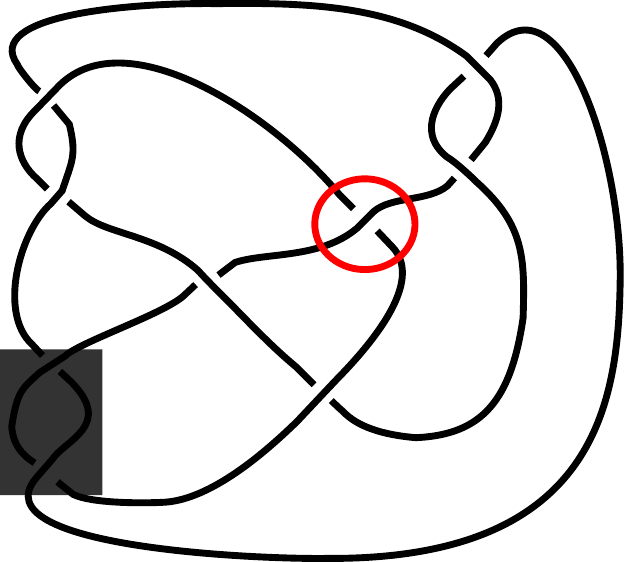} 
     \caption{The family of diagrams $\{D_n\}_{n\geq 1}$ defining the knots $K_n$ of Example~\ref{ex:hyp_family}. The chosen unknotting crossing is circled in red. The black block contains a twist of $n\geq 1$ crossings.}
     \label{fig:322stable_diag}
 \end{figure}


\part{Implementing the algorithms}

In this part of the paper we turn our attention to practical strategies to calculate $\Salt(K)$ for hyperbolic knots. We will use SnapPy~\cite{CDGW} together with sage~\cite{sagemath} and regina~\cite{BBP+} to implement some of the ideas from the general algorithm of Theorem~\ref{thm:computable} joined with some more direct ad-hoc ideas. 

\section{A practical strategy for computing \texorpdfstring{$\Salt$}{Salt}}
In practice, we were always able to determine $\Salt(K)$ by following the steps outlined in Figure~\ref{fig:practical_algorithm_outline}. For the steps pertaining to stable coefficients and the existence of obtuse suberbases, we used Algorithm \ref{alg:stablecoeff} and the methods described in \S\ref{sec:OSB_search}, respectively.

We used the following strategy to decide whether a given slope $r\in \Q$ is an alternating surgery slope for a given knot $K$. In practice, we only ever consider the case that $r$ is an integer or a half-integer. 

\begin{strat}\label{strat:single_alt}(Determining if a single slope is alternating)
    If $r$ is an alternating surgery slope for $K$, then $K$ possesses stable coefficients and the $r$-changemaker lattice $L_r$ with these stable coefficients admits a planar obtuse superbase. This obtuse superbase determines, via the Goeritz matrix, an alternating knot or link diagram $J$ such that $S_r^3(K)$ is an alternating surgery if and only if $S_r^3(K)$ is homeomorphic to the double branched cover of $J$ (Theorem~\ref{thm:alt_surgery_obstructions}). The existence of such a homeomorphism can be verified using SnapPy or regina. 
\end{strat}

This strategy allows us to complete the calculation of $\Salt(K)$ whenever it is known to be a subset of $\{N-1,N\}$.

\begin{rem}
    It turned out that in all examples of hyperbolic L-space knots, that we analyzed, the obstruction from Theorem~\ref{thm:alt_surgery_obstructions} was perfect. I.e.\ a slope $r$ of a hyperbolic L-space knot was alternating if and only if the corresponding changemaker lattice $L_r$ admitted a planar obtuse superbase. If this turns out to be true for general hyperbolic L-space knots then this would imply Conjectures~\ref{conj:atmosttwoslopes} and~\ref{conj:strong_conj} for hyperbolic knots. 
    
    On the other hand, there exist non-hyperbolic L-space knots where the obstruction from Theorem~\ref{thm:alt_surgery_obstructions} is not perfect. Indeed, the $(3,2)$-cable of $T_{3,2}$ admits no alternating surgery but has the same Alexander polynomial as $T_{3,4}$ and thus admits changemaker lattices with planar obtuse superbases. 
\end{rem}

A practical strategy to determine whether a knot is in $\D$ is as follows.

\begin{strat}\label{strat:knot_in_D}(Determining if a knot is in $\D$)
 If $K$ is in $\D$, then it admits compatible stable coefficients and the $(N-\frac12)$-changemaker lattice with these stable coefficients admits a planar obtuse superbase. As in Strategy~\ref{strat:single_alt}, this obtuse superbase allows one to construct an alternating knot diagram $J$ such that $S_{N-\frac12}^3(K)$ is homeomorphic to the double branched cover of $J$. Moreover $J$ has unknotting number one and if $K$ is in $\D$, then $K$ is of the form $K_{(J,c)}$ where $c$ is an unknotting crossing in $J$. Thus we can determine if $K$ is in $\D$ by using SnapPy to check whether the complement of $K$ is homeomorphic to the complement of $K_{(J,c)}$ for any unknotting crossing $c$ in $J$.
\end{strat}

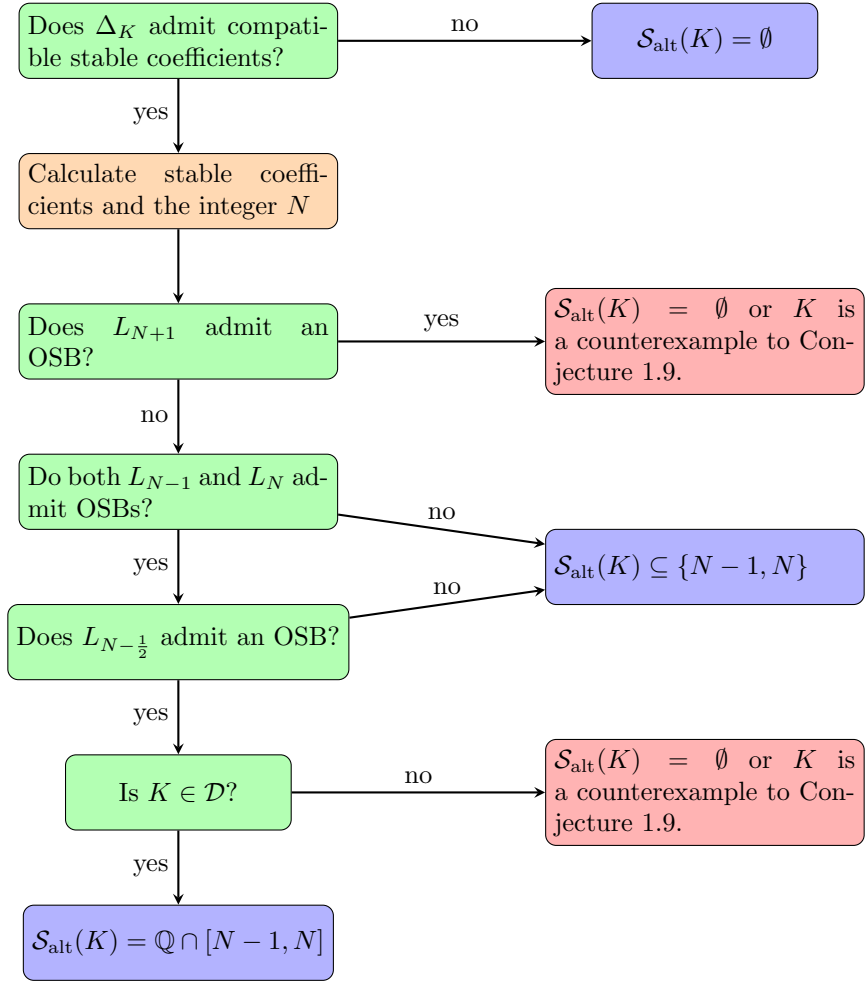
\begin{figure}
\begin{tikzpicture}[node distance=2cm]
\node (stable) [decision] {\parbox{4cm}{Does $\Delta_K$ admit compatible stable coefficients?}};
\node (returnnostable) [io, right of=stable, xshift=5cm] {$\Salt(K)=\emptyset$};
\node (calcstable) [process, below of=stable] {\parbox{4cm}{Calculate stable coefficients and the integer $N$}};
\node (LplusoneOSB) [decision, below of=calcstable] {\parbox{4cm}{Does $L_{N+1}$ admit an OSB?}};
\node (polynotinD) [unexpected, right of=LplusoneOSB, xshift=5cm] {\parbox{4cm}{$\Salt(K)=\emptyset$ or $K$ is a counterexample to Conjecture \ref{conj:strong_conj}.}};
\node (twointOSB) [decision, below of=LplusoneOSB] {\parbox{4cm}{Do both $L_{N-1}$ and $L_N$ admit OSBs?}};
\node (salttwoelements) [io, right of=twointOSB, xshift=5cm,yshift=-1cm] {\parbox{4cm}{$\Salt(K)\subseteq\{N-1,N\}$}};
\node (halfintosb) [decision, below of=twointOSB] {Does $L_{N-\frac12}$ admit an OSB?};
\node (checkKinD) [decision, below of=halfintosb] {Is $K\in \D$?};
\node (polynotinD2) [unexpected, right of=checkKinD, xshift=5cm] {\parbox{4cm}{$\Salt(K)=\emptyset$ or $K$ is a counterexample to Conjecture \ref{conj:strong_conj}.}};
\node (KinD) [io, below of=checkKinD] {$\Salt(K)=\Q\cap [N-1,N]$};
\draw [arrow] (stable) --node[anchor=south] {no}   
(returnnostable);
\draw [arrow] (stable) --node[anchor=east] {yes} (calcstable);
\draw [arrow] (calcstable) -- (LplusoneOSB);
\draw [arrow] (LplusoneOSB) --node[anchor=south] {yes} 
(polynotinD);
\draw [arrow] (LplusoneOSB) --node[anchor=east] {no}  (twointOSB);
\draw [arrow] (twointOSB) --node[anchor=south] {no} 
(salttwoelements);
\draw [arrow] (twointOSB) --node[anchor=east] {yes} (halfint);
\draw [arrow] (halfintosb) --node[anchor=east] {yes} (checkKinD);
\draw [arrow] (halfintosb) --node[anchor=south] {no} 
(salttwoelements);
\draw [arrow] (checkKinD) --node[anchor=east] {yes} 
(KinD);
\draw [arrow] (checkKinD) --node[anchor=south] {no} 
(polynotinD2);
\end{tikzpicture}
\caption{The outline of a practical strategy to calculate $\Salt(K)$ for a hyperbolic knot $K$. The justification of the output follows directly by combining Theorems~\ref{thm:salt_trichotomy} and~\ref{thm:alt_surgery_obstructions}.}
\label{fig:practical_algorithm_outline}
\end{figure}

In the following sections, we apply the above strategy to interesting families of hyperbolic L-space knots (the census knots and the Baker-Luecke knots) to determine their sets of alternating slopes. 

\section{Alternating surgeries on census knots} \label{sec:census}
In this section, we describe our procedures for classifying the alternating surgeries on the knots in the SnapPy census (i.e Theorems~\ref{thm:altsurgclassification} and~\ref{thm:bandings}). As well as calculating $\Salt(K)$ for all such knots in the census, we also explicitly determine branching sets for all their integer and half-integer alternating surgeries. If a knot admits a half-integer alternating surgery, we also create a tangle exterior which allows us to
determine the branching set for all other alternating surgeries by performing appropriate tangle fillings. This data and the code can be found at~\cite{BKM}.

\subsection*{Step 0: Restricting to L-space knots}
Since any knot which admits an alternating surgery is an L-space knot it suffices to consider L-space knots. From Dunfield's data~\cite{Du19}, we can extract a list of all L-space knot exteriors in the SnapPy census, cf.~\cite{ABG+19}. This gives a list of 632 census L-space knots which we denote by \verb|L_space_knots|. We include also two knots, $o9\_30150$ and $o9\_31440$ whose L-space status was not resolved by Dunfield in \verb|L_space_knots|, since these are now both known to be L-space knot exteriors \cite{quasi_alt_slopes}.

\subsection*{Step 1: Triage via stable coefficients and obtuse superbases}
For each knot $K$ in \verb|L_space_knots|, we first applied Algorithm~\ref{alg:stablecoeff} in order to calculate stable coefficients $\underline{\rho}(K)$ whenever they exist. This finds stable coefficients for 508 of the knots in \verb|L_space_knots| and shows that the remaining 124 knots do not admit any compatible stable coefficients and hence have no alternating surgeries.

For each knot $K$ in \verb|L_space_knots| with stable coefficients $\underline{\rho}=\underline{\rho}(K)$, let $N$ be the integer as in Theorem~\ref{thm:salt_trichotomy} and let $L_{N-1}$, $L_{N}$ and $L_{N+1}$ be the $(N-1)$-, $N$- and $(N+1)$-changemaker lattices with stable coefficients $\underline{\rho}$ respectively. Using Algorithm~\ref{alg:findOSB}, we check which of the three lattices $L_{N-1}$, $L_{N}$ and $L_{N+1}$ admit an obtuse superbase.

For the 508 knots under consideration, the lattice $L_{N+1}$ never admitted an obtuse superbase. For the lattices $L_{N-1}$ and $L_{N}$ there were three possible outcomes.
\begin{enumerate}
    \item For 115 knots, neither $L_{N-1}$ nor $L_N$ admitted an obtuse superbase. For these knots $\Salt(K)=\emptyset$.
    \item\label{it:triage_one_surgery} For 12 knots, $L_{N-1}$ is the only lattice which admits an obtuse superbase. For such knots, $K\not\in \D$ and $\Salt(K)$ contains at most one slope, namely $\Salt(K)\subseteq\{N-1\}$.
    \item\label{it:triage_in_D} For 381 knots, $L_{N-1}$ and $L_N$ both admit an obtuse superbase. For such knots, there are two possibilities, either $K\in \D$ and $\Salt(K)=[N-1,N]\cap \Q$ or $K\not\in \D$ and $\Salt(K)$ is a finite subset of $[N-1,N]\cap \Q$.
\end{enumerate}
For each of the unobstructed alternating integer surgeries, the obtuse superbases found in this search determine a Goeritz matrix and hence the crossing number for a potential alternating branching set. Using this data, we see that if $K$ is a knot in the SnapPy census with an integer alternating surgery, then that integer surgery is the branched double cover of an alternating knot or link with at most $13$ crossings.

\subsection*{Step 2: Classification of integer and half-integer alternating surgeries}\label{sec:search_for_alt}
The analysis of the preceding step implies that there are at most $774$ integer alternating surgery amongst the knots in the SnapPy census (1 for each of the 12 knots in \eqref{it:triage_one_surgery} and two for each of the 381 knots in \eqref{it:triage_in_D}). Similarly, there are at most $381$ half-integer alternating slopes on the knots in the SnapPy census.

It turns out that all of these integer and half-integer slopes are actually alternating surgery slopes. This was shown by finding an explicit alternating branching set for each potential alternating surgery. Note that in showing all of these slopes are alternating surgery slopes, we have shown that $\Salt(K)=\{N-1\}$ for all 12 knots in \eqref{it:triage_one_surgery}. The procedure for finding these branching sets was as follows.

First, we recompute these slopes in question with respect to the SnapPy basis (this will in general be different since for some knots in the SnapPy census the $(1,0)$-curve does not necessarily correspond to the meridian).

From Dunfield's data of exceptional fillings of the census knots~\cite{Du18} we can read-off that $59$ of the integer slopes are hyperbolic and the other $715$ integer slopes are exceptional. Amongst the half-integer slopes $298$ are hyperbolic and $83$ are exceptional. To handle the integer hyperbolic slopes, we iterated through all alternating links in the HT link table with crossing number at most $13$. For each of these alternating links we took the double branched cover and checked whether the resulting manifold was homeomorphic to one of the putative alternating surgeries via SnapPy. This quickly identified an explicit alternating branching set $J$ such that $S_{p}^3(K)\cong \Sigma_2(J)$ for all of the $59$ potential hyperbolic integer alternating surgery slopes.

A similar search, but iterating over all alternating knots with crossing number at most $14$ yields an alternating branching set for all of the $298$ potential hyperbolic half-integer alternating surgery slopes.

For the exceptional slopes, we begin by searching Dunfield's data of exceptional surgeries~\cite{Du18} for surgeries to small Seifert fibered spaces (including lens spaces) and check via SnapPy if the corresponding Montesinos link (whose double branched cover is the Seifert fibered space at hand) is alternating. We find that 178 of the integer surgeries are lens space surgeries and that 353 of the integer surgeries are other Seifert fibered spaces. In total, this identifies that $531$ integer slopes admit an alternating Seifert fibered surgery. We note that none of the half-integer exceptional slopes are lens space or Seifert fibered surgeries. This is in keeping with the cyclic surgery theorem and the conjecture that non-integer surgery on a hyperbolic knot is never a Seifert fibered manifold.

We are left with $184$ integer exceptional slopes and $83$ half-integer slopes which might be alternating. In each of these cases the resulting manifold is toroidal. For the integer slopes we run again through all links in the HT link table with crossing number at most $13$, take the double branched covers and check if Dunfield's recognition code~\cite{Du18} (using the combined power of SnapPy and regina) recognizes the double branched cover as the same manifold as a potential integer alternating surgery and if so search for a combinatorial equivalence through the Pachner graph. This identifies for any remaining possible alternating integer slope an explicit alternating branching set and thus verifies that all $774$ integer slopes of the census knots that admit obtuse superbases are actually alternating slopes. We find branching sets for the remaining potential half-integer surgery slopes by a similar procedure, but iterating through the alternating knots with at most 14 crossings.

\subsection*{Step 3: Verifying that knots are in \texorpdfstring{$\D$}{D}.}\label{sec:tangelExteriors}
To finish the proof of Theorem~\ref{thm:altsurgclassification} it remains to show that the 381 census knots with half-integer alternating surgeries are actually contained in the class $\D$. This is done by exhibiting for each knot $K$ an alternating diagram $D$ and an unknotting crossing $c$ such that the complement of $K_{(D,c)}$ is homeomorphic to the complement of $K$.

In practice, the pair $(D,c)$ was found as follows. From the preceding step we have a reduced alternating diagram $D$ such that $S_{N-\frac{1}{2}}^3(K)\cong \Sigma_2(D)$. The crossing $c$ is then taken to be the unknotting crossing such that its two resolutions are branching sets for the $(N-1)$- and $N$-surgeries on $K$.

In practice, to verify that $K$ and $K_{(D,c)}$ have homeomorphic complements as follows. Let $D'$ be the almost-alternating diagram of the unknot obtained by changing $c$ and let $c^*$ be the dealternating crossing in $D'$.  We convert $D'$ into a knot in $S^1\times D^2$ by deleting a small unknotted loop $l$ around the crossing $c^*$. In the double branched cover of $D'$, a Conway sphere containing $l$ and $c^*$ lifts to a torus $T$ which separates $\Sigma_2(D')$ into a copy of the complement of $K_{(D,c)}$ and a pair of pants times $S^1$. For each of the census knots, we are able to use SnapPy to construct this double branched cover and identify the splitting torus $T$ as a normal surface. Cutting along $T$ we obtain two connected components, one of which has a single boundary component. In all cases, SnapPy verifies that this piece is isometric to the original knot complement, thereby verifying that $K$ is in $\D$. This completes the calculation of $\Salt(K)$ for knots in the SnapPy census as presented in Theorem~\ref{thm:altsurgclassification}.

\subsection*{Step 4: Finding the bandings in Theorem~\ref{thm:bandings}}
Let $K$ be one of the 12 census knots such that $\Salt(K)=\{N-1\}$. In the course of proving Theorem~\ref{thm:altsurgclassification} we found an alternating diagram $J$ such that the double branched cover of $J$ yields $S_{N-1}^3(K)\cong \Sigma_2(J)$.

For each $K$ we were able to exhibit, either by hand or brute force computer calculation, an arc $\beta$ in $J$ such that banding along $\beta$ yields an almost alternating diagram of the unknot $D$, see Figure~\ref{fig:12bandings}. Furthermore, we were able to verify using SnapPy, that upon taking the double branched cover of $D$, the dual arc to $\beta$ became a knot whose complement is homeomorphic to the complement of $K$.

\begin{table}[htbp]
	\caption{The complete list of $393$ knots in the SnapPy census that admit alternating surgeries.}
	\label{tab:knots_with_alternating_surgeries}
{\tiny
	\begin{tabular}{ccccccccc}
\toprule	
$m016$&
 $m071$&
 $m082$&
 $m103$&
 $m118$&
 $m144$&
 $m194$&
 $m198$&
 $m239$\\
 $m240$&
 $m270$&
 $m276$&
 $m281$&
  $s042$&
 $s068$&
 $s086$&
 $s104$&
 $s114$\\
 $s294$&
 $s301$&
 $s308$&
 $s336$&
 $s344$&
 $s346$&
 $s367$&
 $s369$&
 $s407$\\
 $s582$&
 $s665$&
 $s684$&
 $s769$&
 $s800$&
 $v0082$&
 $v0114$&
 $v0165$&
 $v0220$\\
 $v0223$&
 $v0330$&
 $v0398$&
 $v0407$&
 $v0424$&
 $v0434$&
 $v0497$&
 $v0554$&
 $v0570$\\
 $v0573$&
 $v0707$&
 $v0709$&
 $v0715$&
 $v0740$&
 $v0741$&
 $v0759$&
 $v0765$&
 $v0847$\\
 $v0912$&
 $v0939$&
 $v0945$&
 $v1077$&
 $v1109$&
 $v1300$&
 $v1392$&
 $v1547$&
 $v1620$\\
 $v1628$&
 $v1690$&
 $v1709$&
 $v1716$&
 $v1718$&
 $v1728$&
 $v1810$&
 $v1832$&
 $v1839$\\
 $v1921$&
 $v1940$&
 $v1966$&
 $v1980$&
 $v1986$&
 $v2024$&
 $v2090$&
 $v2215$&
 $v2325$\\
 $v2759$&
 $v2930$&
 $v3354$&
 $t00110$&
 $t00146$&
 $t00324$&
 $t00423$&
 $t00434$&
 $t00729$\\
 $t00787$&
 $t00826$&
 $t00855$&
 $t00873$&
 $t00932$&
 $t01033$&
 $t01037$&
 $t01125$&
 $t01216$\\
 $t01268$&
 $t01292$&
 $t01318$&
 $t01368$&
 $t01409$&
 $t01422$&
 $t01424$&
 $t01440$&
 $t01598$\\
 $t01636$&
 $t01646$&
 $t01690$&
 $t01757$&
 $t01834$&
 $t01850$&
 $t01863$&
 $t01949$&
 $t02099$\\
 $t02104$&
 $t02238$&
 $t02378$&
 $t02398$&
 $t02404$&
 $t02470$&
 $t02537$&
 $t02567$&
 $t02639$\\
 $t03566$&
 $t03607$&
 $t03709$&
 $t03713$&
 $t03781$&
 $t03864$&
 $t03956$&
 $t03979$&
 $t04003$\\
 $t04019$&
 $t04102$&
 $t04180$&
 $t04228$&
 $t04244$&
 $t04382$&
 $t04721$&
 $t05118$&
 $t05239$\\
 $t05390$&
 $t05425$&
 $t05426$&
 $t05538$&
 $t05564$&
 $t05578$&
 $t05658$&
 $t05663$&
 $t05674$\\
 $t05695$&
 $t06001$&
 $t06440$&
 $t06463$&
 $t06525$&
 $t06570$&
 $t06605$&
 $t07348$&
 $t08111$\\
 $t08201$&
 $t08267$&
 $t08403$&
 $t09016$&
 $t09267$&
 $t09313$&
 $t09455$&
 $t09580$&
 $t09704$\\
 $t09852$&
 $t09954$&
 $t10188$&
 $t10230$&
 $t10462$&
 $t10643$&
 $t10681$&
 $t10985$&
 $t11556$\\
 $t11852$&
 $t12753$&
 $o9\_00133$&
 $o9\_00168$&
 $o9\_00644$&
 $o9\_00797$&
 $o9\_00815$&
 $o9\_01436$&
 $o9\_01496$\\
 $o9\_01584$&
 $o9\_01621$&
 $o9\_01680$&
 $o9\_01765$&
 $o9\_01953$&
 $o9\_01955$&
 $o9\_02255$&
 $o9\_02340$&
 $o9\_02350$\\
 $o9\_02386$&
 $o9\_02655$&
 $o9\_02696$&
 $o9\_02706$&
 $o9\_02735$&
 $o9\_02772$&
 $o9\_02786$&
 $o9\_02794$&
 $o9\_03032$\\
 $o9\_03108$&
 $o9\_03118$&
 $o9\_03133$&
 $o9\_03149$&
 $o9\_03162$&
 $o9\_03188$&
 $o9\_03288$&
 $o9\_03313$&
 $o9\_03412$\\
 $o9\_03526$&
 $o9\_03586$&
 $o9\_03622$&
 $o9\_03802$&
 $o9\_03833$&
 $o9\_03932$& 
 $o9\_04106$&
 $o9\_04205$&
 $o9\_04245$\\
 $o9\_04269$&
 $o9\_04313$&
 $o9\_04431$&
 $o9\_04435$&
 $o9\_04438$&
 $o9\_05021$&
 $o9\_05177$&
 $o9\_05229$&
 $o9\_05357$\\
 $o9\_05426$&
 $o9\_05483$&
 $o9\_05562$&
 $o9\_05618$&
 $o9\_05860$&
 $o9\_05970$&
 $o9\_06060$&
 $o9\_06128$&
 $o9\_06154$\\
 $o9\_06248$&
 $o9\_06301$&
 $o9\_07790$&
 $o9\_07893$&
 $o9\_07943$&
 $o9\_07945$&
 $o9\_08006$&
 $o9\_08042$&
 $o9\_08224$\\
 $o9\_08302$&
 $o9\_08477$&
 $o9\_08647$&
 $o9\_08765$&
 $o9\_08771$&
 $o9\_08776$&
 $o9\_08828$&
 $o9\_08831$&
 $o9\_08852$\\
 $o9\_08875$&
 $o9\_09213$&
 $o9\_09465$&
 $o9\_09808$&
 $o9\_10696$&
 $o9\_11248$&
 $o9\_11467$&
 $o9\_11560$&
 $o9\_11570$\\
 $o9\_11685$&
 $o9\_11795$&
 $o9\_11845$&
 $o9\_11999$&
 $o9\_12144$&
 $o9\_12230$&
 $o9\_12412$&
 $o9\_12459$&
 $o9\_12519$\\
 $o9\_12693$&
 $o9\_12736$&
 $o9\_12757$&
 $o9\_12873$&
 $o9\_12892$&
 $o9\_12919$&
 $o9\_12971$&
 $o9\_13052$&
 $o9\_13056$\\
 $o9\_13125$&
 $o9\_13182$&
 $o9\_13188$&
 $o9\_13400$&
 $o9\_13403$&
 $o9\_13433$&
 $o9\_13508$&
 $o9\_13537$&
 $o9\_13604$\\
 $o9\_13639$&
 $o9\_13649$&
 $o9\_13666$&
 $o9\_13720$&
 $o9\_13952$&
 $o9\_14018$&
 $o9\_14079$&
 $o9\_14136$&
 $o9\_14364$\\
 $o9\_14376$&
 $o9\_14495$&
 $o9\_14599$&
 $o9\_14716$&
 $o9\_14831$&
 $o9\_14974$&
 $o9\_15506$&
 $o9\_15633$&
 $o9\_15997$\\
 $o9\_16065$&
 $o9\_16141$&
 $o9\_16157$&
 $o9\_16181$&
 $o9\_16319$&
 $o9\_16356$&
 $o9\_16527$&
 $o9\_16642$&
 $o9\_16748$\\
 $o9\_16920$&
 $o9\_17450$&
 $o9\_18007$&
 $o9\_18209$&
 $o9\_18633$&
 $o9\_18813$&
 $o9\_19130$&
 $o9\_20219$&
 $o9\_21893$\\
 $o9\_21918$&
 $o9\_22129$&
 $o9\_22477$&
 $o9\_22607$& 
 $o9\_22663$&
 $o9\_22698$&
 $o9\_22925$&
 $o9\_23023$&
 $o9\_23263$\\
 $o9\_23660$&
 $o9\_23955$&
 $o9\_23961$&
 $o9\_23977$&
 $o9\_24149$&
 $o9\_24183$&
 $o9\_24534$&
 $o9\_24592$&
 $o9\_24886$\\
 $o9\_24889$&
 $o9\_25595$&
 $o9\_26604$&
 $o9\_26791$&
 $o9\_27155$&
 $o9\_27261$&
 $o9\_27392$&
 $o9\_27480$&
 $o9\_27737$\\
 $o9\_28113$&
 $o9\_28153$&
 $o9\_28529$&
 $o9\_28592$&
 $o9\_28746$&
 $o9\_28810$&
 $o9\_29246$&
 $o9\_29436$&
 $o9\_29529$\\
 $o9\_30375$&
 $o9\_30721$&
 $o9\_30790$&
 $o9\_31165$&
 $o9\_32132$&
 $o9\_32257$&
 $o9\_32588$&
 $o9\_33526$&
 $o9\_33585$\\
 $o9\_34403$&
 $o9\_35320$&
 $o9\_35549$&
 $o9\_35682$&
 $o9\_35736$&
 $o9\_35772$&
 $o9\_37754$&
 $o9\_37941$&
 $o9\_39394$\\
 $o9\_39451$&
 $o9\_40179$&
 $o9\_43001$&
 $o9\_43679$&
 $o9\_43953$&
 $o9\_44054$&&&\\
 \bottomrule
	\end{tabular}
}
\end{table}

\section{Baker--Luecke asymmetric L-space knots}\label{sec:Baker_Luecke}
Finally we examine alternating surgeries on some small Baker-Luecke knots. These are natural candidates to study since they have $\Salt(K)$ finite and non-empty. The subfamily of knots $K_{(m,b1,a1,a2,a3)}$ defined in~\cite[\S 11.4]{BL17} for non-negative integers $a_3, a_2, a_1, m, b_1$ gives a tractable collection of ``small'' examples. We applied the approach from \S\ref{sec:census} to the $14$ simplest of these knots. Surprisingly, our methods also revealed the existence of a second alternating surgery on all of these $14$ knots. 
\begin{thm} \label{thm:BLaltsurgs}
Let $K$ be one of the 14 Baker-Luecke knots listed in Table~\ref{tab:BL_Knots_alt2}. Then the following are true:
\begin{enumerate}
    \item $K$ is an asymmetric L-space knot
    \item  $\Delta_K(x)\neq \Delta_{K'}(x)$ for all $K'\in \D$
    \item $\Salt(K)=\{N-1,N\}$.
\end{enumerate}
\end{thm}
\begin{proof}
Let $K$ be one of the 14 Baker-Luecke knots in question. One loads the surgery description of $K$ given in~\cite{BL17} into SnapPy. The verified functions in SnapPy show that $K$ is a hyperbolic knot with trivial symmetry group. That $K$ is an L-space knot follows directly from~\cite{BL17} where it is shown to have at least one alternating surgery.
    
Next, we calculate the Alexander polynomial $\Delta_K(x)$ and apply Algorithm~\ref{alg:stablecoeff} to compute its stable coefficients $\underline{\rho}(K)$ and the associated integer $N$. If there were a knot $K'\in \D$ with the same Alexander polynomial as $K$, then $K'$ would have the same stable coefficients as $K$ and, according to Theorem~\ref{thm:salt_trichotomy} would have $N-\frac{1}{2}$ as an alternating surgery slope. Thus the $(N-\frac{1}{2})$-changemaker lattice with stable coefficients $\underline{\rho}(K)$ would admit an obtuse superbase. However, a computer search using Algorithm~\ref{alg:half_int_OSB} shows that this $(N-\frac{1}{2})$-changemaker lattice does not admit an obtuse superbase. 

Thus, Theorem~\ref{thm:salt_trichotomy} shows that $\Salt(K)$ is a subset of $\{N-1, N\}$ and calculation of $\Salt(K)$ is completed by showing that both of these are alternating surgery slopes. This is achieved by finding explicit alternating branching sets for each of these slopes and using SnapPy to verify the necessary homeomorphisms. These branching sets were found through a variety of methods. Firstly, we iterated through the tables of alternating knots of at most $15$ crossings and alternating links of at most $14$ crossings, and checked if their double branched covers yield one of the alternating fillings. This identified an explicit branching set for all but $8$ of the potential alternating surgeries. The remaining $8$ surgery slopes turn out to be the branched covers of alternating links with $15$ crossings (which are not tabulated by SnapPy). It turns out that $6$ of these missing alternating branching sets are alternating $15$ crossing links that were already constructed in~\cite{BL17}. For the remaining two, we used an obtuse superbase for the relevant changemaker lattice to construct Goeritz matrices for possible alternating branching sets. We converted these Goeritz matrices into alternating diagrams by-hand (these diagrams are shown in Figure~\ref{fig:br_sets_BL}) and then used SnapPy to quickly verify that the double branched covers are indeed isometric to the required surgered manifolds.
\end{proof}

\begin{table}[htbp]
 	\caption{The first $14$ asymmetric L-space knots from~\cite{BL17}, their alternating slopes with alternating branching sets and their stable coefficients.} 

\label{tab:BL_Knots_alt2}
 \ra{1.2}
\begin{tabular}{@{}l|l|l|l@{}}
\toprule
$(m,b_1,a_1,a_2,a_3)$ & N & alternating slopes & stable coefficients \\
\midrule
		
$(1, 1, 1, 1, 0)$ & $272$ &$(271,1) : K12a402$ & $[12, 9, 5, 4, 2]$ \\
                  &       &$(272,1) : L12a955$   & \\
\hline
$(1, 1, 0, 1, 1)$& $471$ &$(470,1) : L13a1826$ & $[16, 12, 7, 4, 2]$ \\
                  &       &$(471,1) : K13a4669$   &\\
 \hline                 
$(1, 1, 1, 2, 0)$& $416$ &$(415,1) : K13a1838$ & $[12, 12, 9, 5, 4, 2]$  \\
                  &       &$(416,1) : L13a3060$   &\\
\hline                  
$(1, 1, 2, 1, 0)]$& $557$ &$(556,1) : L14a13430$ & $[17, 14, 5, 5, 4, 2]$ \\
                  &       &$(557,1) : K14a14150$   &\\
 \hline                 
$(1, 2, 1, 1, 0)$& $555$ &$(554,1) : L14a6302$ & $[17, 13, 7, 6, 3]$  \\
                  &       &$(555,1) : K14a12040$   &\\
  \hline                
$(2, 1, 1, 1, 0)$& $588$ &$(587,1) : K14a5753$ & $[18, 13, 7, 6, 2, 2]$ \\
                  &       &$(588,1) : L14a12460$   &\\
   \hline               
$(1, 1, 1, 1, 1)$& $1156$ &$(1155,1) : K15a48589$ & $[26, 17, 12, 5, 4, 2]$\\
                  &       & $(1156,1) :  \text{see \cite{BL17}} $ &\\
   \hline               
$(1, 1, 0, 2, 1)$& $1010$ &$(1009,1) : K15a63354$ & $[23, 19, 7, 7, 4, 2]$  \\
                  &       & $(1010,1) : \text{see \cite{BL17}} $ &\\
  \hline                
$(1, 2, 0, 1, 1)$& $966$ &$(965,1) : K15a80635$ & $[23, 17, 10, 6, 3]$  \\
                  &       & $(966,1) : \text{see \cite{BL17}}$ &\\
   \hline               
$(1, 1, 2, 2, 0)$& $846$ &$(845,1) : K15a71795$ & $[17, 17, 14, 5, 5, 4, 2]$\\
                  &       &$(846,1) : \text{see \cite{BL17}}$   &\\
     \hline             
$(1, 2, 1, 2, 0)$& $844$ &$(843,1) : K15a27596$ & $[17, 17, 13, 7, 6, 3]$  \\
                  &       & $(844,1) : \text{see \cite{BL17}}$  &\\
      \hline            
$(2, 1, 1, 2, 0)$& $912$ &$(911,1) : K15a50514$ & $[18, 18, 13, 7, 6, 2, 2]$\\
                  &       &$(912,1) : \text{see \cite{BL17}}$   & \\
       \hline           
$(1, 1, 0, 1, 2)$& $1143$ &$(1142,1) : \text{see Figure~\ref{fig:br_sets_BL}} $&  $[28, 12, 12, 7, 4, 2]$ \\
                  &       &$(1143,1) : K15a43818$   &\\
       \hline           
$(2, 1, 0, 1, 1)$& $1067$ &$(1066,1) : \text{see Figure~\ref{fig:br_sets_BL}} $& $[24, 18, 11, 6, 2, 2]$  \\
                  &       &$(1067,1) : K15a84691$   &\\
                  
\bottomrule
 	\end{tabular}
\end{table}

  \begin{figure}
      \centering
      \includegraphics[width=.4\textwidth]{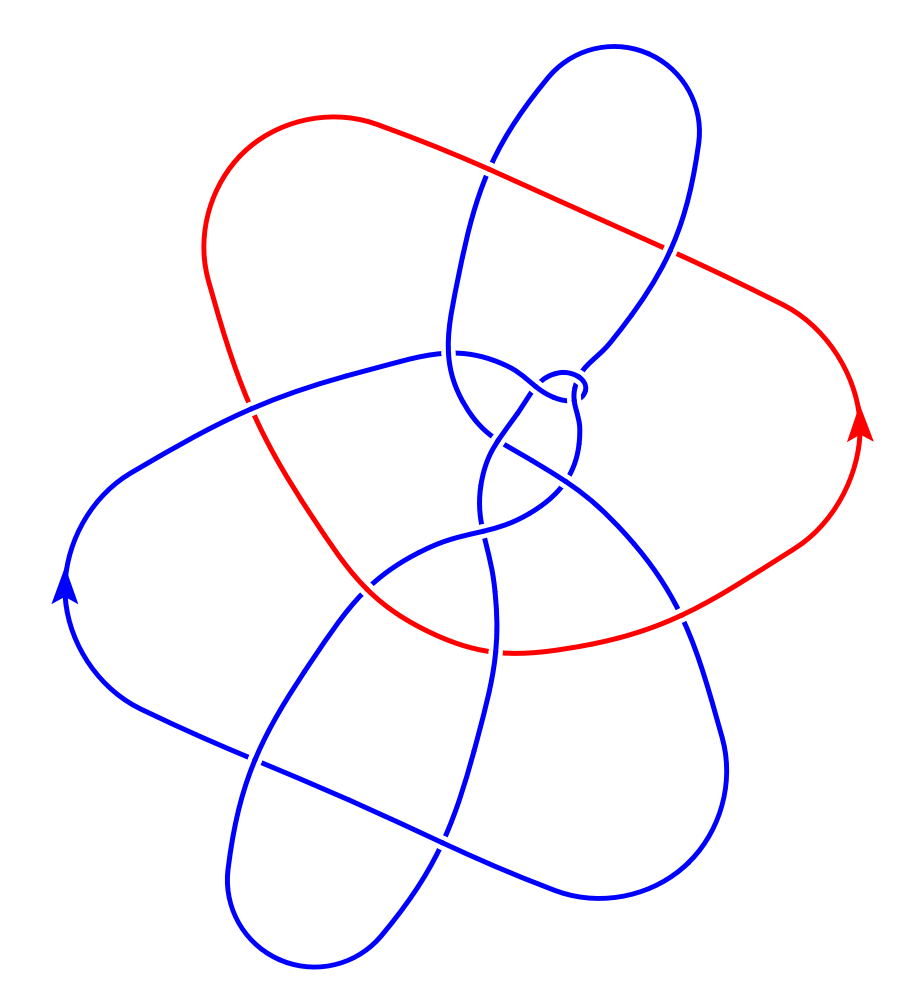} \qquad \includegraphics[width=.4\textwidth]{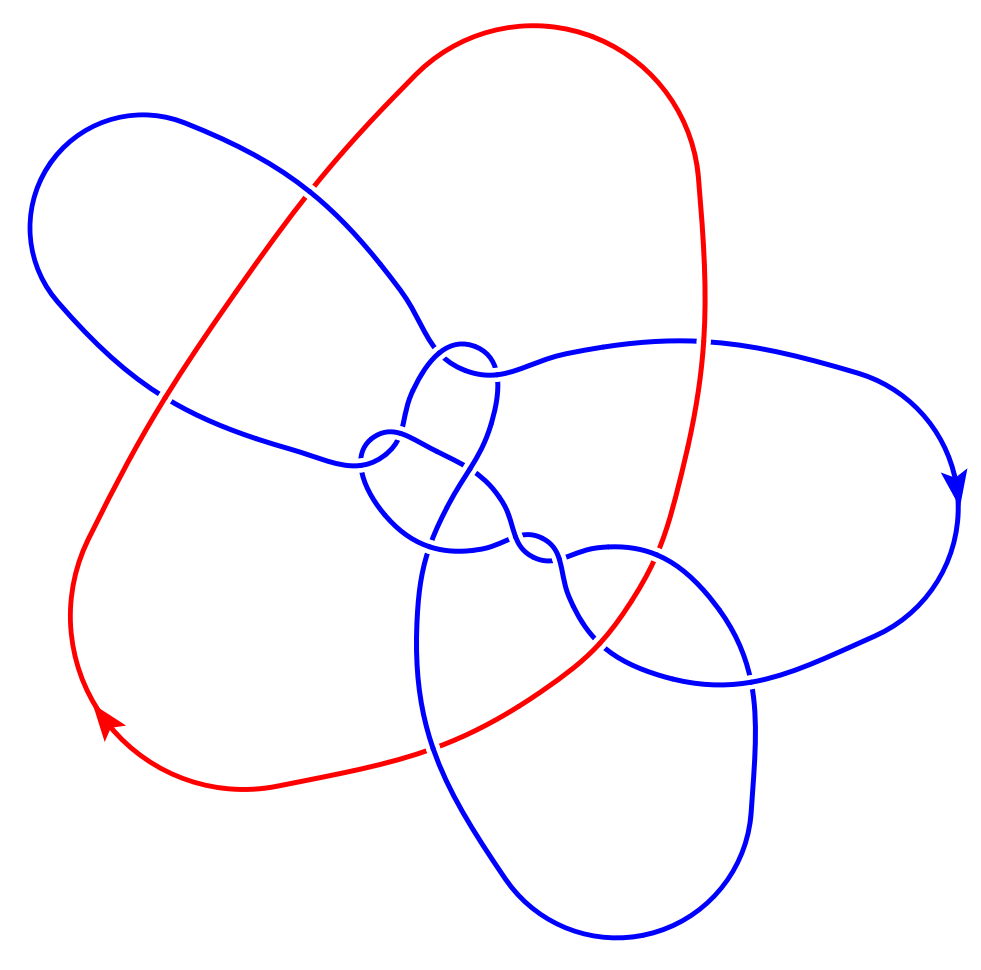} 
      \caption{Two $15$-crossing alternating diagrams of $2$-component links. The double branched cover of the left link yields $K_{(1,1,0,1,2)}(1142,1)$. The double branched cover of the right link yields $K_{(2,1,0,1,1)}(1066,1)$.  }
      \label{fig:br_sets_BL}
  \end{figure}

\newpage
\appendix

\section{Large surgeries on distinct knots.}
This section will contain the details necessary to prove Theorem~\ref{thm:computable_q}.

For a knot $K$ in $S^3$, let $X_K$ denote the exterior of $K$, that is $S^3$ with an open tubular neighbourhood of $K$ removed. In order to study when two knots can have common surgeries we make detailed use of the JSJ decompositions of their exteriors. Budney gives a nice description of the possible pieces that can arise in the JSJ decomposition of a knot exterior \cite{BudneyJSJ}. One of the components in the JSJ decomposition $X_K$ contains the boundary of $X_K$. We will refer to this distinguished piece as the {\em outermost piece} of the decomposition. This boundary component has a distinguished slope corresponding to the meridian of the knot.
 
\begin{thm}\label{thm:knot_JSJ_decomp}
Let $K$ be a non-trivial knot in $S^3$ and let $X$ be a component in the JSJ decomposition of $X_K$. Then $X$ takes one of the following forms.
\begin{enumerate}
    \item $X$ is a Seifert fiberd space over the disk with two exceptional fibers. If the JSJ decomposition of $X_K$ contains such an $X$, then it is necessarily the outermost piece and $K$ is a torus knot.
    \item $X$ is homeomorphic to $S^1\times F$ where $F$ is a compact planar surface with at least three boundary components. If such an $X$ is the outermost component, then $K$ is a composite knot.
    \item $X$ is a hyperbolic manifold. If such an $X$ is the outermost piece, then filling $X$ along the meridian yields $S^3$ or a connected sum of solid tori.
    \item $X$ is a Seifert fiberd space over the annulus with one exceptional fiber. If such an $X$ is the outermost piece, then $K$ is a cable of a non-trivial knot $K'$.
\end{enumerate}
\end{thm}
\begin{proof}
This is a subset of the information contained in \cite[Theorem~4.18]{BudneyJSJ}.
\end{proof}
Note that for all the Seifert fibered spaces appearing in Theorem~\ref{thm:knot_JSJ_decomp} all admit unique Seifert fibered structures, so we can unambiguously refer to the orders of exceptional fibers in these spaces as topological invariants. Furthermore, Mostow-Prasad rigidity implies that for a hyperbolic piece $M$, the systole $\sys(M)$ (i.e.\ the length of the shortest closed geodesic) is also a topological invariant. 

We describe a computable constant $Q(K,K')$ for Theorem~\ref{thm:computable_q} in terms of the JSJ decompositions of the exteriors $X_K$ and $X_{K'}$. We have not seriously attempted to optimize $Q(K,K')$.
\begin{defi}\label{def:Q}
For each knot $K$ in $S^3$ we define a computable constant $Q(K)$. For a distinct pair of knots $K$ and $K'$, we then set
\[
Q(K,K')=\lceil\max \{Q(K), Q(K')\}\rceil +1.
\]
For the unknot, we simply take $Q(U)=0$. For a non-trivial knot $K$, we define $Q(K)$ in terms of the JSJ decomposition of $X_K$.
Let $X$ be a component in the JSJ decomposition $X_K$ for $K$ a nontrivial knot. According to Theorem~\ref{thm:knot_JSJ_decomp} there are four possibilities for $X$ and we define $Q(X)$ as follows in each of these cases.
\begin{enumerate}
    \item If $X$ is the complement of the $T_{r,s}$ torus knot, i.e.\ Seifert fibered over the disk with exceptional fibers of order $|r|$ and $|s|$, then 
    \[Q(X)=\max\{8,|r|,|s|\}.\]
    \item If $X$ is a composing space, ie. homeomorphic to $S^1\times F$ where $F$ is a compact planar surface with at least three boundary components, then $Q(X)=1$.
    \item If $X$ is a hyperbolic manifold, then 
    \[Q(X)=\max\left\{ 32,\sqrt[4]{108}\sqrt{58+ \frac{2\pi}{\sys(X_i)}}\right\}.\]
    \item If $X$ is a Seifert fibered space over the annulus with one exceptional fiber of order $|s|$, then $Q(X)=|s|$.
\end{enumerate} 
Let $K$ be a non-trivial knot with JSJ decomposition
\[
X_K=X_0\cup \dots \cup X_n.
\]

We then define $Q(K)$ to be
\[
Q(K)=\begin{cases}
\max \{ Q(X_i)\,|\, i=0,\dots, n\} &\text{if $K$ is not a cable}\\
|r|+\max\{Q(X_i)\,|\, i=0,\dots, n\} &\text{if $K$ is an $(r,s)$-cable}
\end{cases}
\]
\end{defi}
\begin{rem}
By construction, we see that $Q(K)$ is chosen so that if $X_K$ contains a Seifert fibered piece in its JSJ decomposition, then the exceptional fibers cannot have order exceeding $Q(K)$.
\end{rem}

\begin{lem}
    $Q(K,K')$ is computable.
\end{lem}

\begin{proof}
This follows from several deep results in the algorithmic classification of $3$-manifolds. First, we note that there exists an algorithm that returns the JSJ pieces (and their gluing maps) of a given irreducible $3$-manifold (such as a knot exterior)~\cite[Theorem~6.4.42]{Matveev}. For a piece of a JSJ decomposition, it is decidable if it is hyperbolic or Seifert fibered~\cite{decide_hyperbolic}. The Seifert invariants of a given Seifert fibered $3$-manifold are computable as a corollary of the decision problem for $3$-manifolds~\cite{Kuperberg2019homeomorphism}. Thus it is also decidable if a given knot is a cable, and if it is, determine its cable parameters $(r,s)$. Finally, one can compute the systole of a given hyperbolic manifold up to a given arbitrarily small error term~\cite{length_spectrum_computable}. 
\end{proof}

    It will also be convenient to note that Theorem~\ref{thm:computable_q} follows from known results on characterizing slopes in several cases.
\begin{prop}\label{prop:computable_q_cases}
    Let $K$ and $K'$ be distinct knots in $S^3$ such that $K$ is either (i) a torus knot or (ii) a cable of a torus knot. Then if $S_{p/q}^3(K)\cong S_{p/q}^3(K')$ for some $p/q\in \Q$, then $|q|\leq Q(K,K')$.
\end{prop}
\begin{proof}
If $K$ is the $T_{r,s}$ torus knot, then $Q(K,K')\geq Q(K)=\max\{8,|r|,|s|\}$ and, \cite[Theorem~1.1]{McCoy2020torus_char} gives the desired conclusion. If $K$ is the $(a,b)$-cable of $T_{r,s}$, then $Q(K,K')\geq Q(K)=|a|+\max\{8,|b|,|r|,|s|\}$ and \cite[Theorem~7.3(iii)]{Sorya2023satellite} gives the desired conclusion.
\end{proof}

Next we describe how the JSJ decomposition of a knot complement changes under Dehn surgery. Note that if $q>1$, then a surgery $S_{p/q}^3(K)$ is irreducible \cite{Gordon1987reducible} and so it makes sense to discuss the JSJ decomposition of such a surgery. Given an outermost piece $Y$ in a knot exterior, the boundary component coming from the knot exterior comes equipped with a basis of homology inherited from the the knot exterior, so we may use the same coordinates and can thus refer to $Y(p/q)$ to denote the $p/q$-filling on this boundary component.
\begin{prop}[Proposition~3.6 of \cite{Sorya2023satellite}]\label{prop:surgery_JSJ_decomp}
    Let $K$ be a non-trivial knot in $S^3$ and let $p/q$ be a rational number with $|q|\geq 3$. Let
    \[X_K=X_0 \cup \dots \cup X_n\]
    be the JSJ decomposition of $X_K$ with $X_0$ the outermost piece. If $K$ is not a cable, then the JSJ decomposition of $X_K(p/q)$ takes the form
    \[
    X_K(p/q)=X_0(p/q)\cup X_1 \cup \dots \cup X_n.
    \]
    If $K$ is the $(r,s)$-cable of a non-trivial knot $J$, then the JSJ decomposition of $X_K(p/q)$ takes the form
    \[
    X_K(p/q)\cong \begin{cases}
        X_0(p/q)\cup X_1 \cup \dots \cup X_n &\text{if $|p-rsq|>1$}\\
        X_1(p/(qs^2))\cup X_2 \cup \dots \cup X_n &\text{if $|p-rsq|=1$}
    \end{cases}
    \]
    where $X_1$ is the outermost piece of $X_J$.\qed
\end{prop}
\begin{rem}\label{rem:only_collapsing_once}
Proposition~\ref{prop:surgery_JSJ_decomp} implicitly uses the fact that if $K$ is a twice iterated cable of a non-trivial knot $J$, say $K=C_{r_1,s_1}(C_{r_2,s_2}(J))$ and $|p-r_1s_1q|=1$, then 
\[
|p-qr_2s_2s_1^2|>1.
\]
\end{rem}

\subsection{The hyperbolic case}
We now summarize the quantitative results of Futer-Purcell-Schleimer on hyperbolic Dehn filling that we need in order to prove Theorem~\ref{thm:computable_q}. Given a cusped hyperbolic manifold $N$ with a choice of horocusp neighbourhood $C$ for one of its cusps and a slope $\sigma$ on $\partial C$, the normalized length of $\sigma$ is defined to be
\[
L(\sigma)=\frac{\len(\sigma)}{\sqrt{\operatorname{Area}\partial C}},
\]
where $\len(\sigma)$ is the length measured in the Euclidean metric on $\partial C$.
Note that $L(\sigma)$ is invariant under rescaling $C$ and is hence independent of the choice of horocusp.
\begin{lem}\label{lem:normalized_length_bound}
Let $X_0$ be a hyperbolic outermost piece occurring in the JSJ decomposition of a non-trivial knot $K$. The normalized length of the slope $p/q$ on $\partial X_K\subseteq X_0$ satisfies
\[
L(p/q) \geq\frac{|q|}{\sqrt[4]{108}}.
\]
\end{lem}
\begin{proof}
By \cite[Theorem~1.2]{Gabai_low_vol}, we may consider $C$ a horocusp neighbourhood of $\partial X_K$ such that $\operatorname{Area}\partial C=\sqrt{12}$. Let $\mu$ denote the meridian of $K$ on $\partial X_K$. Since the meridian is an exceptional slope in $X_0$, the 6-theorem implies that $\len(\mu)\leq 6$ \cite{Ag00, La00}. A geometric argument in the universal cover shows that for any slope $p/q$ on $\partial X_K$ we have
\[
\sqrt{12} \Delta(\mu,p/q) \leq \len(\mu) \len(p/q).
\]
Since $\Delta(p/q,\mu)=|q|$, all of these ingredients can be combined to bound the normalized length of $p/q$ as
\[
L(p/q) = \frac{\len(p/q)}{\sqrt[4]{12}}\geq \frac{\sqrt[4]{12} |q|}{\len(\mu)}\geq \frac{ \sqrt[4]{12}|q|}{6}=\frac{|q|}{\sqrt[4]{108}},
\]
as required.
\end{proof}

The following is a quantitative version of Thurston's Dehn filling theorem due to Futer-Purcell-Schleimer \cite{FPS19}.
\begin{thm}\label{thm:quantative_filling}
Let $N$ be a cusped hyperbolic manifold and let $\sigma$ be a slope on a cusp of $N$ such that the normalized length satisfies
\[
L(\sigma)\geq \max\left\{10.1, \sqrt{58 + \frac{2\pi}{\sys N}}  \right\}.
\]

If $M$ is the manifold obtained by filling $N$ along $\sigma$, then $M$ is hyperbolic and the core of the Dehn filling solid torus is isotopic to the shortest closed geodesic in $M$. Moreover, this satisfies
\[\sys(M)\leq \frac{2\pi}{L^2-16.03}.\]
\end{thm}\begin{proof}
    This follows from \cite[Theorem~7.28]{FPS19} applied to filling along a single cusp combined with the fact that the function $\sysmin$ in that paper satisfies
    \[
\frac{2\pi}{L^2}<\sysmin(L)<\frac{2\pi}{L^2-58}
\]
for $L\geq 10.1$.
\end{proof}

The next lemma proves Theorem~\ref{thm:computable_q} in the case where one of the outermost pieces is hyperbolic. We prove a slightly stronger statement, as it will afford us more flexibility later.

\begin{lem}\label{lem:hyp_case}
Let $K, K'$ be distinct knots, where the outermost piece of $X_K$ is hyperbolic. If there are slopes $p/q$ and $p/q'$ such that
\[
\text{$S_{p/q}^3(K)\cong S_{p/q'}^3(K')$ and $q,q'> Q(K, K')$,}
\]
then $K'$ is an $(a,b)$-cable of $K$ for some coprime integers $a,b$ such that $b\geq 2$ and $p/q = \frac{ab\pm 1}{q'b^2}$.
\end{lem}
\begin{proof}
Let $X_K=X_0 \cup \dots \cup X_n$ be the JSJ decomposition of $X_K$ with outermost piece $X_0$ hyperbolic. Let $X_{K'}=X'_0 \cup \dots \cup X'_{n'}$ be the JSJ decomposition for $X(K')$ with outermost piece $X'_0$. Since $|q|\geq 3$, Proposition~\ref{prop:surgery_JSJ_decomp} shows that the JSJ decomposition of $S_{p/q}^3(K)$ is given by
\[
S_{p/q}^3(K)=X_0(p/q) \cup X_1\cup \dots \cup X_n.
\]
 Let $L$ denote the normalized length of the slope $p/q$ on the cusp of $X_0$. By definition of $Q(K,K')$, we have that
 \[
 \text{$|q|> 10.1\sqrt[4]{108}\approx 32.5$ and $|q|> \sqrt[4]{108}\sqrt{58+ \frac{2\pi}{\sys(X_0)}}$}.
 \]
Consequently, Lemma~\ref{lem:normalized_length_bound} implies that
 \begin{equation}\label{eq:hyp_systole_1}
 L>\max \left\{10.1, \sqrt{58+ \frac{2\pi}{\sys(X_0)}} \right\}. 
 \end{equation}
Likewise, for any hyperbolic piece $X_i'$ in $X_{K'}$, Lemma~\ref{lem:normalized_length_bound} and the definition of $Q(K,K')$ implies the inequality
\begin{equation}\label{eq:hyp_systole_2}
 L> \sqrt{58+ \frac{2\pi}{\sys(X'_i)}}. 
 \end{equation}
 
By the quantitative filling bounds of Theorem~\ref{thm:quantative_filling}, \eqref{eq:hyp_systole_1} implies that $X_0(p/q)$ is hyperbolic and that in the resulting hyperbolic structure we can assume that the core of the filling torus is the systole in $X_0(p/q)$. Furthermore, the bound in Theorem~\ref{thm:quantative_filling} combined with \eqref{eq:hyp_systole_2} show that the systole of $X_0(p/q)$ satisfies
\begin{align*}
\sys(X_0(p/q))&\leq \frac{2\pi}{L^2 -17} \\ 
&\leq \frac{2\pi \sys(X_i')}{41\sys(X_i') + 2\pi}\\
&= \sys(X_i') - \frac{41\sys(X_i')^2}{41\sys(X_i')+2\pi}\\
&<\sys(X_i'),
\end{align*}

for any hyperbolic piece $X'_i$ in the JSJ decomposition of $X_{K'}$.
Thus we see that $X_0(p/q)$ cannot be homeomorphic to $X_i'$ for any $X_i'$ in the JSJ decomposition of $X_{K'}$.

Let $f\colon S_{p/q}^3(K) \rightarrow S_{p/q'}^3(K')$ be an orientation-preserving homeomorphism, which, after a suitable isotopy, can be assumed to carry JSJ pieces of $S_{p/q}^3(K)$ to JSJ pieces of $S_{p/q'}^3(K')$. In particular, $f$ restricts to give a homeomorphism between $X_0(p/q)$ and one of the JSJ pieces of $S_{p/q'}^3(K')$. First suppose that the pair $(K',p/q')$ is not of the form $(C_{a,b}(J), ab\pm \frac{1}{q'})$. In this case Proposition~\ref{prop:surgery_JSJ_decomp} implies that the JSJ decomposition of $S_{p/q'}^3(K')$ takes the form
\[S_{p/q'}^3(K')=X'_0(p/q') \cup X'_1 \cup \dots \cup X'_{n'}.\]
Since $X_0(p/q)$ is not homeomorphic to any $X'_i$, we see that $f$ carries $X_0(p/q)$ homeomorphically onto $X_0'(p/q')$. This implies that $X_0'$ must be hyperbolic. Moreover, as above, the bounds on $q'$ imply via Lemma~\ref{lem:normalized_length_bound} and Theorem~\ref{thm:quantative_filling} that the core of the Dehn filling torus in $X_0'(p/q')$ is isotopic to the systole in $X_0'(p/q')$. Thus up to isotopy we can assume that restricting $f$ to the complement of the filling torus yields a homeomorphism between the knot complements $X_K$ and $X_{K'}$. By the knot complement theorem \cite{GL89}, this implies that $K$ and $K'$ must be isotopic, and since the meridian is unique we have that $q=q'$.

If the pair $(K',p/q')$ is of the form $(C_{a,b}(J), ab\pm \frac{1}{q'})$, then
\[
S_{p/q}^3(K)\cong S_{p/q'}^3(K')\cong S_{p/(q'b^2)}^3(J).
\]
Since $Q(K, K')\geq Q(K,J)$, we may apply the above analysis to the JSJ decomposition of $J$ in order to conclude that $J$ and $K$ are isotopic and that $q=q'b^2$.
\end{proof}
\subsection{The composite case}
Next we turn our attention to the case where one of the knots is composite. The argument is very similar to the hyperbolic case given in Lemma~\ref{lem:hyp_case}.
\begin{lem}\label{lem:composite_case}
Let $K, K'$ be distinct knots, where $K$ is a composite knot. If there are slopes $p/q$ and $p/q'$ such that
\[
\text{$S_{p/q}^3(K)\cong S_{p/q'}^3(K')$ and $q,q'> Q(K,K')$,}
\]
then $K'$ is an $(a,b)$-cable of $K$ for some coprime integers $a,b$ such that $b\geq 2$ and $p/q = \frac{ab\pm 1}{q'b^2}$.
\end{lem}
\begin{proof}
Let $X_K=X_0 \cup \dots \cup X_n$ be the JSJ decomposition for $X_K$ with outermost piece $X_0$. By hypothesis, $X_0$ is a composing space. Since $|q|\geq 3$, Proposition~\ref{prop:surgery_JSJ_decomp} shows that the JSJ decomposition of $S_{p/q}^3(K)$ is
\[
S_{p/q}^3(K)=X_0(p/q) \cup X_1\cup \dots \cup X_n.
\]
The JSJ piece $X_0(p/q)$ is a Seifert fibered space with a unique Sefiert fibered structure and the core of the filling torus is an exceptional fiber of order $|q|$ in this structure.

Let $X_{K'}=X'_0 \cup \dots \cup X'_{n'}$ be the JSJ decomposition for $X_{K'}$ with outermost piece $X'_0$. Since $q> Q(K,K')$, we see that none of the Seifert fibered JSJ pieces of $X_{K'}$ can contain an exceptional fiber of order $|q|$ -- by definition, every exceptional fiber in the Seifert fibered JSJ pieces of $X_{K'}$ have order at most $Q(K,K')$. Thus we see that $X_0(p/q)$ is not homeomorphic to $X'_i$ for any $i$.

Let $f\colon S_{p/q}^3(K) \rightarrow S_{p/q'}^3(K')$ be an orientation-preserving homeomorphism, which, after a suitable isotopy, can be assumed to carry JSJ pieces of $S_{p/q}^3(K)$ to JSJ pieces of $S_{p/q'}^3(K')$. First suppose that the pair $(K',p/q')$ is not of the form $(C_{a,b}(J), ab\pm \frac{1}{q'})$. In this case, Proposition~\ref{prop:surgery_JSJ_decomp} implies that the JSJ decomposition of $S_{p/q'}^3(K')$ takes the form
\[
S_{p/q'}^3(K')=X'_0(p/q') \cup X'_1 \cup \dots \cup X'_{n'}.
\]
By the proceeding discussion $f$ carries $X_0(p/q)$ homeomorphically onto $X_0'(p/q')$. Thus $X'_0(p/q')$ is also Seifert fibered space over a planar surface with at least two boundary components and a single exceptional fiber of order $|q|$. This implies that the outermost piece $X_0'$ is also a composing space and that the core of filling torus is the exceptional fiber. Since the exceptional fiber in such a space is unique up to isotopy, we can assume that $f$ maps the exceptional fiber of $X_0(p/q)$ to the exceptional fiber of $X_0'(p/q)$. Thus restricting to the complement of these fibers we obtain a homeomorphism $X_K$ to $X_{K'}$. By the knot complement theorem \cite{GL89}, this implies that $K$ and $K'$ are isotopic and that $p/q=p/q'$. 

If the pair $(K',p/q')$ is of the form $(C_{a,b}(J), ab\pm \frac{1}{q'})$, then
\[
S_{p/q}^3(K)\cong S_{p/q'}^3(K')\cong S_{p/(q'b^2)}^3(J).
\]
Since $Q(K')\geq Q(J)$, we may apply the above analysis to conclude that $J$ and $K$ are isotopic and that $q=q'b^2$.
\end{proof}

\subsection{Surgery on cables}
The most awkward cases of Theorem~\ref{thm:computable_q} arise when both $K$ and $K'$ are cable knots. Unlike the cases previously discussed, we cannot simply consider the topology of the filling on the outermost piece -- we must also use the information on how the JSJ pieces of the surgered manifold are glued together. The problem is that a cable space can be Dehn filled to obtain the complement of a fixed torus knot in infinitely many different ways.
\begin{defi}
    Let $M$ be a closed oriented rational homology 3-sphere and let $X$ be a component in the JSJ decomposition of $M$ with the structure of a Seifert fibered space over the disk with two exceptional fibers. The {\em fiber-longitude number} of $X$ will be the distance between the Seifert fibered slope on $\partial X$ and the rational longitude on $\partial (M\setminus \mathrm{int} X)$.
\end{defi}
Up to isotopy a Seifert fibered space admits at most one Seifert fibered structure over the disk with two exceptional fibers\footnote{The classification of Seifert fibrations shows that amongst all manifolds admitting such Seifert fibered structures, only the twisted $I$-bundle over the Klein bottle admits more than one Seifert structure and only one of these is over the disk, with the other being over the M\"obius band.}.
 
 We calculate the fiber-longitude number in two cases.
\begin{lem}\label{lem:fiber_long_cable}
Consider a non-trivial cable $C_{a,b}(J)$ and let $p/q$ be a slope such that $|p-abq|>1$. Then the fiber-longitude number of the component in $S_{p/q}^3(C_{a,b}(K))$ obtained by filling in the cable space is $|a|$.
\end{lem}
\begin{proof}
By definition we obtain the complement of $C_{a,b}(K)$ by gluing a cable space to $X_K$ so that the Seifert fiber is glued to the slope $a/b$ on $\partial X_K$. When we perform $p/q$-filling on $C_{a,b}(K)$ for $|p-abq|>1$ the the Seifert fibered structure on the cable space extends to give a Seifert fibered structure on the filling. Thus after Dehn filling the Seifert fibered slope on $\partial X_K$ is still $a/b$. This has slope distance $|a|$ from the longitude $0/1$ on $\partial X_K$.
\end{proof}

\begin{lem}\label{lem:fiber_long_torus_piece}
Let $K$ be a satellite knot containing a torus knot complement $X=X_{T_{r,s}}$ in its JSJ decomposition. If $\partial X$ remains incompressible in $S^3_{p/q}(K)$ and $p/q\neq 0/1$, then the fiber-longitude number of $X$ in $S^3_{p/q}(K)$ is $\frac{1}{g}|p-rsqw^2|$ where $w$ is the winding number of the pattern when $K$ is viewed as a satellite with companion $T_{r,s}$ and $g=\gcd(p,w^2)\geq 1$.
\end{lem}
\begin{proof}
Let $\mu$ be the meridian on $\partial X$ and $\lambda$ the null-homologous longitude. In these coordinates, the Seifert fibered slope on $\partial X$ is $rs \mu + \lambda$. On the other hand, a straightforward homological calculation \cite{Gordon1983Satellite} shows that after Dehn filling, the class $p\mu - qw^2 \lambda$ is null-homologous in $X_K(p/q)\setminus \operatorname{int} X$. Thus $p/g\mu - qw^2/g \lambda$ is the rational longitude. This gives a fiber-longitude number of $\frac{1}{g}|p-rsqw^2|$.
\end{proof}

Now we apply these ideas.
\begin{lem}\label{lem:cabled_case}
Let $K$ and $K'$ be two cabled knots in $S^3$, say $K$ is an $(r,s)$-cable and $K'$ is an $(r',s')$-cable. Let $p/q,p/q'\in \Q$ be surgery slopes such that $S_{p/q}^3(K)\cong S_{p/q'}^3(K')$ and $|p-rsq|>1$, $|p-r's'q'|>1$. If $q,q' > Q(K,K')$, then $K$ and $K'$ are isotopic and $q=q'$.
\end{lem}

\begin{proof}
By Proposition~\ref{prop:surgery_JSJ_decomp}, the conditions $|p-rsq|>1$ and $|p-r's'q'|>1$ imply that the JSJ decompositions of $S_{p/q}^3(K)$ and $S_{p/q'}^3(K')$ are of the form
\[
S_{p/q}^3(K)=X_0(p/q) \cup X_1 \cup \dots \cup X_n
\]
and
 \[
S_{p/q'}^3(K')=X'_0(p/q') \cup X'_1 \cup \dots \cup X'_{n},
\]
respectively (note that they necessarily have the same number of pieces since $S_{p/q}^3(K)$ and $S_{p/q'}^3(K')$ are homeomorphic). The piece $X_0(p/q)$ is a Seifert fibered space over the disk with two exceptional fibers of orders $|s|$ and $|p-rsq|$ respectively and $X'_0(p/q')$ is a Seifert fibered space over the disk with two exceptional fibers of orders $|s'|$ and $|p-r's'q'|$.
Let $f\colon S_{p/q}^3(K) \rightarrow S_{p/q'}^3(K')$ be an orientation-preserving homeomorphism, which, after a suitable isotopy, can be assumed to carry JSJ pieces of $S_{p/q}^3(K)$ to JSJ pieces of $S_{p/q'}^3(K')$. The first task to is to show that $f$ carries $X_0(p/q)$ homeomorphically onto $X'_0(p/q)$.
\begin{claim}
    The homeomorphism $f$ carries $X_0(p/q)$ onto $X'_0(p/q)$.
\end{claim}
\begin{proof}[Proof of claim]
    If $f$ does not carry $X_0(p/q)$ homeomorphically onto $X'_0(p/q)$, then it must provide a homeomorphism between $X_0(p/q)$ and some other piece in the JSJ decomposition of $X_{K'}$, which we assume to be $X'_n$. Since $X_0(p/q)$ is a Seifert fibered space over this disk, this implies that $X'_n$ must be the complement of a torus knot $T_{a',b'}$. By comparing the orders of exceptional fibers, we may assume that this knot is of the form $a'=s\geq 2$ and $|b'|=|p-rsq|$. Thus we may write $K'$ as a satellite knot of the form $K'=P'(T_{s, b'})$ with a pattern $P'$ such that its complement in $S^1\times D^2$ is homeomorphic to $X'_0 \cup \dots \cup X'_{n-1}$. Let $w'$ denote the winding number of $K'$ in $P'$. 
    
    Since $K'$ is an $(r',s')$-cable knot, the winding number $w'$ of $K'$ in $P'$ is divisible by $s'$. Next we verify that $\gcd(p,w')=1$. Observe that $f$ restricts to give a homeomorphism between $X'_0(p/q') \cup X'_1 \cup \dots \cup X'_{n-1}$ and $X_1 \cup \dots \cup X_n$. Since the latter space is a knot complement, it is a homology solid torus. However \cite[Lemma~3.3(i)]{Gordon1983Satellite} shows that the homology of $X'_0(p/q') \cup X'_1 \cup \dots \cup X'_{n-1}$ contains torsion of order $\gcd(p,w)$. Hence the conclusion $\gcd(p,w')=1$.

    By Lemma~\ref{lem:fiber_long_cable}, the fiber-longitude number of $X_0(p/q)$ in $S_{p/q}^3(K)$ is $|r|$. On the other hand, Lemma~\ref{lem:fiber_long_torus_piece} calculates the fiber-longitude number of the complement of $T_{s, b'}$ to be $|p-q'sb'(w')^2|$ (using that $\gcd(p,w')=1$). Thus we have that
    \[
    |r|=|p-q'sb'(w')^2|.
    \]

    Likewise, $X'_0(p/q')$ must be homeomorphic to some piece in the JSJ decomposition of $X_K$. This allows us to write $K$ as a satellite in the form $K=P(T_{s',b})$, where $|b|=|p-r's'q'|$. 
    Since $r'$ and $s'$ are coprime, the product $r's'$ is not divisible by $(s')^2$. However, since $w'$ is divisible by $s'$, we see that $(s')^2$ divides $sb'(w')^2$. Thus the integers $sb'(w')^2$ and $r's'$ are distinct. Consequently we see that the conditions
    \[
    \text{$|r|=|p-q'sb'(w')^2|$ and $|b|=|p-r's'q'|$}
    \]
    can only be satisfied if $q'\leq r+b\leq Q(K)$.
\end{proof}
Thus $f$ restricts to a homeomorphism between $X_0(p/q)$ and $X'_0(p/q')$ and between $X_1\cup \dots \cup X_n$ and $X'_1 \cup \dots \cup X'_{n}$. Thus if we write $K$ as the $(r,s)$-cable of $J$ and $K'$ as the $(r',s')$-cable of $J'$, we see that $f$ restricts to a homeomorphism between the complements of $J$ and $J'$. Thus by the knot complement theorem, we have that $J$ and $J'$ are isotopic \cite{GL89}. Furthermore, since the homeomorphism must carry the meridian and longitude of $J$ to the meridian and longitude of $J'$, respectively, we see that $K$ and $K'$ must both be the $(r,s)$-cable of $J$. 

Finally we verify that $q'=q$. The homeomorphism between $X_0(p/q)$ and $X'_0(p/q')$ must preserve the rational longitudes of these manifolds. However, this rational longitude, when expressed as a slope on $\partial X_J$ takes the form $p/(qs^2)$ and $p/(q's^2)$. This implies that $q'=q$ as required.  
\end{proof}

\subsection{Proof of Theorem~\ref{thm:computable_q}}
\begin{thm}\label{thm:computable_q}
\computableq
\end{thm}
\begin{proof}
\setcounter{case}{0}
We will show that if $K$ and $K'$ are two knots and $p/q$ is a slope such that $S_{p/q}^3(K)\cong S_{p/q}^3(K')$ for some $q>Q(K,K')$, then $K$ and $K'$ must be isotopic. We can assume that both $K$ and $K'$ are non-trivial, since every slope is a characterizing slope for the unknot \cite{Kronheimer2007lensspacsurgeries}. Furthermore, by Proposition~\ref{prop:computable_q_cases} we can assume that neither $K$ nor $K'$ is a torus knot. In accordance with Theorem~\ref{thm:knot_JSJ_decomp}, we can consider each of the remaining possibilities for the outermost pieces in the JSJ decompositions of $X_K$ and $X_{K'}$ in turn.

\begin{case}
At least one of $K$ or $K'$ is a composite or has a hyperbolic outermost piece in its JSJ decomposition.
\end{case}
\begin{proof}
Without loss of generality suppose that $K$ is composite or has hyperbolic outermost piece in its JSJ decomposition. Lemma~\ref{lem:hyp_case} and Lemma~\ref{lem:composite_case} imply that $K'$ must be a cable of $K$ with winding number $s$, where $p/q=p/(qs^2)$. This is clearly impossible since for $K'$ to be a non-trivial cable of $K$ we must have $|s|\geq 2$.
\end{proof}
Thus it remains to consider the case that both $K$ and $K'$ are cables. For the remainder of the proof, we assume that $K$ is the $(r,s)$-cable of $J$ and $K'$ is the $(r',s')$-cable of $J'$, where $s,s'>1$ and $J$ and $J'$ are non-trivial knots in $S^3$. Furthermore, Lemma~\ref{lem:cabled_case} implies the theorem if $|p-rsq|>1$ and $|p-r's'q|>1$. Thus, it remains to consider the case that we have $|p-rsq|=1$ or $|p-r's'q|=1$. So without loss of generality we will further assume that $|p-rsq|=1$.

Under these assumptions, we have
\begin{equation}\label{eq:cable_homeo_case}
    S^3_{p/q}(K)\cong S^3_{p/(qs^2)}(J)\cong S^3_{p/q}(K').
\end{equation}
Using Theorem~\ref{thm:knot_JSJ_decomp} we perform a case-by-case analysis based on the outermost piece of $X_J$. Note that that if $J$ is a torus knot then, $K$ is a cable of a torus knot and the theorem follows from Proposition~\ref{prop:computable_q_cases}.
\begin{case}
$J$ is a composite knot or has a hyperbolic outermost piece.
\end{case}
\begin{proof}
    If $J$ has hyperbolic outermost piece or is a composite knot, then Lemma~\ref{lem:hyp_case} and Lemma~\ref{lem:composite_case} applied to the homeomorphism in \eqref{eq:cable_homeo_case} imply that $J$ and $J'$ are isotopic,  $|p-r's'q|=1$ and that $qs^2=qs'^2$. Since $q\geq 3$, the conditions $|p-rsq|=1$ and  $|p-r's'q|=1$ imply that $r's'=rs$. Since $s^2=s'^2$, this implies that $r/s=r'/s'$ and hence we $K$ and $K'$ are both isotopic to the $(r,s)$-cable of $J$.
\end{proof}

Thus the only remaining possibility is that $J$ is itself a cable knot.

\begin{case}
$J$ is a cable.
\end{case}
\begin{proof}
In this case, suppose that $J$ is the $(a,b)$-cable of a knot $L$. By Remark~\ref{rem:only_collapsing_once}, we have $|p-abs^2q|>1$. If $|p-r's'q|>1$, then Lemma~\ref{lem:cabled_case} would imply that $s^2q=q$, which is absurd. Thus we have that $|p-r's'q|=1$ and hence that $S_{p/(qs^2)}^3(J)\cong S_{p/(qs'^2)}^3(J')$. Now we may assume that $J'$ is itself a cable. If it were not, then we could reverse the roles of $K$ and $K'$ and apply the analysis of the preceding case to conclude that $K$ and $K'$ and are isotopic. Thus we may assume that $J'$ is the $(a',b')$-cable of a non-trivial knot. Furthermore, we have that $|p-a'b's'^2q|>1$. In this case, we may apply Lemma~\ref{lem:cabled_case} to see that $J$ and $J'$ are isotopic and that $qs^2=qs'^2$. Furthermore, if $|q|\geq 3$, then we have $rs=r's'$. Altogether, this implies that $K$ and $K'$ are isotopic, both being the $(r,s)$-cable of $J$. 
\end{proof}
This completes the final case and the proof of Theorem~\ref{thm:computable_q}.
\end{proof}

\subsection{Surgeries on torus knots}
Here we present a concrete example of non-isotopic torus knots that share infinitely many surgeries with \textit{different} slopes. This shows that Lemma~\ref{lem:hyp_case} and Lemma~\ref{lem:composite_case} cannot hold for torus knots.

\begin{ex}
We consider the non-isotopic torus knots $T_{3,5}$ and $T_{3,8}$. Then it follows from Moser's classification of surgeries on torus knots~\cite{Moser1971elementary} that for any integer $n$ both $S_{(120n+53)/(8n+3)}^3(T_{3,5})$ and $S_{(120n+53)/(5n+2)}^3(T_{3,8})$ are homeomorphic Seifert fibered spaces.
\end{ex}

\let\MRhref\undefined
\bibliographystyle{hamsalpha}
\bibliography{altsurg.bib}

\end{document}

%% file: incidencenumber.pdf_tex
\begingroup%
  \makeatletter%
  \providecommand\color[2][]{%
    \errmessage{(Inkscape) Color is used for the text in Inkscape, but the package 'color.sty' is not loaded}%
    \renewcommand\color[2][]{}%
  }%
  \providecommand\transparent[1]{%
    \errmessage{(Inkscape) Transparency is used (non-zero) for the text in Inkscape, but the package 'transparent.sty' is not loaded}%
    \renewcommand\transparent[1]{}%
  }%
  \providecommand\rotatebox[2]{#2}%
  \ifx\svgwidth\undefined%
    \setlength{\unitlength}{151bp}%
    \ifx\svgscale\undefined%
      \relax%
    \else%
      \setlength{\unitlength}{\unitlength * \real{\svgscale}}%
    \fi%
  \else%
    \setlength{\unitlength}{\svgwidth}%
  \fi%
  \global\let\svgwidth\undefined%
  \global\let\svgscale\undefined%
  \makeatother%
  \begin{picture}(1,0.43668181)%
    \put(0,0){\includegraphics[width=\unitlength]{incidencenumber.pdf}}%
    \put(0.06246447,0.0054584){\color[rgb]{0,0,0}\makebox(0,0)[lb]{\smash{$\mu = +1$}}}%
    \put(0.6982261,0.00924258){\color[rgb]{0,0,0}\makebox(0,0)[lb]{\smash{$\mu=-1$}}}%
  \end{picture}%
\endgroup%